\documentclass[12pt,reqno]{amsart}
\usepackage{amsmath,amssymb,amsthm,mathtools,wasysym,calc,verbatim,enumitem,tikz,pgfplots,url,mathrsfs,fullpage,bbm}

\usepackage{comment}

\usepackage{bbold}

\newtheorem{mytheo}{Theorem}[section]
\newtheorem{mydef}[mytheo]{Definition}
\newtheorem{myrem}[mytheo]{Remark}
\newtheorem{mylem}[mytheo]{Lemma}
\newtheorem{mycoro}[mytheo]{Corollary}
\newtheorem{mypropo}[mytheo]{Proposition}

\newtheorem{mycla}[mytheo]{Claim}

\theoremstyle{definition}

\theoremstyle{remark}

\usepackage[colorlinks = true,
            linkcolor  = blue,
            citecolor  = blue,
            urlcolor   = blue,
            bookmarks  = true]{hyperref}

\theoremstyle{remark}

\numberwithin{equation}{section}

\usepackage[backend=biber, style=numeric, citestyle=numeric-comp]{biblatex}
\DeclareFieldFormat{labelnumberwidth}{\mkbibbrackets{#1}}
\DeclareNameAlias{author}{family-given}

\DeclareFieldFormat*{title}{#1}

\AtEveryBibitem{%
  \clearfield{issn}
  \clearfield{url}
  \clearfield{eprint}
}

\renewbibmacro{in:}{}

\DeclareFieldFormat[article]{volume}{\textbf{#1}}
\DeclareFieldFormat[article]{number}{\textbf{#1}}
\renewbibmacro*{volume+number+eid}{%
  \printfield{volume}
  \iffieldundef{number}
    {}%
    {\printtext[parens]{\printfield{number}}}
  \setunit{\addcomma\space}%
  \printfield{eid}}

\DeclareFieldFormat[journal]{journaltitle}{\mkbibemph{#1}}

\DeclareNameAlias{author}{family-given}
\DeclareNameAlias{editor}{family-given}

\DeclareDelimFormat{namedelim}{\addcomma\space}
\DeclareFieldFormat[article,incollection,inproceedings,book]{giveninit}{\mkbibnamegiven{#1}}

\DeclareNameFormat{family-given}{%
  \usebibmacro{name:family-given}
    {\namepartfamily}
    {\namepartgiveni} 
    {\namepartprefix}
    {\namepartsuffix}%
  \usebibmacro{name:andothers}}

\DeclareFieldFormat{pages}{#1}

 \allowdisplaybreaks[4]

\begin{document}

\title{The eigenvalue gap of inhomogeneous symmetric discrete random matrix}

\author{Zeyan Song}
\address{Shandong University, Jinan, 250100, China.}
\email{zeyansong8@gmail.com}

\author{Hanchao Wang }
\address{Shandong University, Jinan, 250100, China.}
\email{hcwang06@gmail.com}


\begin{abstract}
Let $A$ be an $n\times n$ symmetric random matrix whose upper-triangular entries $(A_{ij})_{i \le j}$ are independently distributed according to possibly non-identical subgaussian distribution. This paper investigates the spectral properties of $A$, including its eigenvalues and eigenvectors. Firstly, we prove that for $k \le n/\log{n}$, $1 \le i \le n-k$ and $\varepsilon \ge 0$,
\begin{align}
    \mathbb{P}\left( \lambda_{i+k}-\lambda_{i} \le \varepsilon n^{-1/2} \right)\le (C\varepsilon)^{(k^{2}+k)/2}+e^{-cn},\nonumber
\end{align}
where $\lambda_{1} \le \dots \le \lambda_{n}$ are the eigenvalues of $A$. 

Secondly, combining the recent result of Yi Han[arXiv preprint: 2506.01155], we give a quantitative estimate of the singular values of $A$. For $c\log{n} \le  k \le \sqrt{n}$ and $\varepsilon \ge 0$, we have
\begin{align}
    \mathbb{P}\left( \sigma_{n-k+1}(A) \le k\varepsilon n^{-1/2} \right) \le (C\varepsilon)^{ck^{2}}+e^{-ckn},\nonumber
\end{align}
where $\sigma_{n}(A) \le \dots \le \sigma_{1}(A)$ are the singular values of $A$.

Finally, based on the distance analytical framework developed for the eigenvalue gap, we further derive quantitative bounds for singular values and delocalization of eigenvectors. In particular, we establish a quantitative bound for the probability that some eigenvector of $A$ exhibits no-gap delocalization, which improves the result of Rudelson and Vershynin [Geom. Funct. Anal. 26 (2016), 1716–1776].

\end{abstract}

\maketitle

\section{Introduction}\label{Introduction}
In recent years, the distributional behavior of eigenvalues and eigenvectors of random matrices has become a central topic in probability theory and random matrix theory. For classical integrable ensembles such as the Gaussian Unitary Ensemble (GUE) and the Gaussian Orthogonal Ensemble (GOE), the global spectral distribution is governed by the Wigner semicircle law, and the local statistics in the bulk are now well understood with the recent progress in random matrix theory. 

Let $\lambda_{1} \le \dots \le \lambda_{n}$ be the eigenvalues of the symmetric random matrix and $\delta_{i}:=\lambda_{i+1}- \lambda_{i}$ be the gap. For GUE, the limiting distribution of $\delta_{i}$ was obtained by Tao \cite{Tao_ptrf}. The four moment condition establishes that this distribution is universal for any random variable that matches the Gaussian up to the first four moments by \cite{Tao_fourmoment}. This condition was removed by Erd\H{o}s and Yau in \cite{EY_jems} using the PDE technique, which allowed the calculation of the eigenvalue gap distribution under both the GOE and Bernoulli assumptions.

However, for the smallest gap $\delta:=\min \delta_{i}$, there are currently no results as strong as those available for typical eigenvalue gaps. For GUE, Arous and Bourgade \cite{AB_aop} proved that the order of the smallest gap is $n^{-4/3}$. For GOE, Feng, Tian, and Wei \cite{FTW_gafa} established the limiting distribution of the smallest gap and determined that its order is $n^{-3/2}$. Nevertheless, the smallest gap does not fall within the scope of classical universality results, such as those in \cite{Tao_fourmoment}. The best currently known result on the smallest gap was obtained by Bourgade \cite{Bour_jems}, but it requires the entries of the random matrix to have extremely smooth distributions and therefore does not apply to the discrete setting such as the Bernoulli assumption.

The classical Wigner semicircle law tells us that, for an $n\times n$ symmetric random matrix , the bulk of the spectrum lies in the interval $[-C\sqrt{n}, C\sqrt{n}]$; with $n$ eigenvalues packed into this interval, the natural scale for the typical spacing is therefore $n^{-1/2}$. When the Hermitian random matrix $A$ with entries have sufficiently smooth distribution and decay assumption, the Wegner estimate was taken by Erd\H{o}s, Schlein and Yau \cite{ESY_imrn}:
\begin{align}
    \mathbb{P}\left(  En^{1/2}-\varepsilon n^{-1/2} \le \lambda_{i} \le \lambda_{i+k-1} \le En^{1/2}+\varepsilon n^{-1/2} \text{ for some } i  \right) = o(\varepsilon^{k^{2}}),\nonumber
\end{align}
for all fixed $k \ge 1 $, $\varepsilon >0$ and all bounded $E \in \mathbb{R}$. Applying the union bound implies the result as follows
\begin{align}
    \mathbb{P}\left( \delta \le \varepsilon n^{-1/2}  \right) = o(n\varepsilon^{3}+e^{-cn}).\nonumber
\end{align}
Furthermore, for a real symmetric random matrix, an analog of this result was obtained with the exponent $k^{2}$ replaced by $k(k+1)/2$ by \cite{BEYY_cpam}. Similarly, this result does not apply to random matrices under discrete assumptions.

In fact, even for the results we discussed that are applicable to the Bernoulli setting, the asymptotic estimates typically only provide a $o_{n}(1)$ rate of convergence. For adjacency matrices $G$ of the Erd\H{o}s-R\'{e}nyi random graph $G(n,1/2)$, a $o_{n}(1)$ bound on the tail probability is far from sufficient.

In 1980s, Babai conjectured that $G$ has no repeated eigenvalues with probability $1-o_{n}(1)$(see \cite{Vu_crmt}). Tao and Vu proved this conjecture in \cite{Tao_comb}. Furthermore, Vu conjectured that $A_n$ has no repeated eigenvalues with probability $1-e^{-\Omega(n)}$.

The conjecture of an exponential upper bound on probability originates from a simple observation:
\begin{align*}
    \left( \frac{1}{4} +o_{n}(1) \right)^{n} \le \mathbb{P}\left( \exists i, i+1: \lambda_{i}=\lambda_{i+1}=0  \right) \le \mathbb{P}\left( \delta=0 \right).
\end{align*}

Next, we introduce a class of random matrices that includes the discrete setting as a special case, which is important in this field and is more closely related to this paper. Let $A_{n}$ be an $n \times n$ random symmetric matrix whose entries $(A_{ij})_{i\le j}$ are independently identically subgaussian distributed. Nguyen, Tao and Vu \cite{Nguyen_ptrf} prove that there exists $0 < c < k$ such that for all $n^{-c} \le \alpha \le c$ and $\delta \ge n^{1-c/\alpha}$
\begin{align}
    \sup_{1 \le i \le n-k}\mathbb{P}\left( \lambda_{i+k}(A_n)-\lambda_{i}(A_n) \le \delta n^{-1/2}  \right) = O\left(  \left( \frac{\delta}{\sqrt{\alpha}}\right)^{\Theta(k^{2})} \right).\nonumber
\end{align}

Furthermore, Nguyen \cite{Nguyen_jfa} improves the result for $k$ to be large enough. He proves that for all $k \ge 1$ and $\varepsilon >0$,
\begin{align}
    \mathbb{P}\left(\exists i \in [n-k] : \lambda_{i+k}(A_n)-\lambda_i(A_n) \le \varepsilon n^{-1/2}  \right) \le C_k n \varepsilon^{k(k+1)/2-1}+e^{-n^{c}},\nonumber
\end{align}
where $C_{k}>0$ depends only on $k$.

However, the exponential term in that estimate did not achieve the sharp bound conjectured by Vu. Very recently, Campos et al. \cite{Campos_pi} proved the following result, resolving Vu’s conjecture. They prove for all $ 1 \le k \le cn$ and $\varepsilon \ge 0$
\begin{align}
    \sup_{1 \le i \le n-k}\mathbb{P}\left( \lambda_{i+k}(A_n)-\lambda_i(A_n) \le \varepsilon n^{-1/2} \right) \le (C\varepsilon)^{k}+e^{-cn}.\nonumber
\end{align}

This result was the first to show that the probability of having a small gap admits an exponential bound. However, in terms of the exponent in terms of the scale $\varepsilon$, the Gaussian case suggests that it should depend quadratically on $k$. In this paper, after removing the identical distribution assumption in $A$, we precisely establish this quadratic dependence while preserving the same exponential bound.

For other random matrix models, the tail bounds of the gaps between singular values or eigenvalues have also been studied, for example, in the sparse setting \cite{KV_aihp}, for the singular value gaps of matrices \cite{HanY} and \cite{CLOS_gap}, and for the eigenvalue gaps of the Laplacian matrices of random graph \cite{CLNW_gap}. However, these models are not the focus of this paper. Here, we study subgaussian symmetric random matrices $A$ with independent upper-triangular entries, after removing the identical distribution assumption.

Before stating the theorem, we introduce the basic assumptions and definitions that will be used throughout the proof. We define the subgaussian random variable $\xi$ by 
\begin{align}
    \mathbb{E}\exp{\left( (\xi/K)^{2} \right)} < \infty,
\end{align}
for some $K>0$ and denote $\Vert \xi \Vert_{\psi_{2}} = \inf{ \left\{ t>0: \mathbb{E}\exp{\left( \xi/t \right)^{2}} \le2 \right\} }$. Then, define $\xi \in S_{2}(T)$ to be that $\xi$ is mean zero, variance $1$ and $\Vert \xi \Vert_{\psi_{2}} \le T$. The proof in this paper can be generalized to any subgaussian random variable with variance between $a$ and $b$, where $a$ and $b$ are positive constants.
\par
For $n \in \mathbb{N}^{+}$, let $\mathrm{Sym}_{n}(T)$ denote the class of $n \times n$ symmetric matrices with $(A_{ij})_{i \le j}$ independent belonging to $S_{2}(T)$ and $G_{n}(T)$ denote the class of vectors of length $n$ with independent coordinates and belonging to $S_{2}(T)$.

In this paper, our main contributions are as follows.
We first establish a tail bound for eigenvalue gaps. Compared with previous results, our bound improves the dependence on the parameter $\varepsilon$ in two aspects: the exponent of $\varepsilon$ is shown to be quadratic in $k$, while the tail probability remains exponentially small. The main technical innovation of this work is to reformulate several spectral problems. In particular, the estimation of eigenvalue gaps within a distance estimation framework. Within this framework, we introduce an iterative scheme together with a new approach based on negative correlation inequalities, which allows us to control the exponent of $\varepsilon$ more effectively. This is our key  ingredient.  On the other hand, the improvement of the distance theorem allows us to improve the tail probability to an exponential bound.

While negative-correlation ideas and iterative arguments have appeared previously in random matrix theory, our approach incorporates a negative-correlation inequality directly into the geometric recursion for eigenvalue gaps. More precisely, the gap event for $A_n$ is related to a projection small-ball event together with a lower-order gap event for a principal minor. This recursive mechanism is a key ingredient in our control of the $\varepsilon$-exponent.

Toward the end of this section, we introduce a novel distance theorem, which serves as a central analytical tool in our approach. The proof of this theorem relies on a new high-dimensional Littlewood-Offord type result. In particular, building upon the toolbox developed by Campos et al. \cite{Campos_jams, Campos_pi}, we establish a high-dimensional negative correlation inequality via a new approach. For clarity of exposition, we defer the detailed arguments and technical developments to Appendix \ref{Section High-dimensional inhomogeneous Littlewood-Offord problem}.

Moreover, this framework provides a unified method capable of handling non-identically distributed subgaussian entries, while simultaneously overcoming a major limitation of previous techniques-namely, the failure to obtain exponential tail bounds. By integrating these advances, we establish the following main theorem.
\begin{mytheo}\label{Theorem A}
    For all $T \ge 1$, there exist $C_{\ref{Theorem A}}$ and $c_{\ref{Theorem A}}$ depending only on $T$ such that the following holds. Let $n \in \mathbb{N}^{+}$ with $n \ge C_{\ref{Theorem A} }$ and $k \in \mathbb{N}^{+}$ with $1 \le k \le c_{\ref{Theorem A}}n/\log{n}$. Let $A$ be an $n \times n$ symmetric random matrix with $A \sim \mathrm{Sym}_{n}(T)$ and denote the $i$-th eigenvalue of $A$ by $\lambda_{i}$. Then for all $1 \le i \le n-k$ and $\varepsilon \ge 0$, we have 
    \begin{align}
        \mathbb{P}\left( \lambda_{i+k}-\lambda_{i} \le  \frac{\varepsilon}{\sqrt{n}}   \right) \le \left( C_{\ref{Theorem A}}\varepsilon \right)^{\frac{k(k+1)}{2}}+e^{-c_{\ref{Theorem A}}n}.\nonumber
    \end{align}
\end{mytheo}
\begin{myrem}
 All the results in this section can be extended to the case where the variances of the matrix entries lie between two positive constants. The upper bound on the variance follows from the discussion of the properties of the log-RLCD, as shown in Lemma \ref{Net size 2}, while the lower bound is derived from the use of the Hanson–Wright inequality in the appendix, specifically Lemma \ref{HS}.
\end{myrem}
If the average gap order is \(n^{-1/2}\), then the gap between the eigenvalues separated by positions \(k\) should be of order \(k n^{-1/2}\), while here we obtain the order of the gap as $n^{-1/2}$. Furthermore, we give the other estimate for the eigenvalue gap and prove that if $k = \Omega(\log{n})$, the order of the gap is equal to $k n^{-1/2}$. Moreover, Theorem \ref{Gaps} accounts for the case where the matrix entries have a non-zero mean by introducing a bias term $F$ with $\Vert F\Vert = O(\sqrt{n})$.
\begin{mytheo}\label{Gaps}
    For $T \ge 1$, $x \in (0,1)$, and $k, n \in \mathbb{N}$, let $A \sim \mathrm{Sym}_{n}(T)$, and let $F$ be a deterministic $n \times n$ matrix with $\|F\| \le 4\sqrt{n}$. Define $\lambda_{i} := \lambda_{i}(A+F)$ to denote the $i$-th eigenvalue of $A+F$. There exist constants $n_{\ref{Gaps}}=n_{\ref{Gaps}}(T)$, $\gamma_{\ref{Gaps}}$, $\gamma_{\ref{Gaps}}'$ and $c_{\ref{Gaps}}>0$ depending only on $T$ and $C_{\ref{Gaps}}>0$ depending only on $T$ and $x$ such that for any $n \ge n_{\ref{Gaps}}$ and $\gamma_{\ref{Gaps}}\log{n} \le k \le  \gamma_{\ref{Gaps}}'n/\log{n}$, we have
    \begin{align}
        \mathbb{P}\left( \exists 1\le i \le n-k+1: \lambda_{i+k-1}-\lambda_{i} \le\frac{\varepsilon}{\sqrt{n}} \right) \le \left(\frac{C_{\ref{Gaps}}\varepsilon}{k} \right)^{g_{\ref{Gaps}}(k)}+e^{-c_{\ref{Gaps}}n}.\nonumber
    \end{align}
    where $g_{\ref{Gaps}}(k):=\frac{(1-x^{2})k^{2}-(1-x)k}{2}-1$.
\end{mytheo}
\begin{myrem}
     By Theorem \ref{Gaps}, we can directly obtain
    \begin{align}
    \sup_{i \le n-k+1}\mathbb{P}\left( \lambda_{i+k-1}-\lambda_{i} \le\frac{\varepsilon}{\sqrt{n}} \right) \le \left( \frac{C\varepsilon}{k} \right)^{ck^{2}} +e^{-cn}.\nonumber
    \end{align}
    However, it should be noted that the exponent of $\varepsilon$ induced by $g_{\ref{Gaps}}(\cdot)$ here is no longer large as in Theorem \ref{Theorem A}. Consequently, when $k$ is sufficiently small, Theorem \ref{Theorem A} furnishes the sharper bound, whereas for large $k$ the relative precision of Theorems \ref{Theorem A} and \ref{Gaps} depends on the actual size of $\varepsilon$.
\end{myrem}
\subsection{Quantitative estimate for the singular values}\label{Quantitative estimate for the singular values}
Indeed, the gap between two eigenvalues separated by $k$ consecutive indices is intimately tied to the $k$-th smallest singular value. We first introduce the extreme behavior of the smallest singular value $\sigma_{n}(A)$ of $A$. The famous Spielman-Teng conjecture \cite{Sconjecture} provides an upper bound estimate for the smallest singular value of symmetric matrices.
\begin{align}
    \mathbb{P}(\sigma_{n}(A) \le \varepsilon n^{-1/2}) \le \varepsilon +e^{-cn}.\nonumber
\end{align}
In this direction, the most extensively studied and well-understood case is that of random matrices $M_n$ with i.i.d. subgaussian entries. During the past 30 years, an important line of study has focused on the singularity probability of random matrices, particularly models whose entries are discrete random variables, such as the Bernoulli matrix $B_n$, whose entries are uniformly distributed on $\{0,1\}$.

In 1995, Kahn, Koml\'{o}s and Szemer\'{e}di \cite{KKS_jams} proved that the probability of $B_n$ being singular is $(c + o_n(1))^n$ with $c = 0.998$. In fact, if we consider the cases where $B_n$ has one zero column, two identical columns, or three linearly dependent columns, a reasonable conjecture is(see, e.g. \cite{KKS_jams}):
\begin{align}
    \mathbb{P}\left( \det B_n=0 \right) = (1+o_{n}(1))n(n-1)2^{-n}.\nonumber 
\end{align}

One of the important lines of work has been the improvement of the constant $c$. Tao and Vu \cite{Tao_rsa, Tao_jams} improved $c$ to $0.75$, and subsequently Bourgain, Vu, and Wood \cite{Bour_jfa} further improved it to $1 / \sqrt{2}$. The best constant $1/2$ was obtained by Tikhomirov \cite{Tikhomirov}, which is nearly optimal.

On the other hand, if we consider the random matrix $G_n$ whose entries are i.i.d. standard Gaussian variables, Edelman’s work \cite{Edelman} provides the following estimate:
\begin{align}
    \mathbb{P}\left( \sigma_{n}(G_n) \le \varepsilon n^{-1/2} \right) \le \varepsilon,\nonumber
\end{align}
for $\varepsilon \ge 0$. If we consider a more general subgaussian random matrix $M_n$ and focus on the exponential term arising from its singularity probability, an important contribution comes from Rudelson and Vershynin \cite{Rudelsonadvance}, who proved that
\begin{align*}
    \mathbb{P}( \sigma_{n}(M_n) \le \varepsilon n^{-1/2}  ) \le C\varepsilon +e^{-cn}.
\end{align*}
This estimate is optimal up to absolute constant factors.

Using a different approach, Tao and Vu \cite{Tao_gafa} compared $\sigma_{n}(M_n)$ and $\sigma_n(G_n)$ and proved that
\begin{align}
    \mathbb{P}( \sigma_n(M_n) \le \varepsilon n^{-1/2}) \le \mathbb{P}(\sigma_n(G_n) \le \varepsilon n^{-1/2}) +O(n^{-c}).\nonumber
\end{align}

Finally, Sah, Sahasrabudhe and Sawhney \cite{SSS_gafa} proved the $(1+o_{n}(1))$-version of the Spielman-Teng conjecture for the i.i.d. model $M_n$, showing that the coefficient of $\varepsilon$ in the small singular value bound can be taken to be $1+o_{n}(1)$.

However, progress on the case of symmetric random matrices $A_n$ with entries $(A_{ij})_{i\le j}$ are i.i.d subgaussian has been comparatively slow. Even the most basic result: that A is nonsingular with probability $1-o_{n}(1)$ was not proved until 2006 by Costello, Tao and Vu \cite{Costello_duke}. For the Spielman-Teng conjecture specialized in symmetric random matrices, Nguyen \cite{Nguyen_duke,Nguyen_ejp} obtained the striking result that for any constant $B > 0$ there exists a constant $A > 0$ such that
\begin{align}
     \mathbb{P}(\sigma_n(A) \le n^{-A}  ) \le n^{-B}.\nonumber
\end{align}

Another approach was taken by Vershynin \cite{Vershynin_rsa}, he proved that
\begin{align}
    \mathbb{P}\left( \sigma_n(A_n) \le \varepsilon n^{-1/2} \right) \le C_\eta\varepsilon ^{1/8-\eta}+e^{-n^{c}},\nonumber
\end{align}
for all $\eta > 0$. After a series of subsequent improvements, Campos et al \cite{Campos_pi} have established the sharpest result to date. They prove that for all $\varepsilon \ge 0$
\begin{align}
    \mathbb{P}\left( \sigma_n(A_n) \le \varepsilon n^{-1/2} \right)\le C\varepsilon +e^{-cn}.\nonumber
\end{align}
This estimate is also optimal up to absolute constant factors. However, beyond these advances, neither the constant factor in $\varepsilon$ nor the exponent in the exponential term have been further improved.

We now return to the $k$-th smallest singular value. For non-Hermitian models, this question has been explored much more thoroughly: once the $k$-th singular value hits zero, one must estimate the probability that the rank of the matrix is at most $n - k$. Kahn, Koml\'{o}s and Szemer\'{e}di \cite{KKS_jams} show that this probability is super-exponentially small.  

When $k$ is kept fixed, Jain et al. \cite{Jain_ecp} and Huang \cite{Huang_arxiv} independently determined the rank distribution for various non-Hermitian ensembles.  Once $k$ is allowed to grow with $n$, Rudelson \cite{Rudelson_aop} provided the estimate

\[
\mathbb{P}(\mathrm{rank}(M_n) \le n-k) \le e^{-ckn}.\nonumber
\]

Turning to the $k$-th singular value itself, Nguyen \cite{Nguyen_jfa} obtained the bound in the subgaussian non-Hermitian setting. 
\begin{align}
    \mathbb{P}(\sigma_{n-k+1}(A) \le k\varepsilon n^{-1/2}) \le (C\varepsilon)^{ck^{2}}+e^{-cn}.\nonumber
\end{align}
In very recent joint work \cite{Rank} with Dai, we proved that for independent, not necessarily identically distributed subgaussian non-Hermitian matrices the $k$-th smallest singular value satisfies
\begin{align}
    \mathbb{P}(\sigma_{n-k+1} \le k\varepsilon n^{-1/2}  ) \le (C\varepsilon)^{ck^{2}}+e^{-ckn}.\nonumber
\end{align}

For symmetric matrices, the problem becomes both more important, due to the link to the Erd\H{o}s–R\'{e}nyi random graph, and considerably more difficult. Nguyen \cite{Nguyen_jfa} established the bound on the $k$-th smallest singular value in this setting:

\[
\mathbb{P}(\sigma_{n-k+1}(A_n)\le k\varepsilon n^{-1/2}) \le (C\varepsilon)^{ck^{2}}+e^{-n^{c}}.\nonumber
\]

 Very recently, Han \cite{HanY2} obtained an upper bound that mirrors Rudelson’s non-Hermitian result. By refining his approach, we can interpolate between these two extremes and prove the following sharper estimate, which is our second main result in this paper.
\begin{mytheo}\label{The_main}
    For $T \ge 1$ and $k, n \in \mathbb{N}$, let $A \sim \mathrm{Sym}_{n}(T)$ and define $\sigma_{i} := \sigma_{i}(A)$ to denote the $i$-th largest singular value of $A$. Then there exist constants $n_{\ref{The_main}}(T)$, $\gamma_{\ref{The_main}}$, $\gamma_{\ref{The_main}}'$, $c_{\ref{The_main}} > 0$, and $C_{\ref{The_main}} > 0$ depending only on $T$, such that for all $n \ge n_{\ref{The_main}}(T)$ and $\gamma_{\ref{The_main}}\log n \le k \le \gamma_{\ref{The_main}}' \sqrt{n}$, we have:
\begin{align}
\mathbb{P}\left( \sigma_{n-k+1}(A) \le \frac{k\varepsilon}{\sqrt{n}} \right) \le \left( C_{\ref{The_main}} \varepsilon \right)^{c'_{\ref{The_main}}k^{2}} + e^{-c_{\ref{The_main}}k n},\nonumber
\end{align}
where $ c_{\ref{The_main}}' < 1/4 $ is an absolute constant. 
\end{mytheo}

\subsection{Distance theorem for random linear space}\label{Distance theorem for random linear space}
We will introduce the following distance theorem, which will be an important tool in our paper. It is worth noting that analyzing via distance is crucial in non-asymptotic random matrix theory. A classic example is in \cite{Rudelsonadvance}, where Rudelson and Vershynin converted the smallest singular value into a distance sum form, analyzing it through random linear spaces. Other examples include works \cite{Rudelson_cpam} and \cite{Nguyen_jfa}. Recently, in \cite{MFV}, Fernandez proposed a new distance theorem. He considered random matrices meeting fourth moment conditions, offering a small ball probability estimate for distance, with methods for handling fourth moment conditions from Livshyts \cite{Livshyts_jam}. However, since we focus on symmetric cases here, we cannot use their methods to remove the subgaussian property (some difficulties come from Lemma \ref{Lem4.4}). Now, we present the distance theorem in this paper.
\begin{mytheo}\label{Distance}
    For $T \ge 1$, $k$ and $n\in \mathbb{N}$,let $A \sim \mathrm{Sym}_{n}(T)$, let $F$ be a deterministic $n \times n$ matrix with $\|F\| \le 4\sqrt{n}$ and $z \in [-8\sqrt{n},8\sqrt{n}]$. Let $H$ be a random subspace in $\mathbb{R}^{n}$ spanned by any $n-k$ columns selected from the random matrix $A+F-zI_{n}$. Then there exist positive constants $n_{\ref{Distance}}(T)$, $\lambda_{\ref{Distance}}$, $c_{\ref{Distance}}$ and $C_{\ref{Distance}}$ depending only on $T$ such that the following holds. For any $n  \ge n_{\ref{Distance}}(T)$ and $1 \le k < \lambda_{\ref{Distance}} n/\log{n}$, there exists a deterministic subset $\mathcal{H}$ of the Grassmannian of $\mathbb{R}^{n}$ satisfying 
    \begin{align}
        \mathbb{P}(H^{\bot} \in \mathcal{H})  \ge 1- \exp{(-c_{\ref{Distance}}n)},\nonumber
    \end{align}
    and such that for any subspace $E$ with $E^{\bot} \in \mathcal{H}$ and $v \in \mathbb{R}^{n}$, we have 
    \begin{align}
        \mathbb{P}_{a}(\mathrm{dist}(a,E+v) \le t) \le \left( \frac{C_{\ref{Distance}} t}{\sqrt{k}}\right)^{k} + \exp{(-c_{\ref{Distance}}n)},\nonumber
    \end{align}
    where $a \sim G_{n}(T)$ and independent of $H$.
\end{mytheo}

\textbf{Organization of this paper}: In Section \ref{Notation}, we provide a notation. In Section \ref{Proof sketch}, we give the proof sketch. In Section \ref{Preliminaries}, we provide the preliminaries for this paper.  
Further, in Section \ref{Distance theorem for inhomogenous symmetric random matrices}, we complete the proof of the distance theorem.  
In Section \ref{Section Tail bound for the gaps of the eigenvalue}, we estimate the tail bound of the gap of eigenvalues by two distinct methods, thereby proving Theorems \ref{Theorem A} and \ref{Gaps} separately. In Section \ref{Section Quantitative estimate for the singular values}, we give a quantitative estimate of the singular values.
Finally, in Section \ref{Application in the eigenvectors of the symmetric matrices}, we establish the no-gap delocalization of eigenvectors by estimating the smallest singular value of a principal submatrix of a symmetric random matrix as the application of distance theorem. 

\section{Notation}\label{Notation}

We denote by $[n]$ the set of natural numbers from $1$ to $n$. Given a vector $x\in \mathbb{R}^{n}$, we denote by $\Vert x\Vert_{2}$ its standard Euclidean norm: $\Vert x\Vert_{2}=\left(\sum_{j\in [n]}{x_{j}^{2}} \right)^{\frac{1}{2}}$, the supnorm is denoted $\Vert x\Vert_{\infty}=\max_{i}{|x_{i}|}$ and for any $\textbf{t}=(t_{1},\dots,t_{n})\in \mathbb{R}^{n} $, the $t$-norm is denoted $\Vert x\Vert_{\textbf{t}}=\left( \sum_{j\in [n]}t_{i}^{2}x_{i}^{2}  \right)^{1/2}$. The unit sphere of $\mathbb{R}^{n}$ is denoted by $S^{n-1}$. The cardinality of a finite set $\mathrm{I}$ is denoted by $\left| \mathrm{I} \right|$. Let $\xi'$ be an independent copy of $\xi$ and define 
\begin{align}
    \overline{\xi} = \xi - \xi'.\nonumber
\end{align}

For $\mu \in (0,1)$, define $B_{\mu}$ to be random variable satisfying
$$ \mathbb{P}(B_{\mu} = 1) = \mu \ and \ \mathbb{P}(B_{\mu} = 0) = 1-\mu.$$
and define $g =(g_{1},\cdots,g_{d}) \in \Phi_{\mu}(d,\xi)$ by
 $$g_{j} = \delta_{j} \xi_{j} \ for \ any \ j \in [d], $$
where $\xi \sim G_{n}(T)$ and $\delta_{j}$ are independent copy of $B_{\mu}$ and independent with $\xi$.

    Define the L\'{e}vy function of $\xi$ and $t$ as 
\begin{align}
    \mathcal{L}\left( \xi,t \right):=\sup_{w \in \mathbb{R}^{n}}{\mathbb{P}\left( \Vert \xi -w \Vert_{2} \le t \right)}.\nonumber
\end{align}

Let $H$ be an $m \times n$ matrix, define $\mathrm{Row}_{j}(H)$($\mathrm{Col}_{j}(H)$) as the $j$-th rows(columns) of $H$. Define $\Vert H\Vert := \sup_{\Vert x\Vert_{2}=1}{\Vert Hx\Vert_{2}}$ and $\Vert H\Vert_{\mathrm{HS}}:=\left( \sum_{i,j}{h_{ij}^{2}} \right)^{1/2} $\\
For $x \in \mathbb{R}^{n}$, define $\mathrm{dist}\left( x,\mathbb{Z}^{n} \right):=\Vert x\Vert_{\mathbb{T}}$. We then denote $\Vert x\Vert_{\xi} := \sqrt{\mathbb{E}\Vert x \star \overline{\xi}\Vert_{\mathbb{T}}^{2}}$ for $\xi \sim G_{n}(T)$, where $x \star y :=(x_{1}y_{1},\dots,x_{n}y_{n})$. For the set $I \subset [n]$, we define $x_{I}$ to be the vector composed of all the coordinates of $x$ whose indices belong to $I$.\\
In the proofs of the results in this paper, we define $c$, $c',\dots$ as some fixed constant and define $c\left( u\right)$, $C\left( u\right)$ as a constant related to $u$, and depend only on the parameter $u$. Their value can change from line to line.

\section{Proof sketch}\label{Proof sketch}

This section outlines the primary conceptual frameworks underpinning the proofs of the principal theorems. We begin by establishing the properties of the eigenvalues of inhomogeneous symmetric matrices, as formulated in Theorems \ref{Theorem A} and \ref{Gaps}. Essentially, these results are recast as variants of the inhomogeneous high-dimensional Littlewood-Offord problem introduced in Section \ref{Distance theorem for random linear space}. Consequently, the first step is to establish Theorem \ref{Distance}.

\subsection{Distance theorem}\label{Distance theorem proof sketch}
We first introduce how to prove Theorem \ref{Distance}. Using the small ball probability via log-RLCD, we can transform the proof of the distance theorem into showing that the probability of vectors in the kernel space of the matrix having a large log-RLCD tends to $1$ subexponentially (a more detailed analysis will be provided in Section \ref{Distance theorem for inhomogenous symmetric random matrices}).   

For this problem, it is well known that according to previous related studies (i.e. \cite{Vershynin_rsa}), the probability of the existence of ``compressible vectors" in the kernel space is exponentially small. Therefore, we can focus solely on the  ``Incompressible vectors" in the kernel space.

To demonstrate that the probability of incompressible vectors having a small randomized logarithmic least common denominator (log-RLCD) is exponentially small, for a vector $v$ randomly selected from the incompressible vectors on the unit sphere,we show that for any \( \varepsilon >e^{-cn/k} \), it satisfies
\begin{align}
    \mathbb{P}_{A}\left( \Vert Av \Vert_{2} \le\varepsilon \sqrt{n} \right) \le\left( C\varepsilon \right)^{n-k},
\end{align}
with high probability, where $C>0$ is some constant.\par
Since $A$ is symmetric, \(\Vert Av\Vert_{2}\) is not an independent sum, which means that we cannot handle it directly effectively. Following the approach in \cite{Campos_jams}, we will replace \(A\) with the ``zeroed-out" matrix for our analysis.
\begin{align}
     M := \begin{bmatrix}
        \mathbf{0}_{[d]\times[d]} & H_{1}^{\mathrm{T}}                  & \mathbf{0}_{[d]\times[k]}\\
        H_{1}                     & \mathbf{0}_{[d+1,m] \times [d+1,m]} & \mathbf{0}_{[d+1,m] \times [k]}
    \end{bmatrix},
\end{align}
where $m=n-k$, $H_{1}$ is a $(m-d)\times d $ random matrix with independent columns $\mathrm{Col}_{j}(H_{1}) \in \Phi_{\mu}(m-d,\xi_{j})$ and $\xi_{j} \in G_{m-d}(T)$.\par
This reduction is justified by the bound
\begin{align}
    \mathbb{P}_{A}\left( \Vert Av\Vert_{2} \le\varepsilon\sqrt{n} \right) \le C^{n}\cdot \mathbb{P}_{M}\left( \Vert Mv\Vert_{2} \le\varepsilon\sqrt{n}  \right).
\end{align}
Now, we consider the following situation: for some $v \in S^{n-1}$ and $L  \ge2$, let the maximum $\varepsilon \in (0,1)$ for 
\begin{align}
    \mathbb{P}_{M}\left( \Vert Mv\Vert_{2} \le\varepsilon\sqrt{n} \right)  \ge\left( L\varepsilon \right)^{n-k}.
\end{align}
Furthermore, consider the \(\varepsilon\)-net of these vectors. We can analyze the event on a discrete grid. Let us denote this grid by $\mathcal{B}$. Suppose $X$ is a vector randomly chosen uniformly from $\mathcal{B}$. Note that we need to demonstrate the following probability:
\begin{align}\label{iort}
    \mathbb{P}_{X}\left(  \mathbb{P}_{M}\left( \Vert MX\Vert_{2} \le\varepsilon\sqrt{n} \right)  \ge\left( L\varepsilon \right)^{n-k}   \right) \le\left( C/L^{2} \right)^{n-k}.
\end{align}
This constitutes a novel inhomogeneous “inversion of randomness” technique.
We defer this part of the proof to the appendix; there, we extend the approach in \cite{Campos_jams} to the inhomogeneous setting by log-RLCD. In particular, we give the novel High-dimensional inhomogeneous negative correlation inequality for the $m \times n$ matrix $V$ with the ``large log-RLCD" and $k \times n$ ``almost orthogonal" matrix $W$ as follows:
\begin{align*}
    \mathbb{P}\left( \Vert V\xi\Vert_{2} \le \varepsilon \sqrt{m} \text{ and } \Vert W\xi\Vert_{2} \le c\sqrt{k}  \right)\le (C\varepsilon)^{m}e^{-ck},
\end{align*}
where $\xi \sim G_{n}(T)$.
\subsection{Tail bound for the gaps of the eigenvalue}\label{Tail bound for the gaps of the eigenvalue proof sketch}
In addition, we will provide an analysis of the eigenvalues of inhomogeneous symmetric random matrices in Section \ref{Section Tail bound for the gaps of the eigenvalue} using the geometric approach.

The proof of Theorem \ref{Theorem A} roughly follows the approach developed in \cite{Nguyen_ptrf}. Nguyen, Tao and Vu introduced a geometric method for estimating the tail bounds of eigenvalue gaps of random matrices. Adapting their ideas, we consider a random matrix $ A_n$ and its principal submatrix $A_{n-1}$. Denote their eigenvalues by
\[
\lambda_{1}(A_{n}) \le \dots \le \lambda_{n}(A_{n}), \quad
\lambda_{1}(A_{n-1}) \le \dots \le \lambda_{n-1}(A_{n-1}).
\]
If for some $i$ and $k$ we have
\[
\lambda_{i+k}(A_{n}) - \lambda_{i}(A_{n}) \le \epsilon n^{-1/2},
\]
then we can find $k$ orthonormal eigenvectors $ v_{1}, \dots, v_{k} $ of $ A_{n-1} $. Let $H$ be the subspace spanned by these $k$ vectors, and denote by $P_{H}$ the linear projection onto $H$. We then obtain:
\begin{align*}
        & \mathbb{P}( \lambda_{i+k}(A_{n})-\lambda_{i}(A_{n}) \le \varepsilon n^{-1/2} )\\
        \le & C\mathbb{P}_{X}( \Vert P_{H}X \Vert_{2} \le \varepsilon\sqrt{k} |H ) \mathbb{P}(\lambda_{i+k-1}(A_{n-1})-\lambda_{i}(A_{n-1}) \le \varepsilon n^{-1/2}).
\end{align*}
Note that $H$ depends only on $A_{n-1}$, while $X$ is independent of $A_{n-1}$. This negative correlation inequality is the basis for the iterative scheme. The inequality captures the properties of the eigenvalues of submatrices captured by the above inequality. This improves on the method of Campos et al. in \cite{Campos_pi}, where only the $k$ different inner products between $k$ random vectors and $k$ eigenvectors were considered by other methods. Thus, it suffices to analyze the arithmetic structure of $H$. Applying a strengthened version of our distance theorem, we can complete the proof of Theorem \ref{Theorem A}.

The second theorem concerns the distribution of $\lambda_{i+k}$ and $\lambda_{i}$. First, with high probability, the eigenvalues are in the interval \([-4\sqrt n,\,4\sqrt n]\). Fix a center point \(z\) and the average gap \(n^{-1/2}\); we will consider
\begin{align}
    \mathbb{P}\left(  z-\varepsilon n^{-1/2} \le \lambda_{i} \le \lambda_{i+k} \le z+\varepsilon n^{-1/2}  \right).\nonumber
\end{align}
Furthermore, We then apply the union bound to estimate the tail bound of the gap of $\lambda_{i}$ and $\lambda_{i+k}$. 

\subsection{Quantitative estimate for the singular values}\label{Quantitative estimate for the singular values proof sketch}
Finally, we give the proof of Theorem \ref{The_main}. In view of the above approach, it suffices to show that for the linear space $H$ from Theorem \ref{Distance}, there exists a subspace $E$ with $\dim{E} \ge k/8$ and $\mathrm{RlogD}_{L,\alpha}^{a}(E) \ge \exp(Cn/k)$ with probability at least $1-\exp(-ckn)$.

Therefore, we need to partition the space instead of working with a single vector as in the previous proof. Our partitioning strategy is inspired by Han \cite{HanY2}, with an earlier foundational concept from Rudelson \cite{Rudelson_aop}. We decompose the kernel space $H$ into four components:
\begin{itemize}
   \item  The compressible subspace
    \item The subspace fails the delocalization property
    \item The incompressible and delocalized subspace with small log-RLCD
    \item The subspace with large log-RLCD
\end{itemize}
Each of these subspaces has dimension order $k$. Consequently, by replacing the vector in the earlier proof with the almost orthogonal system of vectors in the partitioned space and combining the high-dimensional estimates with the arguments developed in the appendix, this completes the proof of Theorem \ref{The_main}. A detailed discussion is provided in Section \ref{Section Quantitative estimate for the singular values}.
\section{Preliminaries}\label{Preliminaries}
In this section, we collect several fundamental tools and auxiliary results that will be used throughout the proof of the distance theorem and the main results. In particular, we introduce the notion of the randomized logarithmic least common denominator (log-RLCD), establish its key properties, and develop the necessary discretization and net arguments for vectors on the sphere. These ingredients will play a central role in controlling small ball probabilities and handling the arithmetic structure of vectors in later sections.
\subsection{Randomized logarithmic least common denominator}\label{Randomized logarithmic least common denominator}
Additionally, the above discussion assumes that the elements of the random vector are independent and identically distributed (i.i.d.). For cases where the elements are not identically distributed, \cite{Livshyts_aop} introduced the concept of Randomized Least Common Denominator (RLCD) to estimate the small ball probability for random variables that do not have the i.i.d. property.
Based on RLCD, \cite{Rank} and \cite{MFV} introduced the Randomized Logarithmic Denominator (RlogD) concept. This concept extends the small ball probability estimates for matrices, initially developed by Rudelson and Vershynin in \cite{Rudelson_gafa}, to scenarios where the elements are not identically distributed.

We establish a small ball probability estimate for high-dimensional random vectors with different distributions. Before this, we begin with the following definition.
\begin{mydef}\label{RlogD}
Let $V$ be an $m \times n$ (deterministic) matrix, $\xi=\left(\xi_{1},\dots,\xi_{n}\right)$ be a real-valued random vector with independent entries, let $L>0$ and $\alpha \in\left(0,1\right)$. Define the Randomized logarithmic least common denominator(log-RLCD) of $V$ and $\xi$ by
\begin{align}
    \mathrm{RlogD}_{L,\alpha}^{\xi}\left(V\right) := \inf{\left\{\Vert \theta\Vert_{2}:\theta \in \mathbb{R}^{m},\Vert V^{T}\theta\Vert_{\xi}  < L \sqrt{\log_{+}{\frac{\alpha \Vert \theta\Vert_{V}}{L}}} \right\}},\nonumber
\end{align}
where $\Vert \theta \Vert_{V}:=\left( \sum_{i=1}^{m}\Vert V_{i}\Vert_{2}^{2}\theta_{i}^{2} \right)^{1/2} $, $V_{i}$ denotes the i-th row of $V$ and $\Vert x\Vert_{\xi}:=\sqrt{\mathbb{E}\mathrm{dist}^{2}\left( x\star \overline{\xi},\mathbb{Z}^{n} \right)}$.\\
As two special cases, if $m=1$, for a vector v, its log-RLCD is
\begin{align}
    \mathrm{RlogD}_{L,\alpha}^{\xi}\left( v \right):= \inf{\left\{ \theta > 0:\Vert \theta v\Vert_{\xi}  < L \sqrt{\log_{+}{\frac{\alpha \Vert \theta v\Vert_{2}}{L}}} \right\}}.\nonumber
\end{align}
If $E \subset \mathbb{R}^{n}$ is a linear subspace, this definition extends naturally to the orthogonal projection $P_{E}$ on $E$ setting
\begin{align}
\mathrm{RlogD}_{L,\alpha}^{\xi}\left(E\right)= \inf{\left\{\Vert y\Vert_{2}:y \in E,\Vert y\Vert_{\xi}  < L \sqrt{\log_{+}{\frac{\alpha \Vert y\Vert_{2}}{L}}} \right\}}.\nonumber
\end{align}

\end{mydef}
\begin{myrem}
    Note that here we use $\Vert \theta\Vert_{V}$ to replace $\Vert V\theta \Vert_{2}$ in the previous definitions, which is similar to the definition in \cite{HanY1}. On the one hand, we can see that in the two special cases mentioned above, this definition is consistent with the previous ones. For more general matrices $V$, by slightly modifying the proofs of the small ball probabilities of \cite{Rank} or \cite{MFV}, we can show that the log-RLCD defined here can serve the same purpose as the Randomized Logarithmic Least Common Denominator in \cite{MFV} or \cite{Rank}.
\end{myrem}
We recall the following small ball probability estimate, which has been found in \cite{Rank} and \cite{MFV}. Note that for a linear space, the definition of log-RLCD in this paper is identical to that in \cite{MFV}. Therefore, we can directly apply the conclusions of \cite{MFV}. We have
\begin{mylem}[Proposition 4.1 \cite{MFV}]\label{Lem6.1}
    Consider a real-valued $\xi \sim G_{n}(T)$, let $E \subseteq \mathbb{R}^{n}$ with $\mathrm{dim}(E)  \ge m$ and $P_{E}$ be an orthogonal projection of $E$. For any $\alpha \in (0,1)$, $L > c_{\ref{Lem6.1}}\sqrt{m}$, we have 
    \begin{align}
        \mathcal{L}\left( P_{E}\xi,t\sqrt{m} \right) \le \left( \frac{C_{\ref{Lem6.1}}L}{\alpha\sqrt{m}} \right)^{m}\left( t+\frac{\sqrt{m}}{\mathrm{RlogD}_{L,\alpha}^{\xi}(E)}  \right)^{m},\nonumber
    \end{align}
    where $C_{\ref{Lem6.1}}$, $c_{\ref{Lem6.1}} > 0$ are absolute constants.
\end{mylem}
We next establish two properties of log-RLCD. The first property is that incompressible vectors have a larger RlogD. This was introduced in \cite{Rank}.
\begin{mylem}[Lower bound of the log-RLCD]\label{Lower bound of the logRLCD}
    For any $Q$ and $y \in \mathrm{Incomp}(\delta,\rho)$, there exist $C_{\ref{Lower bound of the logRLCD}}:=C_{\ref{Lower bound of the logRLCD}}(T,\delta,\rho)$ and $c_{\ref{Lower bound of the logRLCD}}:=c_{\ref{Lower bound of the logRLCD}}(T,\delta,\rho)$ such that for any $u \le c_{\ref{Lower bound of the logRLCD}}$,
    \begin{align}
        \mathrm{RlogD}_{Q,u}^{\xi}(y)  \ge C_{\ref{Lower bound of the logRLCD}}\sqrt{n},\nonumber
    \end{align}
    where $\xi \sim G_{n}(T)$ and $\delta$, $\rho$ are the constants from Lemma \ref{Comp}.
\end{mylem}
Next, we introduce the second property of log-RLCD.
\begin{mylem}
    [Stability of the log-RLCD]\label{Stability of the logRLCD}For $\delta$, $\rho \in (0,1)$, $\xi \sim G_{n}(T)$ and $D  \ge C_{\ref{Lower bound of the logRLCD}}\sqrt{n}  \ge1$, let $x\in \mathrm{Incomp}(\delta,\rho)$ satisfy $\mathrm{RlogD}_{Q,u}^{\xi}(x) \in [D,2D]$, where $2 \le Q \le 2uC_{\ref{Lower bound of the logRLCD}}\sqrt{n}$, then for all $y$ with $\Vert x-y \Vert_{\infty} \le \frac{\Delta}{\sqrt{n}}$, where $\Delta:=(2D)^{-1}$. We have 
    \begin{align}
        \mathrm{RlogD}_{2Q,4u}^{\xi}(y) \le2D.\nonumber
    \end{align}
\end{mylem}
\begin{proof}
    If $x$ satisfies $\mathrm{RlogD}_{Q,u}^{\xi}(x) \in [D,2D]$, there exists $\theta \in[D,2D]$ such that
    \begin{align}
        \mathbb{E}\Vert \theta \cdot \overline{\xi}\star x\Vert_{\mathbb{T}}^{2} < Q^{2}\log_{+}{\frac{u\theta\Vert x\Vert_{2}}{Q}}.\nonumber
    \end{align} 
   In fact, $$\Vert x\Vert_{2} \le\Delta +\Vert y \Vert_{2} \le2\Vert y\Vert_{2},$$ and $$4\Delta^{2}\theta^{2} \le 4 \le2Q^{2}\log_{+}{\frac{uC_{\ref{Lower bound of the logRLCD}}\sqrt{n}}{Q}} \le2Q^{2}\log_{+}{\frac{u\theta \Vert x\Vert_{2}}{Q}},$$ thus,
    \begin{equation}
        \begin{aligned}
            \mathbb{E}\Vert \theta \cdot \overline{\xi}\star y\Vert_{\mathbb{T}}^{2} 
            & < 2\mathbb{E}\Vert \theta \cdot \overline{\xi}\star (x-y)\Vert_{2}^{2} + 2\mathbb{E}\Vert \theta \cdot \overline{\xi}\star x\Vert_{\mathbb{T}}^{2}\\
            & < 4\Delta^{2}\theta^{2}+2Q^{2}\log_{+}{\frac{u\theta \Vert x\Vert_{2}}{Q}}\\
            & < 4Q^{2}\log_{+}{\frac{2u\theta \Vert y\Vert_{2}}{Q}},\nonumber
        \end{aligned}
    \end{equation}
    This implies that 
    \begin{align}
        \mathrm{RlogD}_{2Q,4u}^{\xi}(y) \le2D.\nonumber
    \end{align}
\end{proof}
\subsection{Decomposition and net of the sphere}\label{Decomposition and net of the sphere}
In this subsection, we construct nets for unit vectors. We begin by recalling standard notions and introducing a decomposition of the sphere.
\begin{mydef}
Let $\delta,\rho \in \left( 0,1 \right)$, we define the sets of sparse, compressible and incompressible vectors as follows:

\begin{itemize}
   \item  $\mathrm{Sparse}\left( \delta \right)=\left\{ x \in \mathbb{R}^{n} : \left| \mathrm{supp}\left( x \right)\right| \le\delta n \right\};$
    \item $\mathrm{Comp}\left( \delta,\rho \right)=\left\{x \in S^{n-1} : \mathrm{dist}\left( x,\mathrm{Sparse}\left( \delta \right) \right) \le\rho \right\};$
    \item $\mathrm{Incomp}\left( \delta,\rho \right)=S^{n-1}\setminus \mathrm{Comp}\left( \delta,\rho \right)$.
\end{itemize}
 
\end{mydef}
Based on this definition, we will present the following lemma, which can also be viewed as a version of Proposition 4.2 in \cite{Vershynin_rsa}. It indicates that for the linear space $H$ given in Theorem \ref{Distance}, its orthogonal vectors are almost incompressible.
\begin{mylem}\label{Comp}
     Let $H^{(k)}$ be a random subspace in $\mathbb{R}^{n}$ spanned by $\mathrm{C}_{1} , \dots, \mathrm{C}_{n-k}$ of $A \sim \mathrm{Sym}_{n}(T)$ for all $k ,n\in \mathbb{N}$ with $0 \le k < n$ and $T >1$. There exist $\rho$, $\delta \in (0,1) $ and $c_{\ref{Comp}}$ depending only on $T$ such that
    \begin{align}
        \mathbb{P}_{A}\left( \exists \ v \in \mathrm{Comp}(\delta, \rho): \Vert P_{H^{(k)}}v\Vert_{2} \le c_{\ref{Comp}}^{-1}\sqrt{n}   \right)  \le  e^{-c_{\ref{Comp}}n},\nonumber
    \end{align}
    where $P_{H^{(k)}}$ is orthogonal projection of $H^{(k)}$.
\end{mylem}
In the following discussion of this paper, we fix the constants $\delta$ and $\rho$ obtained from this lemma. We will introduce the ``spread" property of incompressible vectors, which was introduced by Rudelson and Vershynin \cite{Rudelsonadvance}.
\begin{mylem}\label{Spread incomp}
    Let $v \in \mathrm{Incomp}(\delta,\rho)$, then 
    \begin{align}
        (\rho/2)n^{-1/2} \le|v_{i}| \le\delta^{-1/2}n^{-1/2}\nonumber
    \end{align}
    for at least $\rho^{2}\delta n/2$ values of $i\in [n]$.
\end{mylem}
\begin{mylem}\label{Spectral norm estimate}
       Let $A \sim \mathrm{Sym}_{n}(T)$
   \begin{align}
       \mathbb{P}\left( \Vert A\Vert \ge 4\sqrt{n} \right) \le e^{-c_{\ref{Spectral norm estimate}}n}.\nonumber
   \end{align}
\end{mylem}

Let $\kappa_{0}=\rho/3$ and $\kappa_{1}=\delta^{-1/2} + \rho/6$, where $\delta$, $\rho$ have been fixed in Lemma \ref{Comp}. For $I \subset [n]$, define a ``divergent set" as:
\begin{align}
    \mathcal{I}(I) := \left\{ v \in S^{n-1} : (\kappa_{0}+\kappa_{0}/2)n^{-1/2} \le|v_{i}| \le(\kappa_{1}-\kappa_{0}/2)n^{-1/2} \text{ for all } i \in I\right\},\nonumber
\end{align}
and assume that $m := n-k$ and $k \le\rho^{2}\delta n/4$. It follows that $m < n < 2m$ and get that if $v \in \mathrm{Incomp}(\delta,\rho)$, there exists $I \subset [m]$ with $|I|  \ge \rho^{2}\delta n/4$, such that $v \in \mathcal{I}(I)$. This means that the orthogonal vectors of $\mathrm{C}_{1}, \dots,\mathrm{C}_{m}$ of $A$ almost belong to a ``spread set".

We introduce the definition of a ``zeroed-out" matrix, an approach introduced in \cite{Campos_pi}. Let $A_{k}^{T}$ be an $n \times m$ random submatrix of $A \sim \mathrm{Sym}_{n}(T)$. Without loss of generality, we can assume that the submatrix formed by the first $n-k$ columns of $A_{k}$ is symmetric. For $c_{0}>0$(is a small constant to be determined later), we set $d:=c_{0}^{2}n$ and define $M_{k} \sim \mathcal{M}_{\nu}(T)$ as the $m \times n$ random matrix:
\begin{align}
    M_{k} := \begin{bmatrix}
        \mathbf{0}_{[d]\times[d]} & H_{1}^{\mathrm{T}}                  & \mathbf{0}_{[d]\times[k]}\\
        H_{1}                     & \mathbf{0}_{[d+1,m] \times [d+1,m]} & \mathbf{0}_{[d+1,m] \times [k]}
    \end{bmatrix},
\end{align}
where $H_{1}$ is a $(m-d) \times d$ random matrix with $\mathrm{row}_{j}(H_{1}) \in \Phi_{\nu}(d,\xi_{j})$, where $\nu := 2^{-14}$ and $\xi_{j} \in G_{d}(T)$.

In particular, the matrix $M$ will be useful for analyzing the upward bound of $\Vert A_{0}v \Vert $, when $v \in \mathcal{I}([d])$. Without loss of generality, we will analyze the ``divergent set" in $[d]$ as follows.

For $L>0$ and $v \in \mathbb{R}^{n}$, define the threshold function of $v$ and $k$ as
\begin{align}
    \textbf{g}_{L}(v,k):=\sup{ \left\{ t \in [0,1]: \mathbb{P}\left( \Vert M_{k}v \Vert_{2} \le t \sqrt{n} \right)  \ge(4Lt)^{n-k} \right\}},
\end{align}
and define the ``structure set" as
\begin{align}
    \Sigma_{\varepsilon,k}:=\left\{ v \in \mathcal{I}([d]) : \textbf{g}_{L}(v,k) \in [\varepsilon,2\varepsilon] \right\}.
\end{align}
The main objective of this section is to approximate $\Sigma_{\varepsilon}$ by the small size net. For this, we first provide several definitions. Set the event $\mathcal{E}:=\{\Vert A\Vert \le4n^{1/2} \}$, and for $v \in \mathbb{R}^{n}$, $\varepsilon > 0$, define
\begin{align}
    \mathcal{L}_{T,k}(v,\varepsilon \sqrt{n}):=\sup_{w \in \mathbb{R}^{n}}{\mathbb{P}(\Vert A_{k}v-w\Vert_{2} \le\varepsilon \sqrt{n},\mathcal{E})}.
\end{align}
We define
\begin{align}
    \mathcal{I}'([d]):=\left\{ v\in \mathbb{R}^{n}: \kappa_{0}n^{-1/2} \le|v_{i}| \le\kappa_{1}n^{-1/2} \text{ for all } i \in [d] \right\},\nonumber
\end{align}
and the net
\begin{align}
    \Lambda_{\varepsilon,k} := B_{n}(0,2) \cap \left( 4\varepsilon n^{-1/2} \cdot\mathbb{Z}^{n} \right) \cap \mathcal{I}'([d]),
\end{align}
and our main “discretized grid”
\begin{align}
    \mathcal{N}_{\varepsilon,k}:=\left\{ v \in \Lambda_{\varepsilon,k}: (L\varepsilon)^{m} \le\mathbb{P}(\Vert M_{k}v \Vert_{2} \le4\varepsilon\sqrt{n}) \ and \ \mathcal{L}_{T,k}(v,\varepsilon\sqrt{n}) \le(2^{12}L\varepsilon)^{m} \right\}.\nonumber
\end{align}
We next analyze the properties of structured vectors in two parts. The first part is the size of ``discretized grids" and the second approximate structured vectors through ``discretized grids".
\begin{mypropo}[Size of the net]\label{Net size}
    There exists an absolute positive constant $c_{\ref{Net size}}$ such that the following holds. For all $T \ge 1$, $n \in \mathbb{N}$, $L  \ge 8/\kappa_{0}$ and $0 < c_{0} < c_{\ref{Net size}}T^{-4}$, let $d\in [c_{0}^{2}n/4,c_{0}^{2}n]$, $d  \ge 2^{7}T^{4}$. There exist positive constants $n_{\ref{Net size}}  = n_{\ref{Net size}}(T,L,c_{0})$ and $Q_{\ref{Net size}}(T,L,c_{0})$ depending on $T$, $L$ and $c_{0}$, $C_{\ref{Net size}}$ depending only on $T$ such that the following holds. For $n \ge n_{\ref{Net size}}$ and $ (4L)^{-1}>\varepsilon > \kappa_{0}e^{-Q_{\ref{Net size}}d}/4$, then 
    \begin{align}     
    \left| \mathcal{N}_{\varepsilon,k}  \right| \le L^{-2m}\left( \frac{C_{\ref{Net size}}}{c_{0}^6 \varepsilon} \right)^{n}.\nonumber
    \end{align}
\end{mypropo}
We will prove this proposition in the appendix. Our result uses standard tools developed in \cite{Campos_jams, Campos_pi}, but cannot be deduced without a specific and laborious study. 
\subsection{Approximating for structure vectors}\label{Approximating for unstructure vectors}
This subsection establishes the following approximation property: the structure vector can be approximated by points in the ``discretized grid''. This proposition is a slight variant of Lemma 9.1 in \cite{Campos_jams}, and we omit the proof but instead present a lemma that highlights the main difference, which arises because we have removed additional $k$ column vectors from $A_{k}$.

We first present the following lemma, the main tool for proving the approximation result. The details of the proof are very similar to those in \cite{Campos_jams}, the main difference being that the matrix $A_{k}$ we consider is a $m \times n$ matrix rather than the $n \times n$ random matrix in \cite{Campos_jams}. Hence, our coefficients are slightly different. The specific proof is not provided here.
\begin{mylem}\label{Lem5.12}
    For $v \in \mathbb{R}^{n}$ and $t  \ge\textbf{g}_{L}(v,k)$, we have 
    \begin{align}
        \mathcal{L}\left( A_{k}v,t\sqrt{n} \right) \le(200Lt)^{m}.\nonumber
    \end{align}
\end{mylem}

Now, we present the proposition. 
\begin{mypropo}\label{approximate1}
    Let $\varepsilon \in (\exp{\left( -c_{\ref{approximate1}}n \right)},\kappa_{0}/8)$, where $c_{\ref{approximate1}}\in (0,1)$ is a constant that depends only on $T$. For each $v \in \Sigma_{\varepsilon,k}$, there exists $u \in \mathcal{N}_{\varepsilon,k}$ such that $\Vert u-v\Vert_{\infty} \le4\varepsilon n^{-1/2}$.
\end{mypropo}
\section{Distance theorem for inhomogeneous symmetric random matrices}\label{Distance theorem for inhomogenous symmetric random matrices}

In this section, we complete the proof of Theorem \ref{Distance}. We introduce the notion of RlogD and its key properties, establish a strengthened version of the distance theorem, and then deduce Theorem \ref{Distance} from this stronger result.

\subsection{Discretization of sphere via log-RLCD}\label{Discretization of sphere via log-RLCD}
In this section, we discuss how to discretize the unit sphere. The results of this subsection are based on \cite{MFV}. 
Define
\begin{align}
    S_{D,Q,u}^{\xi}:=\left\{ x \in \mathrm{Incomp}(\delta,\rho): \mathrm{RlogD}_{Q,u}^{\xi}(x) \in [D,2D]  \right\},\nonumber
\end{align}
\begin{align}
    M_{\varepsilon}:=\left( \frac{3}{2}B_{2}^{n} \setminus \frac{1}{2}B_{2}^{n} \right) \cap  \frac{\varepsilon}{\sqrt{n}}\mathbb{Z}^{n}\nonumber
\end{align}
and 
\begin{align}
     \Xi_{\varepsilon}:=\left( \frac{3}{2}B_{2}^{n} \cap \left\{ x \in \mathbb{R}^{n}:\left| \left\{ i:|x_{i}|  \ge\frac{\rho}{8\sqrt{n}}  \right\}  \right|  \ge\rho^{2}\delta n/2  \right\}  \right) \cap \frac{\varepsilon}{\sqrt{n}}\mathbb{Z}^{n}.\nonumber
\end{align}
\begin{mylem}\label{Net size 2}
    There exist $n_{0}:=n_{0}(T)$ and $u_{1}=u_{1}(T) \in (0,1/4)$ such that the following holds. Let $n  \ge n_{0}$, $u \le u_{1}$, $Q>0$ and $X \sim G_{n}(T)$ be a random vector satisfying 
    \begin{align}
        \mathbb{E}\Vert X\Vert_{2}^{2} \le\frac{1}{8}p^{2}\delta \gamma^{2} n^{2}.\nonumber
    \end{align}
    Fix $\varepsilon$ satisfying $6^{-n} \le\varepsilon \le0.01$ and let $W$ be chosen uniformly at random from $\Xi_{\varepsilon}$, then 
    \begin{align}
        \mathbb{P}_{W}\left( \mathrm{RlogD}_{Q,u}^{X}(W) < \min{\left( \frac{2Q}{u}e^{n\left( \gamma/Q^{2} \right)},\varepsilon^{-1}  \right)}  \right) \le\left( C_{\ref{Net size 2}}\gamma \right)^{c_{\ref{Net size 2}}n},\nonumber
    \end{align}
    where $C_{\ref{Net size 2}}$, $c_{\ref{Net size 2}} > 0$ depending only on $T$.
\end{mylem}
Lemma \ref{Net size 2} is quite similar to a lemma from \cite{DF}, we omit the proof. The following lemma presents how to approximate $S_{D,Q,u}^{\xi}$.
\begin{mylem}\label{Lem6.5}
    Let $D  \ge C_{\ref{Lower bound of the logRLCD}} \sqrt{n}$, $n  \ge n_{1}(T)$, $Q \in [2,2uC_{\ref{Lower bound of the logRLCD}}\sqrt{n}]$ and $\xi \sim G_{n}(T)$, for any $x \in S_{D,Q,u}^{\xi}$, there exist $y \in \Xi_{\varepsilon}$ with $\Vert x- y \Vert_{\infty} \le\frac{\varepsilon}{\sqrt{n}}$, where $\varepsilon :=(2D)^{-1}$, such that
    \begin{align}
        \mathrm{RlogD}_{2Q,4u}^{\xi}(y) \le2D.\nonumber
    \end{align}
\end{mylem}
\begin{proof}
    To prove this lemma, we only need to take $y \in \frac{\varepsilon}{\sqrt{n}}\mathbb{Z}^{n} $ such that each coordinate is closest to the corresponding coordinate of $x$. Thus, we get $\Vert x-y \Vert_{\infty} \le\frac{\varepsilon}{\sqrt{n}}$, then $\mathrm{RlogD}_{2Q,4u}^{\xi}(y) \le2D$ according to Lemma \ref{Stability of the logRLCD}.
    
     Choose $ y \in \frac{3}{2}B_{2}^{n}$ such that $\Vert y \Vert_{2} \le\Vert x\Vert_{2} + \varepsilon \le\frac{3}{2}$.
    Furthermore, there exists $J$ with $|J|  \ge\rho^{2}\delta n/2$ satisfying 
    \begin{align}
        |x_{j}|  \ge\frac{\rho}{2\sqrt{n}},\nonumber
    \end{align}
    for all $j \in J$. Thus, we get 
    \begin{align}
        |y_{j}|  \ge|x_{j}| -\frac{\varepsilon}{\sqrt{n}}  \ge\frac{\rho}{8\sqrt{n}},\nonumber
    \end{align}
    for all $j \in J$, where $\varepsilon = \frac{1}{2D} \le\frac{\rho}{4}$. We complete the proof by taking $n_{1}(T):=4(h\rho)^{-2}$.
\end{proof}
\subsection{Stronger version of the distance theorem}\label{Stronger version of the distance theorem}
In this subsection, we prove the stronger version of Theorem \ref{Distance}, which asserts that the RlogD of any vector $x$ for which $\Vert Ax \Vert_{2}$ is small must be large.
\begin{mytheo}\label{Large RlogD with small distance}
    For $n \in \mathbb{N}^{+}$ and $k \in \mathbb{N}$. Let $A \sim \mathrm{Sym}_{n}(T)$ and $F \in \mathbb{R}^{m \times n}$ with $\Vert F\Vert \le 12\sqrt{n}$ and $m:=n-k$. There exists $n_{\ref{Large RlogD with small distance}}$, $c_{\ref{Large RlogD with small distance}}$ and $C_{\ref{Large RlogD with small distance}}$ depending only on $T$ such that the following holds. Let $Q \in [2,C_{\ref{Large RlogD with small distance}}\sqrt{n}]$, $\alpha \in (0,c_{\ref{Large RlogD with small distance}})$ and $n \ge n_{\ref{Large RlogD with small distance}}$.  For all $c_{\ref{Large RlogD with small distance}} > t \ge  \max{ \{ \sqrt{n}e^{-n/k},\frac{c_{\ref{Large RlogD with small distance}}\sqrt{n}}{2Q}e^{-c_{\ref{Large RlogD with small distance}}n/(Q^{2} )} \} }$, we have
    \begin{align}
        \mathbb{P}\left(\exists x \in S^{n-1}: \mathrm{RlogD}_{Q,\alpha}^{\xi}(x) \le \sqrt{n}/t \text{ and } \Vert (A_{k}+F)x\Vert_{2} \le t   \right) \le e^{-cn},\nonumber
    \end{align}
    where $A_{k}$ was defined in Subsection \ref{Decomposition and net of the sphere}.
\end{mytheo}
\begin{proof}
    We may assume $k\ge 1$, as the case $k=0$ can be reduced to replace $A$ by $A_{1}$. Set 
    \begin{align}
        E_{t}:=\{ x : \mathrm{RlogD}_{Q,\alpha}^{\xi}(x) \le \sqrt{n}/t \text{ and } \Vert (A_{k}+F)x\Vert_{2} \le t \}
    \end{align}

    Applying Lemma \ref{Comp} and noting that $t < 1$, we have
    \begin{equation}
        \begin{aligned}
            \mathbb{P}_{A}\left( \exists x \in E_{t}\cap S^{n-1} \right) \le \mathbb{P}_{A}\left( \exists x \in \mathrm{Incomp}(\delta,\rho)\cap E_{t}\right)+e^{-c_{\ref{Comp}}n} .\nonumber
        \end{aligned}
    \end{equation}
    Recall the ``Spread set":  
   \begin{align}
    \mathcal{I}(I) := \left\{ v \in S^{n-1} : (\kappa_{0}+\kappa_{0}/2)n^{-1/2} \le|v_{i}| \le(\kappa_{1}-\kappa_{0}/2)n^{-1/2} \text{ for all } \ i \in I\right\}.\nonumber
\end{align}
For any $d:=c_{0}^{2}n$ and $k \le\rho^{2}\delta n/4$, 
\begin{align}
    \mathcal{I}:=\bigcup_{I \subseteq [m], |I|=d}{\mathcal{I}(I)}.\nonumber
\end{align}
We now set $\mathcal{E}:=\left\{ \Vert A\Vert \le4n^{1/2}  \right\}$. Note that $\mathrm{Incomp}(\delta,\rho) \subseteq \mathcal{I}$ and $\mathbb{P}(\mathcal{E})  \ge 1-e^{-cn}$, then we have:
\begin{eqnarray*}
    & &\mathbb{P}\left( \exists y \in \mathrm{Incomp}(\delta,\rho)\cap E_{t}\right) \le\mathbb{P}\left( \exists y \in \mathrm{Incomp}(\delta,\rho)\cap E_{t}\cap \mathcal{I} \right) \nonumber \\
    & \le&\sum_{I \subseteq [m],|I|=d}{\mathbb{P}\left( \exists y \in \mathcal{I}(I) \cap \mathrm{Incomp}(\delta,\rho)\cap E_{t},\mathcal{E}  \right)} + e^{-\Omega_{T}(n)}\nonumber \\
    & \le& 2^{m}\max_{I \subseteq [m],|I|=d}{\mathbb{P}\left( \exists y \in \mathcal{I}(I) \cap \mathrm{Incomp}(\delta,\rho)\cap E_{t},\mathcal{E}  \right)} +e^{-\Omega_{T}(n)},\nonumber
\end{eqnarray*}
where we use $\binom{m}{d} \le2^{m}$.
In the following, we will take $I=[d]$ as an example for analysis and set
$H_t:=\{x \in \mathrm{Incomp}(\delta,\rho): \Vert (A_{k}+F)x\Vert_{2} \le t  \}$. 
We begin with the following claim.
\begin{mycla}\label{cla6.6}
    For $d:=c_{0}^{2}n$ with $d  \ge 2^{7}T^{4}$ and $n  \ge n_{T}(c_{0}):=2^{6}\kappa^{2}T^{24}L^{2^{12}T^{4}/c_{0}^{2}}$, there exists $L>2$ and $c_{\ref{cla6.6}} \in (0,1)$ depending only on $T$ and $c_{0}$ such that for all $\max\{e^{-c_{\ref{cla6.6}}n}, \sqrt{n}e^{-n/k} \} < t <1$
    \begin{align}
        \mathbb{P}\left( \exists y \in \mathcal{I}([d]) \cap H_{t} : \textbf{g}_{L}(y,k)  \ge \sqrt{n}/t ,\mathcal{E} \right) \le 4^{-m}.\nonumber
    \end{align}
\end{mycla}
\begin{proof}
    Let $L = \max{(8/\kappa_{0},2^{42}C_{\ref{Net size}}^{2}c_{0}^{-12}) } $  and $c_{\ref{cla6.6}} =\min{( c_{\ref{approximate1}},c_{0}^{2}Q_{\ref{Net size}}(T,L,c_{0}) )} $. Recall the definition of $\Sigma_{\varepsilon,k}$ and $\mathcal{N}_{\varepsilon,k}$, applying Proposition \ref{approximate1} with $\varepsilon \in [e^{-c_{\ref{cla6.6}}n},\kappa_{0}/8]$, for each $v \in \Sigma_{\varepsilon,k}$, there exists $u \in \mathcal{N}_{\varepsilon,k}$ such that $\Vert u - v\Vert_{\infty} \le4\varepsilon n^{-1/2}$.
    
     We divide the interval $[t/\sqrt{n},\kappa_{0}/8]$ by setting $\varepsilon_{j}=2^{j}t n^{-1/2}$ for all $j  \ge 0$ such that $\varepsilon_{j} \le \kappa_{0}/8$. Furthermore, we obtain the bound:
    \begin{eqnarray*}
            & &\mathbb{P}\left( \exists y \in \mathcal{I}([d]) \cap H_{t}  : \ \textbf{g}_{L}(y,k)  \ge \sqrt{n}/t , \mathcal{E} \right)\\
            & \le&\sum_{j}{\mathbb{P}\left( \exists v \in \Sigma_{\varepsilon_{j},k}: \Vert (A_{k}+F)v\Vert_{2} \le t ,\mathcal{E} \right)}\\
            & \le&\sum_{j}{\mathbb{P}\left( \exists u \in \mathcal{N}_{\varepsilon_{j},k}: \Vert (A_{k}+F)u\Vert_{2} \le 65\varepsilon_{j} \sqrt{n},\mathcal{E} \right)}\\
            & \le&\sum_{j}{ \left| \mathcal{N}_{\varepsilon_{j},k} \right| \left( 2^{20}L\varepsilon_{j} \right)^{m} },\nonumber
      \end{eqnarray*}
    Recall that Proposition \ref{Net size} with $c_{0} \le c_{0}(T)$ is small enough constant, $n  \ge n_{T}(c_{0})$, $d \ge 2^{7}T^{4}$ and $c_{\ref{cla6.6}} <c_{0}^{2}Q_{\ref{Net size}}(T,L,\kappa_{0},c_{0})$, the last side of the above inequality can be bounded by
    \begin{align}
        \sum_{j } { L^{-2m} \left( \frac{C_{\ref{Net size}}}{c_{0}^{6}\varepsilon_{j}} \right)^{n} \cdot \left( 2^{20}L\varepsilon_{j} \right)^{m}} \le\sum_{j}{L^{-m}\left( \frac{2^{20}C_{\ref{Net size}}}{c_{0}^{6}} \right)^{n} \varepsilon_{j}^{-k}} \le\left( \frac{2^{40}C_{\ref{Net size}}^{2}}{c_{0}^{12}L} \right)^{m},\nonumber
    \end{align}
    where the last inequality applies $\varepsilon_{j}^{-k} \le (\frac{\sqrt{n}}{t})^{k} \le e^{n}$.\par 
    Furthermore, we have
    \begin{align}
        \mathbb{P}\left( \exists y\in \mathcal{I}([d]) \cap H_{t}  : \textbf{g}_{L}(y,k)  \ge \sqrt{n}/t,\mathcal{E} \right) \le 4^{-m}.\nonumber
    \end{align}
    This completes the proof of this claim.
\end{proof}
We define
\begin{align}
    \mathcal{G}:=\left\{ \exists y \in \mathcal{I}([d]) \cap H_{t}: \textbf{g}_{L}(y,k)  \ge \sqrt{n}/t \right\},\nonumber
\end{align}
Then, establishing the following claim.
\begin{mycla}\label{cla6.7}
    There exists $n_{\ref{cla6.7}}$ $C_{\ref{cla6.7}}$, $c_{\ref{cla6.7}}$ and $\gamma$ depending only on $T$ and $c_{0}$ such that the following holds. For any $n \ge n_{\ref{cla6.7}}$, $Q \in [2,C_{\ref{cla6.7}}\sqrt{n}]$, $u \le c_{\ref{cla6.7}}$ and $\frac{u\sqrt{n}}{2Q}e^{-n(\gamma/(2Q^{2}))} \le t \le c_{\ref{cla6.7}} $
    \begin{align*}
        \mathbb{P}\left( \exists y \in \mathcal{I}([d]) \cap H_{t} : \mathrm{RlogD}_{Q,u}^{a}(y) \le \sqrt{n}/t,\mathcal{G}^{c},\mathcal{E} \right)\le4^{-n},
    \end{align*}
    where $a \sim G_{n}(T)$.
\end{mycla}
\begin{proof}
    Let $u \le \min{(u_{0},u_{1}/4)}$ and $n  \ge n_{0}(T)$. 
    Recall $Q=q\sqrt{k} >c\sqrt{k}$ and 
    \begin{align}
        S_{D,Q,u}^{a}:=\left\{ x \in \mathrm{Incomp}(\delta,\rho): \mathrm{RlogD}_{Q,u}^{a}(x) \in [D,2D]  \right\}.\nonumber
    \end{align}
    For the lower bound of log-RLCD, using Lemma \ref{Lower bound of the logRLCD}, we have $\mathrm{RlogD}_{Q,u}^{a}(x)  \ge h \sqrt{n}$ for all $x \in \mathrm{Incomp}(\delta,\rho)$. Separate $G:=[h\sqrt{n},\sqrt{n}/t]$ by setting $D_{j}=2^{j}h\sqrt{n}$ for all $j  \ge 0$ such that $D_{j} \le \sqrt{n}/t$. Using Lemma \ref{Net size 2}, let $\gamma = 4q^{2}b$, then for each $\varepsilon_{j}:=(2D_{j})^{-1}$, 
    $$\varepsilon_{j}^{-1} \le \frac{\sqrt{n}}{t} \le \frac{Q}{2u}e^{n(\gamma/(2Q)^{2})}.$$
    At the same time, for $n  \ge 2^{7}T^{4}\delta^{-1}\gamma^{-2}$ and $a \sim G_{n}(T)$, 
    $$\mathbb{E}a^{2} = n \le \frac{1}{8}p^{2}\delta \gamma^{2}n^{2}.$$
    We apply Lemma \ref{Net size 2} for $\varepsilon_{j} =(2D_{j})^{-1}$ with $W$ being chosen uniformly at random from $\Xi_{\varepsilon_{j}}$, 
    \begin{align}\label{6.2}
         \mathbb{P}_{W}\left( \mathrm{RlogD}_{2Q,4u}^{a}(W) \le 2D_{j}  \right) \le\left( C_{T}\gamma \right)^{c_{T}n},       
    \end{align}
    where $C_{T}$ and $c_{T}$ are positive constants depending only on $T$. Note that for all $y \in S_{D_{j},Q,u}^{a}$, that there exists $x \in \Xi_{\varepsilon_{j}}$ with $\Vert x-y\Vert_{\infty} \le \varepsilon_{j} n^{-1/2}$ such that
\begin{align}
    \mathrm{RlogD}_{2Q,4u}^{a}(x) \le 2D_{j}.\nonumber
\end{align}
We define $\mathcal{M}_{j}$ for $j  \ge 0$ as follows
\begin{align}
    \mathcal{M}_{j} := \Xi_{\varepsilon_{j}} \cap \left\{ \mathrm{RlogD}_{2Q,u}^{a}(y) \le 2D_{j} \right\} \cap T_{j},
\end{align}
where we set $T_{j}:=\left\{y : \exists x \in \mathcal{I}([d]) \cap \mathrm{Incomp}(\delta,\rho)\cap S_{D_{j},Q,u}^{a} , \Vert x-y \Vert_{\infty} \le\frac{\varepsilon_{j}}{\sqrt{n}} \right\}$.

Using \eqref{6.2} to obtain 
\begin{align}
    \left| \mathcal{M}_{j} \right| \le\left( C_{T}\gamma^{c_{T}}D_{j} \right)^{n}.\nonumber
\end{align}
Combining the above conclusions, we can provide the following estimates.
\begin{eqnarray*}
        & &\mathbb{P}\left( \exists y \in \mathcal{I}([d]) \cap \mathrm{Incomp}(\delta,\rho)\cap H_{t}: \mathrm{RlogD}_{Q,u}^{a}(y) \le \sqrt{n}/t,\mathcal{G}^{c},\mathcal{E} \right)\\
        & \le&\sum_{j \ge 0}{\mathbb{P}\left( \exists y \in S_{D_{j},Q,u}^{a} \cap \mathcal{I}([d]) \cap H_{t}, \mathcal{G}^{c},\mathcal{E}   \right)}\\
        & \le&\sum_{j \ge 0}{\mathbb{P}\left( \exists y \in \mathcal{M}_{j}:\Vert (A_{k}+F)y\Vert_{2} \le 16\varepsilon_{j} \sqrt{n},\mathcal{G}^{c},\mathcal{E} \right)}\\
        & \le&\sum_{j \ge 0}\sum_{y \in \mathcal{M}_{j}}{\mathbb{P}\left( \Vert (A_{k}+F)y \Vert_{2} \le16 \varepsilon_{j} \sqrt{n},\mathcal{G}^{c},\mathcal{E} \right)}.\nonumber
  \end{eqnarray*}
Recall the definition of $\mathcal{M}_{j}$ and for $b < \omega$, we have 
\begin{eqnarray*}
    & &\sum_{y \in \mathcal{M}_{j}}{\mathbb{P}\left( \Vert (A_{k}+F)y \Vert_{2} \le 16\varepsilon_{j}\sqrt{n},\mathcal{G}^{c},\mathcal{E} \right)}\nonumber\\ 
    & \le&\left| \mathcal{M}_{j} \right|\max_{x \in \mathcal{I}([d]) \cap \{ x:\textbf{g}_{L}(x,k) \ge \sqrt{n}/t \}  }{\mathbb{P}\left( \Vert (A_{k}+F)x-w\Vert_{2} \le 32 \varepsilon_{j} \sqrt{n},\mathcal{E} \right)} \nonumber\\
    & \le&\left( C_{T}L \right)^{n}D_{j}^{k}\gamma^{c_{T}n} \le(C_{T}L\gamma^{c_{T}})^{n},\nonumber
\end{eqnarray*}
where we use Lemma \ref{Lem5.12} in the last inequality. Thus
\begin{align}
    \mathbb{P}\left( \exists y \in \mathcal{I}([d]) \cap \mathrm{Incomp}(\delta,\rho)\cap E_{t} : \mathrm{RlogD}_{Q,u}^{a}(y) \le \sqrt{n}/t,\mathcal{G}^{c},\mathcal{E} \right) \le \left( O_{T}(L\gamma^{c})\right)^{n}.\nonumber
\end{align}
We finally complete the proof of this claim.
\end{proof}
We have made all the preparations, and now we can prove this theorem. Combining Claim \ref{cla6.6} and Claim \ref{cla6.7}, we obtain
\begin{align}
     \mathbb{P}\left( \exists y \in \mathcal{I}([d])\cap \mathrm{Incomp}(\delta,\rho)\cap E_{t}, \mathcal{E} \right) \le4^{-m}+4^{-n} \le4^{-3m/4}.\nonumber
\end{align}
This implies that 
\begin{align}
    \mathbb{P}\left( \exists y \in S^{n-1}: \mathrm{RlogD}_{Q,u}^{a}(y) \le \sqrt{n}/t \text{ and } \Vert (A_{k}+F)y\Vert_{2} \le t \right) \le e^{-cn}.\nonumber
\end{align}
\end{proof}
\subsection{Proof of the distance theorem}\label{Proof of the distance theorem}
Using the proposition established in the previous subsection, we can complete the proof of Theorem \ref{Distance}.
\begin{proof}[\textsf{Proof of the Theorem \ref{Distance}}]
    Recall the definition of $H$ in Theorem \ref{Distance} and $k \le \lambda_{\ref{Distance}}n/\log n$, applying Theorem \ref{Large RlogD with small distance} with $Q:= c_{\ref{Lem6.1}}\sqrt{k}$ and $t= c\sqrt{n}e^{-cn/k}/\sqrt{k}$, we have
    \begin{align}
        \mathbb{P}(\exists x \in S^{n-1}\cap H^{\perp}: \mathrm{RlogD}_{Q,c_{\ref{Large RlogD with small distance}}}^{a}(x) \le c\sqrt{k}e^{cn/k}  ) \le e^{-cn}.\nonumber
    \end{align}
    We define the $\mathcal{H}$ as follows 
    \begin{align}
        \mathcal{H}:=\{ E:  \mathrm{RlogD}_{Q,c_{\ref{Large RlogD with small distance}}}^{a}(x) \ge  c\sqrt{k}e^{cn/k} \text{ for all }x \in E \cap S^{n-1}   \}.\nonumber
    \end{align}
    Furthermore, using Lemma \ref{Lem6.1} for $E$ with $E^{\perp} \in \mathcal{H}$ and $v \in \mathbb{R}^{n}$   
    \begin{align}
        \mathbb{P}_{a}\left( \mathrm{dist}(a,E+v) \le \varepsilon \sqrt{k}  \right) \le (C\varepsilon)^{k}+e^{-cn}.\nonumber
    \end{align}
\end{proof}

\section{Tail bound for the gaps of the eigenvalue}\label{Section Tail bound for the gaps of the eigenvalue}
In this section, we will estimate the tail bound for the gaps of the eigenvalue by two approaches.
\subsection{Proof of the Theorem \ref{Theorem A}}\label{Proof of the Theorem A}
Before starting the proof, we first state the following auxiliary lemma, which can be derived from Lemmas \ref{Comp} and \ref{Spread incomp}; we omit the detailed proof.
\begin{mylem}\label{coordinates delocalization for eigenvectors}
    Let $A_{n} \sim \mathrm{Sym}_{n}(T)$ be an $n \times n$ symmetric random matrix. Then there exists $c_{\ref{coordinates delocalization for eigenvectors}}$ depending only on $T$ such that the following holds. With probability at least $1-e^{-c_{\ref{coordinates delocalization for eigenvectors}}n}$, for any eigenvectors of $A$, there exist $c_{\ref{coordinates delocalization for eigenvectors}}n$ coordinates with absolute value at least $c_{\ref{coordinates delocalization for eigenvectors}}n^{-1/2}$.
\end{mylem}

\begin{proof}[\textsf{Proof of the Theorem \ref{Theorem A}}]
    Let $A_{n} \sim \mathrm{Sym}_{n}(T)$ and denote the eigenvalues of $A_{n}$ by $\lambda_{1}(A_{n}) \le \dots \le \lambda_{n}(A_{n})$.

    Let $\mathcal{E}_{n,k}:= \{ \lambda_{i+k}(A_{n})-\lambda_{i}(A_{n}) \le \varepsilon n^{-1/2}  \}$ and $\mathcal{E}_{n}:=\{ \Vert A_{n}\Vert \le 4\sqrt{n} \}$.
    Applying Lemma \ref{coordinates delocalization for eigenvectors}, with probability at least $1- e^{-cn}$, for the unit eigenvectors $u$ corresponding to $\lambda_{i}(A_{n})$, there exists $J \subset [n]$ with $|J|=c_{\ref{coordinates delocalization for eigenvectors}}n$ and for all $i \in J$, $|u_{i}| \ge c_{\ref{coordinates delocalization for eigenvectors}}n^{-1/2}$. We randomly choose $i$ from $J$ and, without loss of generality, assume that $i = n$.

    Assume that 
    \begin{align}
        A_{n}:= \begin{bmatrix}
        A_{n-1}       & X\\
        X^{T}        & a_{nn}
    \end{bmatrix} \text{ and } u:= \begin{bmatrix}
        w\\
        b
    \end{bmatrix}.\nonumber
    \end{align}
    By our assumption, we have $|b| \ge c_{\ref{coordinates delocalization for eigenvectors}}n^{-1/2} $ and $A_{n}u=\lambda_{i}(A_{n})u$. Furthermore, we have
    \begin{align}
        (A_{n-1}-\lambda_{i}(A_{n}))w+bX=0.\nonumber
    \end{align}
    Consider the unit eigenvectors $v_{1},\dots,v_{k}$ corresponding to $\lambda_{i}(A_{n-1}) \le \dots \le \lambda_{i+k-1}(A_{n-1})$, for all $j \in [k]$, we have 
    \begin{align}
        v_{j}^{T}(A_{n-1}-\lambda_{i}(A_{n}))w=-bv_{j}^{T}X.\nonumber
    \end{align}
    Furthermore, we take the absolute value for the above equation to obtain the following. 
    \begin{align}
        |b||v_{j}^{T}X| \le |\lambda_{i+j-1}(A_{n-1})-\lambda_{i}(A_{n})||v_{j}^{T}w| \le |\lambda_{i+j-1}(A_{n-1})-\lambda_{i}(A_{n})|.\nonumber
    \end{align}
    In $\mathcal{E}_{n,k}$, using the Cauchy interlacing law, and in $|b|\ge c_{\ref{coordinates delocalization for eigenvectors}}n^{-1/2}$, we have the following. 
    \begin{align}
        \mathbb{P}(\mathcal{E}_{n,k}) \le \frac{n}{c_{\ref{coordinates delocalization for eigenvectors}}n}\mathbb{P}( |v_{j}^{T}X| \le c_{\ref{coordinates delocalization for eigenvectors}}^{-1}\varepsilon \text{ for all } j \in [k] ,\mathcal{E}_{n-1,k-1}) + e^{-c_{\ref{coordinates delocalization for eigenvectors}}n}.\nonumber
    \end{align}
    Let $H_{(n-1,k)}:= \mathrm{Span}(v_{1},\dots,v_{k})$, it is sufficient to estimate 
    \begin{align}
        \mathbb{P}(\mathrm{dist}(X,H_{(n-1,k)}^{\perp}) \le \varepsilon\sqrt{k},\mathcal{E}_{n-1,k-1}).\nonumber
    \end{align}
    Recalling that $\mathcal{E}_{n-1,k-1}:=\{ \lambda_{i+k-1}(A_{n-1})-\lambda_{i}(A_{n-1}) \le \varepsilon n^{-1/2} \}$, we first make the following claim.
    \begin{mycla}\label{Comb probability of gap and rlogd}
        There exist $c_{\ref{Comb probability of gap and rlogd}}$ and $C_{\ref{Comb probability of gap and rlogd}}$ depending only on $T$ such that for $n \ge C_{\ref{Comb probability of gap and rlogd}}$, $Q = c_{\ref{Lem6.1}}\sqrt{k} $ and $c_{\ref{Large RlogD with small distance}}\sqrt{n} \ge  \varepsilon \ge \max{ \{ ne^{-n},\frac{c_{\ref{Large RlogD with small distance}}n}{Q}e^{-c_{\ref{Large RlogD with small distance}}n/Q^{2}} \}  } $, we have 
        \begin{align}
            \mathbb{P}(\mathcal{E}_{n-1,k-1}\cap \{  \mathrm{RlogD}_{Q,\alpha}^{X}(H_{(n-1,k)}) \le n/\varepsilon  \}) \le e^{-c_{\ref{Comb probability of gap and rlogd}}n}.\nonumber
        \end{align}
    \end{mycla}
    \begin{proof}
        For $x \in H_{(n-1,k)} \cap S^{(n-1)-1}$, assume that 
        \begin{align}
            x:= a_{1}v_{1} + \dots+a_{k}v_{k}, \text{ for some } a_{1}^{2}+\dots+ a_{k}^{2}=1.\nonumber
        \end{align}
        On $\mathcal{E}_{n-1,k-1}$, we have 
        \begin{align}
            \Vert (A_{n-1}-\lambda_{i}(A_{n-1}))x\Vert_{2}^{2} \le \sum_{j=1}^{k}a_{j}^{2}(\lambda_{j+i-1}(A_{n-1})-\lambda_{i}(A_{n-1}))^{2} \le \varepsilon^{2}n^{-1}.\nonumber
        \end{align}
        Furthermore, on $\mathcal{E}_{n-1}$, we have 
        \begin{align}
            \lambda_{i}(A_{n-1})\in [-4\sqrt{n},4\sqrt{n}].\nonumber
        \end{align}
        Then there exists $\lambda_{0} \in \frac{\varepsilon}{\sqrt{n}}\mathbb{Z}\cap [-4\sqrt{n},4\sqrt{n}]$ satisfying for all $x \in H_{(n-1,k)}\cap S^{(n-1)-1}$: 
        \begin{align}
            \Vert (A_{n-1}-\lambda_{0})x\Vert_{2} \le 2\varepsilon n^{-1/2}.\nonumber
        \end{align}
        Applying Theorem \ref{Large RlogD with small distance} with $t = 2\varepsilon n^{-1/2}$, $A_{n-1}$ and $F=-\lambda_{0}$ to obtain
      \begin{eqnarray*}
            & & \mathbb{P}\left( \mathcal{E}_{n-1,k-1} , \mathrm{RlogD}_{Q,\alpha}^{X}(H_{(n-1,k)}) \le n/\varepsilon \right) \nonumber \\
            &\le& \frac{\sqrt{n}}{\varepsilon}\mathbb{P}\left( \exists x \in S^{(n-1)-1} :\Vert (A_{n-1}-\lambda_{0})x\Vert_{2} \le t, \mathrm{RlogD}_{Q,\alpha}^{X}(x) \le \sqrt{n}/t  \right) +e^{-c_{\ref{Spectral norm estimate}}n}\nonumber\\
            &\le& e^{-cn},\nonumber
          \end{eqnarray*}
        which implies our result.
        \end{proof}
        Taking $\mathcal{H}_{n-1,k}:=\{  \mathrm{RlogD}_{Q,\alpha}^{X}(H_{(n-1,k)}) \ge n/\varepsilon \}$, we have 
       \begin{eqnarray*}
& & \mathbb{P}(\mathrm{dist}(X,H_{(n-1,k)}^{\perp}) \le \varepsilon\sqrt{k},\mathcal{E}_{n-1,k-1} )\nonumber \\
            &\le& \mathbb{P}(\mathrm{dist}(X,H_{(n-1,k)}^{\perp}) \le \varepsilon\sqrt{k},\mathcal{E}_{n-1,k-1}\cap\mathcal{H}_{n-1,k} )+e^{-cn}\nonumber\\
            &\le&\mathbb{P}(\mathrm{dist}(X,H_{(n-1,k)}^{\perp}) \le \varepsilon \sqrt{k}|\mathcal{E}_{n-1,k-1}\cap\mathcal{H}_{n-1,k} )\mathbb{P}(\mathcal{E}_{n-1,k-1}\cap \mathcal{H}_{n-1,k})+e^{-cn}\nonumber\\
            &\le& \left(C\varepsilon \right)^{k}\mathbb{P}\left( \mathcal{E}_{n-1,k-1} \right)+e^{-cn}.
       \end{eqnarray*}

        Furthermore, we conduct the same analysis on $\mathcal{E}_{n-1,k-1},\dots,\mathcal{E}_{n-k+1,1}$.
        
        For all $\varepsilon \ge\max{ \{ ne^{-n},\frac{c_{\ref{Large RlogD with small distance}}n}{Q}e^{-c_{\ref{Large RlogD with small distance}}n/Q^{2}} \}  }  $, we have
        \begin{align}
            \mathbb{P}(\mathcal{E}_{n,k}) \le \prod_{j=1}^{k}\left(C\varepsilon \right)^{j} +e^{-cn} \le \left(C\varepsilon \right)^{\frac{k^{2}+k}{2}}+e^{-cn}.\nonumber
        \end{align}
        Note that if $\varepsilon \le \max{ \{ ne^{-n},\frac{c_{\ref{Large RlogD with small distance}}n}{Q}e^{-c_{\ref{Large RlogD with small distance}}n/Q^{2}} \}  } $, we have 
        \begin{align}
            (C\varepsilon)^{k^{2}+k} \le e^{-cn},\nonumber
        \end{align}
        for $k \le c_{\ref{Theorem A}}n/\log n$. Now, we complete the proof of our result.
\end{proof}
\subsection{Proof of the Theorem \ref{Gaps}}\label{Proof of the Theorem Gaps}
Before  proving the theorem, we first introduce a simple lemma to characterize singular values, and then complete the proof of Theorem \ref{Gaps}.
\begin{mylem}
    [Theorem 6 in \cite{Naor_JTDM}]\label{Lem7.1} Assume that M is a full-rank matrix of size $k \times d$ with $k \le d$. Then for $1 \le l \le k-1$, there exists $l$ different indices $i_{1},\dots,i_{l}$ such that the matrix $M_{i_{1},\dots,i_{l}}$ with columns $\mathrm{Col}_{i_{1}}\left( M \right),\dots,\mathrm{Col}_{i_{l}}\left( M \right)$ has the smallest non-zero singular value $s_{l}\left( M_{i_{1},\dots,i_{l}} \right)$ satisfying
    \begin{align}
        s_{l}\left(M_{i_{1},\dots,i_{l}} \right)^{-1} \le C_{\ref{Lem7.1}}\min_{r \in \left\{ l+1, \dots , k \right\}}{\sqrt{\frac{dr}{\left( r-l \right) \sum_{i=r}^{k}{s_{i}\left( M \right)^{2}}}}}.\nonumber
    \end{align}
    where $C_{\ref{Lem7.1}}$ is an absolute constant.
\end{mylem}
\begin{proof}[\textsf{Proof of Theorem \ref{Gaps}}]
    For $z \in [-8\sqrt{n},8\sqrt{n}]$, $1 \le k \le\lambda n /\log{n}$ and $A \sim \mathrm{Sym}_{n}(T)$, we first study the distribution of the eigenvalues of $A$. If there exists some $i$ such that 
    \begin{align}
        z-\frac{\varepsilon}{\sqrt{n}} \le\lambda_{i} \le\lambda_{i+k-1} \le z+\frac{\varepsilon}{\sqrt{n}}.\nonumber
    \end{align}
    We can get $k$ orthogonal unit vectors(eigenvectors of $A+F$) such that 
    \begin{align}
        \Vert (A+F-zI_{n})w_{i} \Vert_{2} \le\frac{\varepsilon}{\sqrt{n}} , \ 1 \le i \le k.
    \end{align}
    Assume that $W^{\mathrm{T}}:=\left( w_{1},\dots,w_{k} \right)$ is a $n \times k$ full-rank matrix and apply Lemma \ref{Lem7.1} for $W$, for all $1 \le l \le k-1$, there exists $i_{1},\dots,i_{l} \in [n]$ such that 
    \begin{align}\label{7.2}
        s_{l}\left( W_{i_{1},\dots,i_{l}} \right)^{-1} \le C\min_{r \in \left\{ l+1,\dots,k \right\}  }{\sqrt{\frac{rn}{(r-l)\sum_{i=r}^{k}{s_{i}(W)^{2}} } } } \le C\sqrt{\frac{kn}{(k-l)^{2}}}.
    \end{align}
    Let $F_{1}:=\left( w_{i_{1}},\dots,w_{i_{l}} \right)^{\mathrm{T}}$, $F_{2}:=\left(w_{i_{l+1}},\dots, w_{i_{n}}\right)^{\mathrm{T}}$, $B:=A+F-zI_{n}$ and $B_{i}$ be the $i$-th column of $B$.
    
    Note that
    \begin{align}
        U:=BW^{T} = \left( B_{i_{1}},\dots,B_{i_{l}} \right)F_{1}+\left( B_{i_{l+1}},\dots,B_{i_{n}} \right)F_{2}:=B_{1}F_{1}+B_{2}F_{2}.\nonumber
    \end{align}
    Furthermore, define the right inverse of $F_{1}$ as $\overline{F}:=F_{1}^{\mathrm{T}}\left( F_{1}F_{1}^{\mathrm{T}} \right)^{-1}$ and we have 
    \begin{align}
        U\overline{F}=B_{1}+B_{2}F_{2}\overline{F}.\nonumber
    \end{align}
    Rewrite \eqref{7.2} to $s_{l}\left( F_{1}^{\mathrm{T}} \right)^{-1} \le C\sqrt{\frac{kn}{(k-l)^{2}}}$ and assume that $P$ is an orthogonal projection of $H^{\perp}:=\ker{B_{2}^{\mathrm{T}}}$. Thus, we obtain
    \begin{eqnarray*}
        & &\sum_{j=1}^{l}{\mathrm{dist}^{2}\left( B_{i_{j}},H \right)} =\Vert PB_{1}\Vert_{2}^{2} =\Vert PU\overline{F}\Vert_{2}^{2} \le\Vert \overline{F}\Vert^{2}\Vert U \Vert_{\mathrm{HS}}^{2}\nonumber \\
        & \le& s_{l}(F_{1}^{\mathrm{T}})^{-2}\sum_{j=1}^{k}{\Vert (A+F-zI_{n})w_{j}\Vert_{2}^{2}} \le\frac{Ck^{2}\varepsilon^{2}}{(k-l)^{2}}.\nonumber 
 \end{eqnarray*}
    It means that there exists $v_{i_{1}},\dots,v_{i_{l}}$ such that
    \begin{align}
        \sum_{j=1}^{l}{\mathrm{dist}^{2}(A_{i_{j}},H+v_{i_{j}})} \le\frac{Ck^{2}\varepsilon^{2}}{(k-l)^{2}}.\nonumber
    \end{align}
    Let $l = u k$ for $u \in (0,1)$ and we have 
    \begin{eqnarray*}
        & &\mathbb{P}\left( \exists i:z-\frac{\varepsilon}{\sqrt{n}} \le \lambda_{i} \le\lambda_{i+k-1} \le z+\frac{\varepsilon}{\sqrt{n}} \right)\\
        &\le&\mathbb{P}\left( \exists i_{1},\dots,i_{l} \in [n]: \sum_{j=1}^{l}{\mathrm{dist}^{2}\left( A_{i_{j}},H+v_{i_{j}} \right)} \le C_{u}^{2}\varepsilon^{2} \right)\\
        &\le&\binom{n}{l}\max_{i_{1},\dots,i_{l} \in [n]}{\mathbb{P}\left(  \sum_{j=1}^{l}{\mathrm{dist}^{2}\left( A_{i_{j}},H+v_{i_{j}} \right)} \le C_{u}^{2}\varepsilon^{2} \right)} \\
        &\le& \left( \frac{Cn}{l} \right)^{l} \max_{i_{1},\dots,i_{l} \in [n]}{\mathbb{P}\left( \mathcal{E}_{u,\gamma} \right)},\nonumber
   \end{eqnarray*}
    where $C_{u}=C/(1-u)$ and we define the event
    \begin{align}\label{9def}
    \mathcal{E}_{u,\gamma}:=\left\{\exists n_{1},\dots,n_{m} \in \left\{ i_{1},\dots,i_{l} \right\}:  \mathrm{dist}(A_{n_{i}},H+v_{n_{i}})  \le\frac{C_{u}\varepsilon}{\sqrt{(1-\gamma)l}} \right\},
    \end{align}
    for any $\gamma \in (0,1)$ and $m:=\gamma l$. 
    
    Without loss of generality, we  assume that $n_{j}=j$, and $H$ is the subspace spanned by last $n-k$ columns of $A+F-zI_{n}$(In other cases, similar methods can be used to obtain an upper bound).\\
    Define $A_{i,I}$ and $H_{I}$ as the projections of $A_{i}$ and $H$ onto the coordinates indexed by $I$ respectively. We have 
  \begin{eqnarray*}
        & &\mathbb{P}\left( \forall i \in [m]: \mathrm{dist}\left( A_{i},H+v_{i} \right) \le\frac{C_{u}\varepsilon}{\sqrt{(1-\gamma)l}} \right)\nonumber\\
        &\le&\mathbb{P}\left(\forall i \in [m]: \mathrm{dist}\left( A_{i,\left\{ i+1,\dots,n \right\}},H_{ \left\{ i+1,\dots,n \right\} }+v_{i, \left\{ i+1,\dots,n \right\} }  \right) \le\frac{C_{u}\varepsilon}{\sqrt{(1-\gamma)l}} \right).\nonumber
    \end{eqnarray*}
    
    Define ``distance events" as
    \begin{align}
        \mathcal{F}_{i}:=\left\{  \mathrm{dist}\left( A_{i,\left\{ i+1,\dots,n \right\}},H_{ \left\{ i+1,\dots,n \right\} }+v_{i, \left\{ i+1,\dots,n \right\} }  \right) \le\frac{C_{u}\varepsilon}{\sqrt{(1-\gamma)l}} \right\}.
    \end{align}
     
     Let $A'_{i}:=A_{i,\left\{ i+1,\dots,n \right\} }$, $H'_{i}:=H_{\left\{ i+1,\dots,n \right\} }$ and define ``the large log-RLCD events" as
    \begin{align}
        \mathcal{E}_{i}:=\left\{H_{i}': \mathrm{RlogD}_{Q,u}^{A'_{i}}\left( H'_{i} \right)  \ge e^{bn/k}  \right\}.
    \end{align}
    Assume that $\mathcal{V}_{i}:=\mathcal{F}_{i} \cap \mathcal{E}_{i-1}$ and applying Theorem \ref{Distance} with all $A'_{i}$ and $H'_{i}$, we have 
    \begin{equation}
    \begin{aligned}
        \mathbb{P}\left( \bigcap_{i=1}^{m}{\mathcal{F}_{i}} \right)
        & \le\mathbb{P}\left( \bigcap_{i=1}^{m}{\mathcal{V}_{i}} \right) + e^{-cn}\\
        & = \mathbb{P}\left( \mathcal{V}_{m} \right)\mathbb{P}\left( \mathcal{V}_{m-1}|\mathcal{V}_{m} \right) \cdots \mathbb{P}\left( \mathcal{V}_{1} | \bigcap_{i=2}^{m}{\mathcal{V}_{i}}\right)+e^{-cn}. \nonumber
    \end{aligned}
    \end{equation}
 In fact, 
    \begin{align}
        \mathbb{P}\left( \mathcal{V}_{i}|\bigcap_{j=i+1}^{m}{\mathcal{V}_{j}} \right) \le\mathbb{P}\left( \mathcal{F}_{i} | \mathcal{E}_{i} \right) \le\left( \frac{C_{u}\varepsilon}{\sqrt{(1-\gamma)l(k-i)}} \right)^{k-i} +e^{-cn}.
    \end{align}
    Thus, we have 
    \begin{align}
        \mathbb{P}\left( \mathrm{dist}\left( A_{i},H+v_{i} \right) \right) \le\left( \frac{C_{u,\gamma}\varepsilon}{k} \right)^{f_{u,\gamma}(k)}+e^{-cn},\nonumber
    \end{align}
    where  $f_{u,\gamma}(k)=\sum_{i=1}^{m}{(k-i)}=(1-x^{2})k^{2}/2-(1-x))k/2:=f_{x}(k)$, $x= 1- u \gamma \in (0,1)$,  $C_{u,\gamma}=C(\sqrt{(1-\gamma)u}(1-u))^{-1}$.
    Combining all the above estimates, we ultimately arrive 
    \begin{align}\label{eigengen}
        \mathbb{P}\left( \exists i:z-\frac{\varepsilon}{\sqrt{n}} \le\lambda_{i} \le\lambda_{i+k-1} \le z+\frac{\varepsilon}{\sqrt{n}} \right) \le \left( \frac{C_{x}\varepsilon}{k} \right)^{f_{x}(k)}+e^{-c_{1}n},
    \end{align}
    for all $\gamma_{\ref{Gaps}}\log{n} \le k \le\gamma_{\ref{Gaps}}'n/\log{n}$, where $f_{x}(k):=\frac{(1-x^{2})k^{2}-(1-x)k}{2}$, $C_{x}$ depending only on $T$ and $x$ and  $\gamma_{\ref{Gaps}}$, $\gamma_{\ref{Gaps}}'$ and $c_{1}>0$ depending only on $T$.
    
    Observe that
    \begin{align*}
        \mathbb{P}\left( \Vert A\Vert  \ge 4\sqrt{n}\right) \le e^{-cn},
    \end{align*}
    Thus, we have the eigenvalues of $A+F$ belong to $[-8\sqrt{n},8\sqrt{n}]$ with probability at least $1-e^{-cn}$.
        
    Set $z_{j}:=j\varepsilon/\sqrt{n}$ and $I:=[-8\sqrt{n},8\sqrt{n}]\cap \frac{\varepsilon}{\sqrt{n}}\mathbb{Z}$, we have
        \begin{eqnarray*}
            & &\mathbb{P}\left( \exists 1\le i \le n-k+1: \lambda_{i+k-1}-\lambda_{i} \le\frac{\varepsilon}{\sqrt{n}}   \right)\nonumber\\
            &\le&\mathbb{P}\left( \exists 1\le i \le n-k+1: \exists j \in I, \lambda_{i}, \lambda_{i+k-1} \in [z_{j}-\frac{\varepsilon}{\sqrt{n}},z_{j}+\frac{\varepsilon}{\sqrt{n}}] \right) + e^{-cn}\nonumber\\
            &\le&\frac{16n}{\varepsilon}\left( \left( \frac{C\varepsilon}{k} \right)^{f_{x}(k)} + e^{-cn} \right)+e^{-cn} \le\left( \frac{C\varepsilon}{k} \right)^{g_{\ref{Gaps}}(k)}+e^{-cn},\nonumber
       \end{eqnarray*}
        where the last inequality uses the \eqref{eigengen}.
\end{proof}

\section{Quantitative estimate for the singular values}\label{Section Quantitative estimate for the singular values}
The main goal of this section is to present the proof of Theorem \ref{The_main}. Inspired by the proofs of Theorem \ref{Gaps} and Theorem \ref{Distance}, our aim is to show that, for a linear space $H$ akin to the assumptions of Theorem \ref{Distance}, the probability of the existence of a subspace of its orthogonal space with a large log-RLCD is at least $1 - e^{-ckn}$. In fact, for the asymmetric case, we have already provided the corresponding results in \cite{Rank}. Drawing inspiration from the methods in \cite{HanY2}, we also complete the proof in this section. We will divide the process into four major parts and begin with the first part.
\subsection{Almost orthogonal vectors}\label{Almost orthogonal vectors}
In this subsection, we introduce the main space partitioning methods. These are based on the key approaches of \cite{Rudelson_aop}. We present a generalized version given in \cite{HanY1}. We first give the ``almost orthogonal vectors" as follows.
\begin{mydef}\label{def9.1.1}
Let $\nu \in \left( 0,1 \right)$. An l-tuple of vectors $\left( v_{1},v_{2},\dots,v_{l} \right) \subset \mathbb{R}^{n} \setminus \left\{ 0 \right\}$ is called $\nu$-almost orthogonal if the $n \times l$ matrix $V_{0}$ with $\mathrm{Col}_{j}\left( V_{0} \right)= v_{j}/\Vert v_{j} \Vert_{2} $ satisfies:
$$1-\nu \le s_{l}\left( V_{0} \right) \le s_{1}\left( V_{0} \right) \le 1+\nu.$$
\end{mydef}
We next show how to divide the linear subspace $E$ into the $E\cap W$(where W is a closed set and $W \subset \mathbb{R}^{n} \setminus \left\{0 \right\}$) and the linear subspace $F \subset E$ with high dimension.
\begin{mylem} 
   [Lemma 2.6 in \cite{HanY1}]\label{lem9.1.2}Let $l<k \le n$, and let $E \subset \mathbb{R}^{n}$ be a linear subspace of dimension $k$. Fix an interval $I \subset [n]$ with $|I| \ge k$ and let $W_{I} \subset \mathbb{R}^{n} \setminus \left\{ 0 \right\}$ be the closed set such that $\{ v_{I}:v \in W_{I} \} \subset \mathbb{R}^{I} \setminus \{0 \}$ is also a closed set. Then, at least one of the following holds:
\begin{enumerate}
    \item There exist vectors $v_{1},\dots,v_{l} \in E\cap W_{I}$ such that
\begin{itemize}
    \item The l-tuple $\left((v_{1})_{I},\dots,(v_{l})_{I}\right)$is $\left(\frac{1}{8}\right)$-almost orthogonal in $\mathbb{R}^{I}$;
    \item For any $\theta \in \mathbb{R}^{l}$ such that $\Vert \theta \Vert_{2} \le \frac{1}{20\sqrt{l}}$,
$$\sum_{i=1}^{l}\theta_{i}v_{i}\notin W_{I};$$
\end{itemize}
\item There exists a subspace $F \subset E$ of dimension $k-l$ such that $F \cap W_{I}=\emptyset$.
\end{enumerate}

\end{mylem}
\subsection{Small ball probability via RD}\label{Small ball probability via RD}
In this section, we will introduce the Randomized least common denominator and compare it with our log-RLCD.
\begin{mydef}[Definition 2.5 in \cite{Rank}]\label{def9.2.1}
Let $V$ be an $m \times n$ (deterministic) matrix, $\xi=\left(\xi_{1},\dots,\xi_{n}\right) \sim G_{n}(T)$ and let $L>0$, $\alpha \in\left(0,1\right)$. Define the Randomized log-least common denominator(RLCD) of $V$ and $\xi$ by
$$RD_{L,\alpha}^{\xi}\left(V\right) = \inf{\left\{\Vert \theta\Vert_{2}:\theta \in \mathbb{R}^{m},\Vert V^{T}\theta \Vert_{\xi} < L\sqrt{\log_{+}{\frac{\alpha\Vert V^{T}\theta\Vert_{2}}{L}}} \right\}}.$$
If $E \subset \mathbb{R}^{n}$ is a linear subspace, we can adapt this definition to the orthogonal projection $P_{E}$ on $E$ setting
$$RD_{L,\alpha}^{\xi}\left(E\right)= \inf{\left\{\Vert y\Vert_{2}:y \in E,\Vert y \Vert_{\xi} < L\sqrt{ \log_{+}{\frac{\alpha\Vert y\Vert_{2}}{L}} }\right\}}.$$
\end{mydef}
We  give the following result to estimate the small ball probability via RD.
\begin{mypropo}[Proposition 2.6 in \cite{Rank}]\label{pro9.2.2}
 Consider a real-valued random vector $\xi = \left(\xi_{1},\dots,\xi_{n}\right) \sim G_{n}(T)$ and $V\in \mathbb{R}^{m\times n}$. Then there exists a universal constant $c_{\ref{pro9.2.2}}>0$, for any $L\ge c_{\ref{pro9.2.2}}\sqrt{m}$, we have 
\begin{align}
\mathcal{L}\left( V\xi,t\sqrt{m} \right) \le \frac{\left(C_{\ref{pro9.2.2}}L/\left(\alpha\sqrt{m}\right)\right)^{m}}{\mathrm{det}\left(VV^{T}\right)^{1/2}}\left( t+\frac{\sqrt{m}}{RD_{L,\alpha}^{\xi}\left( V\right)}\right)^{m},\quad t \ge 0.
\end{align} 
where $C_{\ref{pro9.2.2}}$  is an absolute constant.
\end{mypropo}
We now introduce the lower bound of the Randomized least common denominator.
\begin{mylem}[Lemma 2.9 in \cite{Rank}]\label{Large RLCD}
Let $\delta, \rho \in \left(0,1 \right)$ and $T \ge 1$, let $U$ be an $n\times l$ matrix(deterministic) such that $U\mathbb{R}^{l} \cap S^{n-1} \in \mathrm{Incomp}\left(\delta ,\rho \right)$. Then there exists $h_{\ref{Large RLCD}} = h_{\ref{Large RLCD}}\left( \delta,\rho,T\right) > 0$ such that for any $\theta \in \mathbb{R}^{l}$ with $\Vert U\theta \Vert_{2} \le h_{\ref{Large RLCD}}\sqrt{n}$ and $j \in \left[ n\right]$ satisfies
\begin{align}
\Vert U\theta \Vert_{\xi} \ge L\sqrt{\log_{+}{\frac{\alpha \Vert U\theta\Vert_{2}}{L}}}.
\end{align}
where $\alpha \le \alpha_{\ref{Large RLCD}} = \alpha_{\ref{Large RLCD}}\left( \delta,\rho,T\right)$.
\end{mylem}
We present the following lemma, which, as a version of Fact 7.2 in \cite{HanY1}, offers a comparison between RD and log-RLCD.
\begin{mylem}\label{compare}
    Assume that $Y=(Y_{1},\dots,Y_{m})^{T}$ with $(Y_{1},\dots,Y_{m})$ is $\frac{1}{4}$-almost orthogonal. For any $\xi \sim G_{n}(T)$, we have 
    \begin{align}
        RD_{L,4\alpha}^{\xi}(Y) \le \mathrm{RlogD}_{L,\alpha}^{\xi}(Y) \le RD_{L,\alpha/4}^{\xi}(Y).\nonumber 
    \end{align}
\end{mylem}
At the end of this subsection, we present a high-dimensional mixed-rank estimation. It is a high-dimensional extension of Proposition \ref{propo4.1} in this paper. In addition, it can be viewed as an inhomogeneous version of Theorems 4.8 and 5.8 in \cite{HanY2}. The proof combines the methods of Section \ref{Section Tail bound for the gaps of the eigenvalue} of this paper and Theorem \ref{Smallball}, and uses Lemma \ref{compare} to replace log-RLCD with RD. We do not give a detailed proof here.
\begin{mypropo}\label{pro9.2.3}
     For $T\ge 1$, $t >0$, $n,m,k,d,l\in \mathbb{N}$ with $n  \ge m \ge d \ge  k$, $\alpha \in (0,1)$ and $\nu \le 1/4$. Let $X_{1},\dots,X_{l} \in \mathbb{R}^{d}$ be $1/4$-almost orthogonal satisfying $\Vert X_{i}\Vert_{2} \ge 32\sqrt{n}c_{0}^{-1}t_{i}^{-1}$ and $H$ be a random $(m-d) \times 2d$ matrix with independent rows satisfying $\mathrm{Row}_{j}(H) \in \Phi_{\nu}(2d,(\xi_{j},\xi_{j}'))$ for all $j\in [m-d]$, where $\xi_{1},\dots,\xi_{m-d} \in G_{d}(T)$ are independent random variables. There exist absolute positive constants $c_{\ref{pro9.2.3}}$, $c_{\ref{pro9.2.3}}'$, $d_{\ref{pro9.2.3}}$ and $C_{\ref{pro9.2.3}}$ such that the following holds. For $0 < c_{0} \le c_{\ref{pro9.2.3}}T^{-4} \nu$, $L \ge d_{\ref{pro9.2.3}}T^{2}\mu^{-1/2}\sqrt{l}$, $d \le 2c_{0}^{2}m$ and $\prod_{i=1}^{l}t_{i}  \ge c_{\ref{pro9.2.3}}'\alpha\gamma^{-1}\exp{(-2^{-3} d)}$, if $RD_{L,\alpha}^{\xi_{j}}(r_{n}X) >16\sqrt{l}$, then
     \begin{equation}
    \begin{aligned}
       & \mathbb{P}_{H}\left(  \sigma_{2d-k+1}(H) \le \frac{c_{0}\sqrt{m}}{16} \text{ and } \Vert H_{1}X_{i}\Vert_{2}, \Vert H_{2}X_{i}\Vert_{2} \le m \text{ for all } i \in [l] \right)\\ 
        & \le e^{-\frac{c_{0}nk}{4}}\cdot (R_{\ref{pro9.2.3}}^{l}\prod_{i=1}^{l}t_{i})^{2m-2d},\nonumber
    \end{aligned}
    \end{equation}
    where $H_{1}:=H_{[m-d]\times[d]}$, $H_{2}:=H_{[m-d]\times[d+1,2d]}$, $r_{n}:= \frac{c_{0}}{32\sqrt{n}}$ and $R_{\ref{pro9.2.3}} := C_{\ref{pro9.2.3}}L/(\alpha\sqrt{l})$.
\end{mypropo}
\subsection{Partition of the space}
 We now begin to analyze the properties of linear subspaces. We start by reviewing the relevant subspaces, guided by the assumptions of Theorem \ref{Distance}. Let $A \sim \mathrm{Sym}_n(T)$, and let $H$ be a random subspace spanned by any $n - k$ columns of $A$. Let $E = H^\perp$, so that $\dim E \ge k$. We now partition the space $E$. We first consider the compressible case, which is the inhomogeneous version of Proposition 3.1 in \cite{HanY2}.
 \begin{mypropo}\label{pro9.3.1}
     Let $A \sim \mathrm{Sym}_{n}(T)$ with $T \ge 1$ and $k,n \in \mathbb{N}$ with $k \le n$. Then there exist $c_{\ref{pro9.3.1}}>0$ and $\tau_{\ref{pro9.3.1}}>0$ depending only on $T$ such that the following holds. There exists a linear subspace $E_{1} \subset E$ with $\dim{E_{1}}=3k/4$ such that $E_{1} \cap \mathrm{Comp}(\tau_{\ref{pro9.3.1}}^{2},\tau_{\ref{pro9.3.1}}^{4}) = \emptyset$ with probability at least $1-e^{-c_{\ref{pro9.3.1}}kn}$.
 \end{mypropo}
Next, we will provide the second segmentation. We show that there is a subspace in which all vectors satisfy the no-gap delocalization property. This can be regarded as a rough high-dimensional version of Theorem \ref{eigenvectors} in this paper.
\begin{mypropo}\label{pro9.3.2}
    For $k, n \in \mathbb{N}$ with $k=o(n)$, $A \sim \mathrm{Sym}_{n}(T)$ with $T \ge 1$. For any $b_{0}>0$, there exist constants $c_{\ref{pro9.3.2}}>0$ and $\tau_{\ref{pro9.3.2}}>0$ depending only on $b_{0}$ and $T$ such that the following holds. There exists a linear subspace $E_{2} \subset E_{1}$ with $\dim{E_{2}} = k/2$ such that $E_{2} \cap \mathrm{Comp}\left((1-b_{0}^{2}/10),\tau_{\ref{pro9.3.2}}   \right) = \emptyset $ with probability $1-e^{-c_{\ref{pro9.3.2}}kn}$.
\end{mypropo}
\begin{proof}
    This result is equivalent to an inhomogeneous version of Proposition 4.1 in \cite{HanY2}. Thus, we apply the method of Section 4 in \cite{HanY2}, using Proposition \ref{pro9.2.3} instead of Theorem 4.8 in \cite{HanY2}, and replace Lemma 2.11 in \cite{HanY2} with Lemma \ref{Large RLCD} of the present paper to bound RD and obtain the inhomogeneous version of Fact 4.15 in \cite{HanY2}. Consequently, we directly reach this result without providing a detailed proof.
\end{proof}
Finally, we present the third subdivision. To begin, we require a discretization estimate for an almost orthogonal system, which has been established in our paper \cite{Rank}.
\begin{mypropo}[Proposition 3.2 in \cite{Rank}]\label{pro9.3.3}
There exists $r_{\ref{pro9.3.3}}>0$ depending only on $T$ such that the following holds. For $\rho >0$, $b_{0} >0$ and $L >0$, let $d = \left( d_{1},\dots,d_{l} \right)$ be a vector such that $d_{j} \in [r_{\ref{pro9.3.3}}\sqrt{n},R]$ with $R:=\exp{\frac{\rho n}{4L^{2}}}$, for all $j \in [l]$. 
Let $\delta \ge 0$ with $\delta \le \rho$ and $\mathcal{N}_{d} \subset \left( \delta \mathbb{Z}^{n} \right)^{l}$ be the set of all $l$-tuples of vectors $u_{1},\dots,u_{l}$ such that 
$$\Vert u_{j} \Vert_{2} \in [\frac{1}{2}d_{j},4d_{j}] \text{ for all } j \in [l],$$
$$\mathrm{d}_{A}\left( u_{j},\mathbb{Z}^{n} \right) < 2 \rho \sqrt{n}$$
and
$$\mathrm{span}\left( u_{1},\dots,u_{l} \right) \cap S^{n-1} \subset \mathrm{Incomp}\left( b_{0}^{2}/8,\tau_{\ref{pro9.3.2}} \right).$$
Then 
$$\left| \mathcal{N}_{d} \right| \le \left( \frac{C_{\ref{pro9.3.3}}\rho^{c_{\ref{pro9.3.3}}}}{r\delta} \right)^{ln}\left( \prod_{j=1}^{l}{\frac{d_{j}}{\sqrt{n}}} \right)^{n},$$
where $C_{\ref{pro9.3.3}},c_{\ref{pro9.3.3}}> 0 $ depending only on $b_{0},T$.
\end{mypropo}
 We now give the last proposition.
 \begin{mypropo}\label{pro9.3.4}
     Let $A \sim \mathrm{Sym}_{n}(T)$ with $T \ge 1$. For $b_{0}>0$, let $E_{2} \subset E$ from Proposition \ref{pro9.3.2} satisfy $E_{2} \cap S^{n-1}  \subset \mathrm{Incomp}\left((1-b_{0}^{2}/10),\tau_{\ref{pro9.3.2}}  \right)$. Then there exist constants $\alpha \in (0,1)$, $L_{\ref{pro9.3.4}}$, $c_{\ref{pro9.3.4}}$, $c_{\ref{pro9.3.4}}'$ and $C_{\ref{pro9.3.4}}$ depending only on $b_{0}$ and $T$ such that the following holds. For any $1 \le k \le c_{\ref{pro9.3.4}}\sqrt{n}$, there exists a linear subspace $E_{3} \subset E_{2}$ with $\dim{E_{3}} \ge k/4$ such that for any $\xi \sim G_{n}(T)$,
     \begin{align}
         RD_{L,\alpha}^{\xi}(E_{3}) \ge \exp{\left( \frac{C_{\ref{pro9.3.4}}n}{k} \right)},\nonumber
     \end{align}
     with probability at least $1-e^{-c_{\ref{pro9.3.4}}'kn}$.
 \end{mypropo}
 \begin{proof}
     Following the proof of Proposition \ref{pro9.3.2}, we apply the approach of Section 5 of \cite{HanY2}, substituting Theorem 5.8 there with Proposition \ref{pro9.2.3} in this paper, and replacing Lemma 5.3 there with Proposition \ref{pro9.3.3} here. This allows us to directly draw this conclusion without a detailed proof.
 \end{proof}

\subsection{Proof of Theorem \ref{The_main}}\label{Proof of Theorem main}
We have now completed all the preparations. Next, we proceed to complete the proof of Theorem \ref{The_main}.
\begin{proof}[\textsf{Proof of Theorem \ref{The_main}}]
    Recall the analysis of Theorem \ref{Gaps}, consider $F=0$ and $z=0$, we have 
    \begin{align}\label{9ine}
        \mathbb{P}\left( \sigma_{n-k+1}(A) \le \frac{k\varepsilon}{\sqrt{n}} \right) \le \left( \frac{Cn}{l} \right)^{l}\max_{i_{1},\dots,i_{l}}\mathbb{P}(\mathcal{E}_{u,\gamma}),
    \end{align}
    where $\mathcal{E}_{u,\gamma}$ is defined in \eqref{9def} and $l = uk$ with $u \in (0,1)$.
    By combining the proof process of Theorem \ref{Gaps} and Proposition \ref{pro9.3.4}, we can obtain
    \begin{align}
        \mathbb{P}\left( \mathrm{dist}(A_{i},H+v_{i}) \right) \le \left( C\varepsilon/k \right)^{ck^{2}}+e^{-cn},\nonumber
    \end{align}
    with probability at least $1-e^{-ckn}$. It implies our result by combining the \eqref{9ine}.
    
\end{proof}

\section{Application in the eigenvectors of the symmetric matrices}\label{Application in the eigenvectors of the symmetric matrices}
In this section, we present another application of the new distance theorem in random matrix theory, namely the distributional properties of eigenvectors of symmetric random matrices. In recent years, the properties of eigenvectors of random matrices have received extensive attention. For the non-Hermitian case, \cite{Rudelson_gafa}, \cite{LT_ptrf}, and \cite{KS_rsa} have provided excellent estimates, while for the symmetric case, related research is more challenging and comparatively scarce; indeed, Rudelson and Vershynin \cite{Rudelson_gafa} have provided an estimate for the symmetric setting. The theorem is as follows.

\begin{mytheo}[Theorem 1.5 in \cite{Rudelson_gafa}]
    Let $A$ be $n \times n $ symmetric random matrix with entries $(A_{ij})_{i \le j}$ i.i.d. according to a subgaussian distribution. Let $\varepsilon  \ge 1/n$ and $s  \ge c_{1}\varepsilon^{-7/6}n^{-1/6}+e^{-c_{2}/\sqrt{\varepsilon}}$, the following holds with the probability at least $1-(cs)^{\varepsilon n}-e^{-cn}$. Every eigenvector $v$ of $A$ satisfies
\begin{align}
    \Vert v_{I}\Vert_{2}  \ge(\varepsilon s)^{6} \Vert v\Vert_{2} ~~ \text{for}~~\text{all}~~I \subseteq [n] , |I| \ge\varepsilon n.\nonumber
\end{align}
\end{mytheo}
Indeed, extensive past research has shown that the delocalization of eigenvectors is related to estimates of the smallest singular value of a submatrix formed by selecting several columns of the random matrix. In the first theorem of this subsection, we provide an estimate for the smallest singular value of a symmetric submatrix or, equivalently, a symmetric rectangular matrix. Indeed, to make the proof of Theorem \ref{eigenvectors} go through, we again introduce a bias term $F$ in the same spirit as in Theorem \ref{Gaps}, thereby showing that the conclusion remains valid even when the matrix entries have non-zero mean.

\begin{mytheo}\label{Singular}
    For $T \ge 1 $, $N,n \in \mathbb{N}$ and $M >0$. Let $A$ be an $N \times n$ random matrix that satisfies that there exists a submatrix $A_{0} \sim \mathrm{Sym}_{n}(T)$, $d= N-n+1$ and $F$ be a $N \times n$ matrix satisfying for all $J \subseteq [n]$, $|J|=d$, $\Vert F_{J}\Vert_{\mathrm{HS}} \le M\sqrt{Nd}$, where $F_J$ denotes the submatrix of $F$ formed by the columns indexed by $J$. Then there exist constants $c_{\ref{Singular}}$, $c_{\ref{Singular}}'$ and $d_{\ref{Singular}}$ depending only on $T$ and $C_{\ref{Singular}}$ depending on $T$ and $M$ such that the following holds. For any $c_{\ref{Singular}} \le N-n \le d_{\ref{Singular}}n/\log{n}$, $\varepsilon  \ge 0$ and $\gamma \in (0,1)$, we have 
    \begin{align}
        \mathbb{P}\left( \sigma_{n}(A+F) \le\varepsilon \left( \sqrt{N}-\sqrt{n-1} \right) \right) \le\left( \frac{N}{d}\right)^{d/2}\left( C_{\ref{Singular}}\gamma^{-1}\varepsilon \right)^{(1-\gamma)d-1}+e^{-c_{\ref{Singular}}'N}.\nonumber
    \end{align}
\end{mytheo}
In fact, for the non-Hermitian random matrix, Rudelson and Vershynin introduced the optimal bound for the least singular value in \cite{Rudelson_cpam}. However, for symmetric random matrices, no comparable results were available prior to this work. For the first time, we employ a geometric approach to obtain an estimate of the smallest singular value when the difference between $N$ and $n$ lies between a large constant and $c n / \log n$.

Now, we introduce our second main result, which gives the no-gap delocalization of the eigenvectors.
\begin{mytheo}\label{eigenvectors}
    Under the assumption of Theorem \ref{Singular}, there exist positive constants $c_{\ref{eigenvectors}}$, $c_{\ref{eigenvectors}}'$, $d_{\ref{eigenvectors}}$ and $C_{\ref{eigenvectors}}$ depending only on $T$ such that the following holds. For any $\varepsilon \in (\frac{c_{\ref{eigenvectors}}}{n}, \frac{d_{\ref{eigenvectors}}}{\log{n}} )$, $\gamma \in (0,1)$ and $\delta >0$, we have 
    \begin{align}
        \mathbb{P}\left( \exists \text{eigenvector}~ v \in S^{n-1}: \min_{I \subseteq [n], |I|=\varepsilon n}{\Vert v_{I}\Vert_{2}} < \delta \right) \le\delta^{-3}\left( \frac{C_{\ref{eigenvectors}}\delta^{1-\gamma}}{\gamma\varepsilon^{5/2-\gamma}} \right)^{\varepsilon n} +e^{-c_{\ref{eigenvectors}}'n}.\nonumber
    \end{align}
\end{mytheo}
If we take $\delta=\left( \varepsilon s \right)^{6}$ and $\gamma=1/2$, then we can obtain the following corollary.
\begin{mycoro}\label{eig2}
    Under the assumptions of Theorem \ref{Singular}, there exist positive constants $c_{\ref{eig2}}$, $c_{\ref{eig2}}'$, $d_{\ref{eig2}}$ and $C_{\ref{eig2}}$ depending only on $T$ such that the following holds. For $c_{\ref{eig2}}/n \le \varepsilon \le d_{\ref{eig2}}/\log{n} $, we have
    \begin{align}
        \mathbb{P}\left( \min_{I\subseteq [n],|I|=\varepsilon n}{\Vert v_{I}\Vert_{2}} \le (\varepsilon s)^{6}   \right) \le(C_{\ref{eig2}}\varepsilon s^{3})^{\varepsilon n}(\varepsilon s)^{-18} +e^{-c_{\ref{eig2}}'n},\nonumber
    \end{align}
    for all $s<1$. 
\end{mycoro}
Reviewing the theorem from \cite{Rudelson_gafa} presented at the beginning of this section, we find that, of the same order, the upper bound of probability in the theorem of this paper is lower than that of \cite{Rudelson_gafa}.

We now sketch the proof. In fact, following the approach developed by Rudelson and Vershynin \cite{Rudelson_gafa}, this problem can be reduced to estimating the smallest singular value of submatrices of symmetric random matrices. Using the standard techniques developed by Rudelson and Vershynin \cite{Rudelson_cpam}, we can further reduce the problem to analyzing a certain distance:
\begin{align*}
    \mathbb{P}\left( \mathrm{dist}(A_{J}z,H_{J^{c}}) \le \varepsilon \sqrt{d}\right)
\end{align*}
However, it is important to note that the $A_{J}$ and $H_{J^{c}}$ involved are not independent. To address this issue, we partition the matrix into suitable blocks and derive a new estimate, thereby completing the proof of Theorem \ref{eigenvectors}.

\subsection{The smallest singular values of the submatrix of symmetric matrices}\label{The smallest singular values of the submatrix of symmetric matrices}
Before proving the theorem, we first give some definitions. First, let $N = n - 1 + d$. Next, we define 
$H_{J} := \text{span}\{M_{k}\}_{k \in J}$, $M_{J}$ as the submatrix of $M$ consisting of all column vectors whose indices belong to $J$.
Define $M_{J,I}$ and $H_{J,I}$ as projections of $M_{J}$ and $H_{J}$ onto the coordinates indexed by $I$ respectively.

For all $\delta$ and $\rho$, we set $K_{1}:=\rho \sqrt{\frac{\delta}{2}}$ and $K_{2}:=1/K_{1}$, define
\begin{align}
    \mathrm{Spread}_{J}:=\left\{ x \in S^{|J|-1}: \frac{K_{1}}{\sqrt{d}} \le|x_{k}| \le\frac{K_{2}}{\sqrt{d}}  \right\}.\nonumber
\end{align}
We have the following lemma.
\begin{mylem}[Lemma 6.3  in \cite{Rudelson_cpam}]\label{Lem8.1} 
    Let $\delta$, $\rho \in (0,1)$. There exist $C_{\ref{Lem8.1}}$, $c_{\ref{Lem8.1}} >0$, which depend only on $\delta$, $\rho$ and such that the following hold: Let $J$ be any subset of $d$ elements of $[n]$. Then for every $\varepsilon > 0$ and $M:=A+F$ satisfying condition of Theorem \ref{Singular}
    \begin{align}
        \mathbb{P}\left( \inf_{x \in \mathrm{Incomp}(\delta,\rho)}{\Vert Mx\Vert_{2}} < c_{\ref{Lem8.1}}\varepsilon \sqrt{\frac{d}{n}}  \right) < C_{\ref{Lem8.1}}^{d}\mathbb{P}\left( \inf_{z \in \mathrm{Spread}_{J}}{\mathrm{dist}\left( Mz,H_{J^{c}} \right)} <\varepsilon \right).\nonumber
    \end{align}
\end{mylem}
Based on the above lemma, we can provide the proof of Theorem \ref{Singular}.
\begin{proof}[\textsf{Proof of Theorem \ref{Singular}}]
    We first state the following definition: \begin{align}
        S_{J}:=\left\{ y \in \frac{3}{2}B_{2}^{n} \setminus \frac{1}{2}B_{2}^{n} : y_{k} \in \frac{t}{\sqrt{KN}}\mathbb{Z}, \ \frac{K_{1}}{2\sqrt{d}} \le|y_{k}| \le\frac{2K_{2}}{\sqrt{d}} ~~\text{for all } k \in J  \right\},\nonumber
    \end{align}
    where $K$ is to be determined and will be specified later.
    We can obtain that for any $x \in \mathrm{Spread}_{J}$, there exists $y \in S_{J}$ such that $$\Vert A(x-y)\Vert_{2} \le\frac{t\Vert A_{J}\Vert_{\mathrm{HS}}}{\sqrt{KN}},$$ which can be derived using the random rounding method. Set $$\mathcal{E}:=\left\{  \Vert A_{J}\Vert_{\mathrm{HS}} \le\sqrt{KNd} \right\},$$ we know that $\mathcal{E}$ occurs with probability at least $1-e^{-cNd}$ if we choose $K$ large enough.
    
    Without loss of generality, we assume $J = [d]$ and set $J_{0} = [\gamma d]$, $\gamma \in (0,1)$. For all $y \in S_{J}$ and $t > e^{-UN/d}$($U$ is to be determined), 
    \begin{equation}
        \begin{aligned}
            \mathbb{P}\left( \mathrm{dist}\left( My,H_{J^{c}} \right) \le t\sqrt{d} \right) & \le\mathbb{P}\left( \mathrm{dist}\left( M_{J,J_{0}^{c}}y,H_{J^{c},J_{0}^{c}} \right) \le t\sqrt{d} \right)\\
            & \le\mathbb{P}\left( \Vert P_{H_{J^{c},J_{0}^{c}}^{\perp} } \left( M_{J_{0},J_{0}^{c}}y+M_{J \setminus J_{0},J_{0}^{c} }y  \right) \Vert_{2} \le t\sqrt{d}     \right)\\
            & \le\mathcal{L}\left( P_{H_{J^{c},J_{0}^{c}}^{\perp} }  A_{J_{0},J_{0}^{c}}y , t\sqrt{d} \right)\\
            & \le\left( C\gamma^{-1}t \right)^{(2-\gamma)d-1},\nonumber
        \end{aligned}
    \end{equation}
    by Theorem \ref{Distance}. It is important to note that for $A_{J_{0},J_{0}^{c}}y$, although the variance is not $1$, it can be bounded by two positive constants depending only on $K_{1}$ and $K_{2}$, thus satisfying the conditions of Theorem \ref{Distance}. It is also important to note that both $C$ and $U$ are constants depending only on $T$ and $d \le cn/\log{n}$.
    Furthermore, we have 
    \begin{align}
        \mathbb{P}\left(  \inf_{z \in \mathrm{Spread}_{J}}{\mathrm{dist}\left( Mz,H_{J^{c}} \right) } \le t\sqrt{d} \right) \le\left| S_{J} \right|\max_{y \in S_{J}}{\mathbb{P}\left( \mathrm{dist}\left( My,H_{J} \right) \le2t\sqrt{d}  \right)}.\nonumber
    \end{align}
    Recall the definition of $S_{J}$, we now apply the Lemma \ref{Comp} and Lemma \ref{Lem8.1}, 
    \begin{equation}
    \begin{aligned}
        \mathbb{P}\left( s_{n}(M) \le\varepsilon(\sqrt{N}-\sqrt{n-1}) \right) 
        & \le\mathbb{P}\left( \inf_{x \in \mathrm{Incomp}\left( \delta,\rho \right)}{\Vert Mx\Vert_{2} } \le\varepsilon \frac{d}{\sqrt{n}} \right) +e^{-cN}\\
        & \le(N/d)^{d/2}\left( C\gamma^{-1}t\right)^{(1-\gamma)d-1} +e^{-cN}.\nonumber
    \end{aligned}
    \end{equation}
    It completes the proof of Theorem \ref{Singular}.
\end{proof}
\subsection{Proof of Theorem \ref{eigenvectors}}\label{Proof of Theorem eigenvectors}
A method is introduced to calculate the delocalization of matrix eigenvectors using submatrices in Rudelson and Vershynin  \cite{Rudelson_gafa}. Now, we will use the conclusions from Section \ref{The smallest singular values of the submatrix of symmetric matrices} in conjunction with their method to obtain the delocalization of the eigenvectors. We first define:
\begin{align}
    \mathrm{Loc}\left( A,\varepsilon,\delta \right):=\left\{ \exists \text{eigenvector } v \in S^{n-1}:\exists I \subseteq [n], |I|=\varepsilon n :\Vert v_{I}\Vert_{2} < \delta \right\}.\nonumber
\end{align}
We first need a lemma to establish the connection between the smallest singular value of a submatrix and the eigenvectors of the matrix. For a detailed proof of this method, refer to the work of Rudelson and Vershynin in \cite{Rudelson_gafa}.
\begin{mylem}[Proposition 4.1 in \cite{Rudelson_gafa}]\label{gafa}
    Let $ A $ be an $n \times n$ random matrix. Let $ M \geq 1 $ and \( \varepsilon, p_0, \delta \in (0, 1/2) \). Assume that for any number $ \lambda_0 \in \mathbb{C} $, \( |\lambda_0| \leq M \sqrt{n} \), and for any set \( I \subset [n] \), \( |I| = \varepsilon n \), we have
\begin{align}
\mathbb{P}\left\{ s_{\min}\left( (A - \lambda_0)_{I^{c}} \right) \leq 8 \delta M \sqrt{n} \text{ and } \mathcal{B}_{A, M} \right\} \leq p_0.\nonumber
\end{align}
Then
\begin{align}
\mathbb{P}\left\{ \mathrm{Loc}(A, \varepsilon, \delta) \text{ and } \mathcal{B}_{A, M} \right\} \leq 5 \delta^{-2} \left( \frac{e}{\varepsilon} \right)^{\varepsilon n} p_0,\nonumber
\end{align}
where $\mathcal{B}_{A,M}:=\left\{  \Vert A\Vert \le M\sqrt{n} \right\}$.
\end{mylem}
We now give the proof of Theorem \ref{eigenvectors}
\begin{proof}[\textsf{Proof of Theorem \ref{eigenvectors}}]
    Applying Theorem \ref{Singular} with $N=n$, $n=(1-\varepsilon)n$ and $F:= (-\lambda_{0})I_{n}$($|\lambda_{0}| \le M\sqrt{n}$ and $I_{n}$ is identity matrix), we have
    \begin{align}
        \mathbb{P}\left( s_{min}\left( (A-\lambda_{0})_{I^{c}} \right)  \le8\delta M \sqrt{n} \right) \le\varepsilon^{-\varepsilon n/2}\left( C\gamma^{-1}\delta/\varepsilon \right)^{(1-\gamma)\varepsilon n -1} +e^{-cn}.\nonumber
    \end{align}
    Using the Lemma \ref{gafa}, we finally have
    \begin{align}
        \mathbb{P}\left( \mathrm{Loc}\left( A,\varepsilon,\delta \right) \right) \le\delta^{-3}\left( C\gamma^{-1}\delta^{1-\gamma}/\varepsilon^{5/2-\gamma} \right)^{\varepsilon n} +e^{-cn}.\nonumber
    \end{align}
    It completes the proof of Theorem \ref{eigenvectors}
\end{proof}


\textbf{Acknowledgment:} This work was supported by the National Key R\&D Program of China (No.2024YFA1013501), the National Natural Science Foundation of China  (No. 12571162),   Shandong Provincial Natural Science Foundation (No. ZR2024MA082), and the Youth Student Fundamental study Funds of Shandong University  (No. SDU-QM-B202407).

\textbf{Statements and Declarations:}  On behalf of all authors, Hanchao Wang (the corresponding author) states that there is no conflict of interest.

\printbibliography

@article {BEYY_cpam,
    AUTHOR = {Bourgade, Paul and Erd\H os, Laszlo and Yau, Horng-Tzer and
              Yin, Jun},
     TITLE = {Fixed energy universality for generalized {W}igner matrices},
   JOURNAL = {Comm. Pure Appl. Math.},
  FJOURNAL = {Communications on Pure and Applied Mathematics},
    VOLUME = {69},
      YEAR = {2016},
    NUMBER = {10},
     PAGES = {1815--1881},
      ISSN = {0010-3640,1097-0312},
   MRCLASS = {60B20 (81R12)},
  MRNUMBER = {3541852},
MRREVIEWER = {Nizar\ Demni},
       DOI = {10.1002/cpa.21624},
       URL = {https://doi.org/10.1002/cpa.21624},
}

@article {Bour_jfa,
    AUTHOR = {Bourgain, Jean and Vu, Van H. and Wood, Philip Matchett},
     TITLE = {On the singularity probability of discrete random matrices},
   JOURNAL = {J. Funct. Anal.},
  FJOURNAL = {Journal of Functional Analysis},
    VOLUME = {258},
      YEAR = {2010},
    NUMBER = {2},
     PAGES = {559--603},
      ISSN = {0022-1236,1096-0783},
   MRCLASS = {60B20},
  MRNUMBER = {2557947},
MRREVIEWER = {Guangyu\ Yang},
       DOI = {10.1016/j.jfa.2009.04.016},
       URL = {https://doi.org/10.1016/j.jfa.2009.04.016},
}

@article {Campos_jams,
    AUTHOR = {Campos, Marcelo and Jenssen, Matthew and Michelen, Marcus and
              Sahasrabudhe, Julian},
     TITLE = {The singularity probability of a random symmetric matrix is
              exponentially small},
   JOURNAL = {J. Amer. Math. Soc.},
  FJOURNAL = {Journal of the American Mathematical Society},
    VOLUME = {38},
      YEAR = {2025},
    NUMBER = {1},
     PAGES = {179--224},
      ISSN = {0894-0347,1088-6834},
   MRCLASS = {60B20 (15B52)},
  MRNUMBER = {4810062},
       DOI = {10.1090/jams/1042},
       URL = {https://doi.org/10.1090/jams/1042},
}

@article {Campos_pi,
    AUTHOR = {Campos, Marcelo and Jenssen, Matthew and Michelen, Marcus and
              Sahasrabudhe, Julian},
     TITLE = {The least singular value of a random symmetric matrix},
   JOURNAL = {Forum Math. Pi},
  FJOURNAL = {Forum of Mathematics. Pi},
    VOLUME = {12},
      YEAR = {2024},
     PAGES = {Paper No. e3, 69},
      ISSN = {2050-5086},
   MRCLASS = {60B20 (15B52)},
  MRNUMBER = {4695501},
MRREVIEWER = {Mostafa\ Sabri},
       DOI = {10.1017/fmp.2023.29},
       URL = {https://doi.org/10.1017/fmp.2023.29},
}

@article {CLOS_gap,
    AUTHOR = {Christoffersen, Nicholas and Luh, Kyle and O'Rourke, Sean and Shearer Calum},
     TITLE = {Gaps between singular values of sample covariance matrices.},
   JOURNAL = {arXiv: 2502.15002.},
          YEAR = {2025},
       URL = {arXiv: 2502.15002.},
}

@article {CLNW_gap,
    AUTHOR = {Christoffersen, Nicholas and Luh, Kyle and Nguyen, Hoi and Wang, Jingheng},
     TITLE = {Eigenvalue gaps of the Laplacian of random graphs.},
   JOURNAL = {arXiv: 2501.00234.},
          YEAR = {2025},
       URL = {arXiv: 2501.00234.},
}

@article {Costello_duke,
    AUTHOR = {Costello, Kevin P. and Tao, Terence and Vu, Van},
     TITLE = {Random symmetric matrices are almost surely nonsingular},
   JOURNAL = {Duke Math. J.},
  FJOURNAL = {Duke Mathematical Journal},
    VOLUME = {135},
      YEAR = {2006},
    NUMBER = {2},
     PAGES = {395--413},
      ISSN = {0012-7094,1547-7398},
   MRCLASS = {15A52 (05C35 05D40 60C05)},
  MRNUMBER = {2267289},
MRREVIEWER = {Steven\ Joel\ Miller},
       DOI = {10.1215/S0012-7094-06-13527-5},
       URL = {https://doi.org/10.1215/S0012-7094-06-13527-5},
}

@article {DF,
    AUTHOR = {Dabagia, Max  and  Fernandez,  Manuel},
     TITLE = {The smallest singular value of inhomogenous random rectangular matrices.},
   JOURNAL = {arXiv: 2408.14389.},
          YEAR = {2024},
       URL = {arXiv: 2408.14389.},
}

@article {Rank,
    AUTHOR = {Dai, Guozheng  and Song, Zeyan and Wang, Hanchao},
     TITLE = {The rank and singular values of the inhomogeneous subgaussian random matrices.},
   JOURNAL = {arXiv: 2412.18906.},
          YEAR = {2024},
       URL = {arXiv: 2412.18906.},
}

@article {HanY,
    AUTHOR = {Han, Yi},
     TITLE = {Simplicity of singular value spectrum of random matrices and two-point quantitative invertibility.},
   JOURNAL = {arXiv: 2502.13819.},
          YEAR = {2025},
       URL = {arXiv: 2502.13819.},
}

@article {HanY1,
    AUTHOR = {Han, Yi},
     TITLE = { Repeated singular values of a random symmetric matrix and decoupled singular value estimates.},
   JOURNAL = {arXiv: 2504.15992.},
          YEAR = {2025},
       URL = {arXiv: 2504.15992.},
}

@article {HanY2,
    AUTHOR = {Han, Yi},
     TITLE = {On the rank of a random symmetric matrix in the large deviation regime.},
   JOURNAL = {arXiv: 2506.01155.},
          YEAR = {2025},
       URL = {arXiv: 2506.01155.},
}

@article {MFV,
    AUTHOR = {Fernandez, Manuel},
     TITLE = {A distance theorem for inhomogeneous random rectangular matrices.},
   JOURNAL = {arXiv: 2408.06309.},
          YEAR = {2024},
       URL = {arXiv: 2408.06309.},
}

@article {Edelman,
    AUTHOR = {Edelman, Alan},
     TITLE = {Eigenvalues and condition numbers of random matrices},
   JOURNAL = {SIAM J. Matrix Anal. Appl.},
  FJOURNAL = {SIAM Journal on Matrix Analysis and Applications},
    VOLUME = {9},
      YEAR = {1988},
    NUMBER = {4},
     PAGES = {543--560},
      ISSN = {0895-4798},
   MRCLASS = {15A52 (62H10)},
  MRNUMBER = {964668},
MRREVIEWER = {D.\ S.\ Tracy},
       DOI = {10.1137/0609045},
       URL = {https://doi.org/10.1137/0609045},
}

@article {ESY_imrn,
    AUTHOR = {Erd\H os, L\'aszl\'o{} and Schlein, Benjamin and Yau,
              Horng-Tzer},
     TITLE = {Wegner estimate and level repulsion for {W}igner random
              matrices},
   JOURNAL = {Int. Math. Res. Not. IMRN},
  FJOURNAL = {International Mathematics Research Notices. IMRN},
      YEAR = {2010},
    NUMBER = {3},
     PAGES = {436--479},
      ISSN = {1073-7928,1687-0247},
   MRCLASS = {60B20},
  MRNUMBER = {2587574},
MRREVIEWER = {Longmin\ Wang},
       DOI = {10.1093/imrn/rnp136},
       URL = {https://doi.org/10.1093/imrn/rnp136},
}

@article {FTW_gafa,
    AUTHOR = {Feng, Renjie and Tian, Gang and Wei, Dongyi},
     TITLE = {Small gaps of {GOE}},
   JOURNAL = {Geom. Funct. Anal.},
  FJOURNAL = {Geometric and Functional Analysis},
    VOLUME = {29},
      YEAR = {2019},
    NUMBER = {6},
     PAGES = {1794--1827},
      ISSN = {1016-443X,1420-8970},
   MRCLASS = {60B20},
  MRNUMBER = {4034920},
       DOI = {10.1007/s00039-019-00520-5},
       URL = {https://doi.org/10.1007/s00039-019-00520-5},
}

@article {Halasz1,
    AUTHOR = {Hal\'asz, G\'abor},
     TITLE = {On the distribution of additive arithmetic functions},
   JOURNAL = {Acta Arith.},
  FJOURNAL = {Polska Akademia Nauk. Instytut Matematyczny. Acta Arithmetica},
    VOLUME = {27},
      YEAR = {1975},
     PAGES = {143--152},
      ISSN = {0065-1036},
   MRCLASS = {10H25},
  MRNUMBER = {369292},
MRREVIEWER = {S.\ L.\ Segal},
       DOI = {10.4064/aa-27-1-143-152},
       URL = {https://doi.org/10.4064/aa-27-1-143-152},
}

@article {Halasz2,
    AUTHOR = {Hal\'asz, G.},
     TITLE = {Estimates for the concentration function of combinatorial
              number theory and probability},
   JOURNAL = {Period. Math. Hungar.},
  FJOURNAL = {Periodica Mathematica Hungarica. Journal of the J\'anos Bolyai
              Mathematical Society},
    VOLUME = {8},
      YEAR = {1977},
    NUMBER = {3-4},
     PAGES = {197--211},
      ISSN = {0031-5303,1588-2829},
   MRCLASS = {60G50 (10K20 60C05)},
  MRNUMBER = {494478},
MRREVIEWER = {H.\ Kesten},
       DOI = {10.1007/BF02018403},
       URL = {https://doi.org/10.1007/BF02018403},
}

@article {Huang_arxiv,
    AUTHOR = {Huang, Han},
     TITLE = {Rank of sparse {B}ernoulli matrices},
   JOURNAL = {Pure Appl. Funct. Anal.},
  FJOURNAL = {Pure and Applied Functional Analysis},
    VOLUME = {10},
      YEAR = {2025},
    NUMBER = {6},
     PAGES = {1365--1441},
      ISSN = {2189-3756,2189-3764},
   MRCLASS = {60B20 (15A18 15B52)},
  MRNUMBER = {5012558},
}

@article {Wangke,
    AUTHOR = {He, Yiyun and Wang, Ke and Zhu, Yizhe},
     TITLE = {Sparse {H}anson-{W}right inequalities with applications},
   JOURNAL = {Electron. J. Probab.},
  FJOURNAL = {Electronic Journal of Probability},
    VOLUME = {31},
      YEAR = {2026},
      ISSN = {1083-6489},
   MRCLASS = {60E15 (60B20)},
  MRNUMBER = {5030037},
       DOI = {10.1214/26-ejp1493},
       URL = {https://doi.org/10.1214/26-ejp1493},
}

@article {Jain_ecp,
    AUTHOR = {Jain, Vishesh and Sah, Ashwin and Sawhney, Mehtaab},
     TITLE = {Rank deficiency of random matrices},
   JOURNAL = {Electron. Commun. Probab.},
  FJOURNAL = {Electronic Communications in Probability},
    VOLUME = {27},
      YEAR = {2022},
     PAGES = {Paper No. 14, 9},
      ISSN = {1083-589X},
   MRCLASS = {60B20},
  MRNUMBER = {4386039},
MRREVIEWER = {Marcus\ Michelen},
       DOI = {10.1214/22-ECP455},
       URL = {https://doi.org/10.1214/22-ECP455},
}

@article {KKS_jams,
    AUTHOR = {Kahn, Jeff and Koml\'os, J\'anos and Szemer\'edi, Endre},
     TITLE = {On the probability that a random {$\pm 1$}-matrix is singular},
   JOURNAL = {J. Amer. Math. Soc.},
  FJOURNAL = {Journal of the American Mathematical Society},
    VOLUME = {8},
      YEAR = {1995},
    NUMBER = {1},
     PAGES = {223--240},
      ISSN = {0894-0347,1088-6834},
   MRCLASS = {15A52 (11K99 60C05)},
  MRNUMBER = {1260107},
       DOI = {10.2307/2152887},
       URL = {https://doi.org/10.2307/2152887},
}

@article {Livshyts_jam,
    AUTHOR = {Livshyts, Galyna V.},
     TITLE = {The smallest singular value of heavy-tailed not necessarily
              i.i.d. random matrices via random rounding},
   JOURNAL = {J. Anal. Math.},
  FJOURNAL = {Journal d'Analyse Math\'ematique},
    VOLUME = {145},
      YEAR = {2021},
    NUMBER = {1},
     PAGES = {257--306},
      ISSN = {0021-7670,1565-8538},
   MRCLASS = {60B20},
  MRNUMBER = {4361906},
MRREVIEWER = {Khanh\ Duy\ Trinh},
       DOI = {10.1007/s11854-021-0183-2},
       URL = {https://doi.org/10.1007/s11854-021-0183-2},
}

@article {Livshyts_aop,
    AUTHOR = {Livshyts, Galyna V. and Tikhomirov, Konstantin and Vershynin,
              Roman},
     TITLE = {The smallest singular value of inhomogeneous square random
              matrices},
   JOURNAL = {Ann. Probab.},
  FJOURNAL = {The Annals of Probability},
    VOLUME = {49},
      YEAR = {2021},
    NUMBER = {3},
     PAGES = {1286--1309},
      ISSN = {0091-1798,2168-894X},
   MRCLASS = {60B20 (15B52)},
  MRNUMBER = {4255145},
MRREVIEWER = {Asad\ Lodhia},
       DOI = {10.1214/20-aop1481},
       URL = {https://doi.org/10.1214/20-aop1481},
}

@article {KS_rsa,
    AUTHOR = {Luh, Kyle and O'Rourke, Sean},
     TITLE = {Eigenvector delocalization for non-{H}ermitian random matrices
              and applications},
   JOURNAL = {Random Structures Algorithms},
  FJOURNAL = {Random Structures \& Algorithms},
    VOLUME = {57},
      YEAR = {2020},
    NUMBER = {1},
     PAGES = {169--210},
      ISSN = {1042-9832,1098-2418},
   MRCLASS = {60B20 (15A18 15B52)},
  MRNUMBER = {4120597},
MRREVIEWER = {Guilherme\ L. F. Silva},
       DOI = {10.1002/rsa.20917},
       URL = {https://doi.org/10.1002/rsa.20917},
}

@article {KV_aihp,
    AUTHOR = {Luh, Kyle and Vu, Van},
     TITLE = {Sparse random matrices have simple spectrum},
   JOURNAL = {Ann. Inst. Henri Poincar\'e{} Probab. Stat.},
  FJOURNAL = {Annales de l'Institut Henri Poincar\'e{} Probabilit\'es et
              Statistiques},
    VOLUME = {56},
      YEAR = {2020},
    NUMBER = {4},
     PAGES = {2307--2328},
      ISSN = {0246-0203,1778-7017},
   MRCLASS = {60B20 (05C80 15B52)},
  MRNUMBER = {4164838},
MRREVIEWER = {Florent\ Benaych-Georges},
       DOI = {10.1214/19-AIHP1032},
       URL = {https://doi.org/10.1214/19-AIHP1032},
}

@article {LT_ptrf,
    AUTHOR = {Lytova, Anna and Tikhomirov, Konstantin},
     TITLE = {On delocalization of eigenvectors of random non-{H}ermitian
              matrices},
   JOURNAL = {Probab. Theory Related Fields},
  FJOURNAL = {Probability Theory and Related Fields},
    VOLUME = {177},
      YEAR = {2020},
    NUMBER = {1-2},
     PAGES = {465--524},
      ISSN = {0178-8051,1432-2064},
   MRCLASS = {15B52 (15A18 60B20)},
  MRNUMBER = {4095020},
MRREVIEWER = {Dominique\ L\'epingle},
       DOI = {10.1007/s00440-019-00956-8},
       URL = {https://doi.org/10.1007/s00440-019-00956-8},
}

@incollection {Naor_JTDM,
    AUTHOR = {Naor, Assaf and Youssef, Pierre},
     TITLE = {Restricted invertibility revisited},
 BOOKTITLE = {A journey through discrete mathematics},
     PAGES = {657--691},
 PUBLISHER = {Springer, Cham},
      YEAR = {2017},
      ISBN = {978-3-319-44478-9; 978-3-319-44479-6},
   MRCLASS = {15A60 (47J07)},
  MRNUMBER = {3726618},
MRREVIEWER = {Selcuk\ Koyuncu},
}

@article {Nguyen_duke,
    AUTHOR = {Nguyen, Hoi H.},
     TITLE = {Inverse {L}ittlewood-{O}fford problems and the singularity of
              random symmetric matrices},
   JOURNAL = {Duke Math. J.},
  FJOURNAL = {Duke Mathematical Journal},
    VOLUME = {161},
      YEAR = {2012},
    NUMBER = {4},
     PAGES = {545--586},
      ISSN = {0012-7094,1547-7398},
   MRCLASS = {11B30 (11M50 60B20)},
  MRNUMBER = {2891529},
MRREVIEWER = {Michael\ Stolz},
       DOI = {10.1215/00127094-1548344},
       URL = {https://doi.org/10.1215/00127094-1548344},
}

@article {Nguyen_ejp,
    AUTHOR = {Nguyen, Hoi H.},
     TITLE = {On the least singular value of random symmetric matrices},
   JOURNAL = {Electron. J. Probab.},
  FJOURNAL = {Electronic Journal of Probability},
    VOLUME = {17},
      YEAR = {2012},
     PAGES = {no. 53, 19},
      ISSN = {1083-6489},
   MRCLASS = {60B20},
  MRNUMBER = {2955045},
MRREVIEWER = {Anna\ Lytova},
       DOI = {10.1214/EJP.v17-2165},
       URL = {https://doi.org/10.1214/EJP.v17-2165},
}

@article {Nguyen_ptrf,
    AUTHOR = {Nguyen, Hoi and Tao, Terence and Vu, Van},
     TITLE = {Random matrices: tail bounds for gaps between eigenvalues},
   JOURNAL = {Probab. Theory Related Fields},
  FJOURNAL = {Probability Theory and Related Fields},
    VOLUME = {167},
      YEAR = {2017},
    NUMBER = {3-4},
     PAGES = {777--816},
      ISSN = {0178-8051,1432-2064},
   MRCLASS = {15B52 (11B25 11M50 15A63 60B20)},
  MRNUMBER = {3627428},
MRREVIEWER = {Steven\ Joel\ Miller},
       DOI = {10.1007/s00440-016-0693-5},
       URL = {https://doi.org/10.1007/s00440-016-0693-5},
}

@article {Nguyen_jfa,
    AUTHOR = {Nguyen, Hoi H.},
     TITLE = {Random matrices: overcrowding estimates for the spectrum},
   JOURNAL = {J. Funct. Anal.},
  FJOURNAL = {Journal of Functional Analysis},
    VOLUME = {275},
      YEAR = {2018},
    NUMBER = {8},
     PAGES = {2197--2224},
      ISSN = {0022-1236,1096-0783},
   MRCLASS = {15B52},
  MRNUMBER = {3841540},
MRREVIEWER = {Beno\^it\ Collins},
       DOI = {10.1016/j.jfa.2018.06.010},
       URL = {https://doi.org/10.1016/j.jfa.2018.06.010},
}

@article {Rudelsonadvance,
    AUTHOR = {Rudelson, Mark and Vershynin, Roman},
     TITLE = {The {L}ittlewood-{O}fford problem and invertibility of random
              matrices},
   JOURNAL = {Adv. Math.},
  FJOURNAL = {Advances in Mathematics},
    VOLUME = {218},
      YEAR = {2008},
    NUMBER = {2},
     PAGES = {600--633},
      ISSN = {0001-8708,1090-2082},
   MRCLASS = {60E15 (60B20)},
  MRNUMBER = {2407948},
MRREVIEWER = {Ben\ Joseph\ Green},
       DOI = {10.1016/j.aim.2008.01.010},
       URL = {https://doi.org/10.1016/j.aim.2008.01.010},
}

@article {Rudelson_cpam,
    AUTHOR = {Rudelson, Mark and Vershynin, Roman},
     TITLE = {Smallest singular value of a random rectangular matrix},
   JOURNAL = {Comm. Pure Appl. Math.},
  FJOURNAL = {Communications on Pure and Applied Mathematics},
    VOLUME = {62},
      YEAR = {2009},
    NUMBER = {12},
     PAGES = {1707--1739},
      ISSN = {0010-3640,1097-0312},
   MRCLASS = {60B20 (15A42 15B52 60E15)},
  MRNUMBER = {2569075},
MRREVIEWER = {Mark\ W.\ Meckes},
       DOI = {10.1002/cpa.20294},
       URL = {https://doi.org/10.1002/cpa.20294},
}

@article {Rudelson_HW,
    AUTHOR = {Rudelson, Mark and Vershynin, Roman},
     TITLE = {Hanson-{W}right inequality and sub-{G}aussian concentration},
   JOURNAL = {Electron. Commun. Probab.},
  FJOURNAL = {Electronic Communications in Probability},
    VOLUME = {18},
      YEAR = {2013},
     PAGES = {no. 82, 9},
      ISSN = {1083-589X},
   MRCLASS = {60F10},
  MRNUMBER = {3125258},
MRREVIEWER = {Xiequan\ Fan},
       DOI = {10.1214/ECP.v18-2865},
       URL = {https://doi.org/10.1214/ECP.v18-2865},
}

@article {Rudelson_imrn,
    AUTHOR = {Rudelson, Mark and Vershynin, Roman},
     TITLE = {Small ball probabilities for linear images of high-dimensional
              distributions},
   JOURNAL = {Int. Math. Res. Not. IMRN},
  FJOURNAL = {International Mathematics Research Notices. IMRN},
      YEAR = {2015},
    NUMBER = {19},
     PAGES = {9594--9617},
      ISSN = {1073-7928,1687-0247},
   MRCLASS = {60E15},
  MRNUMBER = {3431603},
MRREVIEWER = {Mikhail\ A.\ Lifshits},
       DOI = {10.1093/imrn/rnu243},
       URL = {https://doi.org/10.1093/imrn/rnu243},
}

@article {Rudelson_gafa,
    AUTHOR = {Rudelson, Mark and Vershynin, Roman},
     TITLE = {No-gaps delocalization for general random matrices},
   JOURNAL = {Geom. Funct. Anal.},
  FJOURNAL = {Geometric and Functional Analysis},
    VOLUME = {26},
      YEAR = {2016},
    NUMBER = {6},
     PAGES = {1716--1776},
      ISSN = {1016-443X,1420-8970},
   MRCLASS = {15B52 (60B20)},
  MRNUMBER = {3579707},
MRREVIEWER = {Walid\ Hachem},
       DOI = {10.1007/s00039-016-0389-0},
       URL = {https://doi.org/10.1007/s00039-016-0389-0},
}

@article {Rudelson_aop,
    AUTHOR = {Rudelson, Mark},
     TITLE = {A large deviation inequality for the rank of a random matrix},
   JOURNAL = {Ann. Probab.},
  FJOURNAL = {The Annals of Probability},
    VOLUME = {52},
      YEAR = {2024},
    NUMBER = {5},
     PAGES = {1992--2018},
      ISSN = {0091-1798,2168-894X},
   MRCLASS = {60B20 (60E15 60F10)},
  MRNUMBER = {4791424},
MRREVIEWER = {Vladislav\ Kargin},
       DOI = {10.1214/24-aop1695},
       URL = {https://doi.org/10.1214/24-aop1695},
}

@article {SSS_gafa,
    AUTHOR = {Sah, Ashwin and Sahasrabudhe, Julian and Sawhney, Mehtaab},
     TITLE = {On the {S}pielman-{T}eng conjecture},
   JOURNAL = {Geom. Funct. Anal.},
  FJOURNAL = {Geometric and Functional Analysis},
    VOLUME = {35},
      YEAR = {2025},
    NUMBER = {2},
     PAGES = {633--671},
      ISSN = {1016-443X,1420-8970},
   MRCLASS = {60B20 (15B51)},
  MRNUMBER = {4880206},
MRREVIEWER = {Wangjun\ Yuan},
       DOI = {10.1007/s00039-025-00707-z},
       URL = {https://doi.org/10.1007/s00039-025-00707-z},
}

@inproceedings {Sconjecture,
    AUTHOR = {Spielman, Daniel A. and Teng, Shang-Hua},
     TITLE = {Smoothed analysis of algorithms},
 BOOKTITLE = {Proceedings of the {I}nternational {C}ongress of
              {M}athematicians, {V}ol. {I} ({B}eijing, 2002)},
     PAGES = {597--606},
 PUBLISHER = {Higher Ed. Press, Beijing},
      YEAR = {2002},
      ISBN = {7-04-008690-5},
   MRCLASS = {90C60 (68Q25 68W40 90C51)},
  MRNUMBER = {1989210},
MRREVIEWER = {K.\ G.\ Murty},
}

@article {Tao_rsa,
    AUTHOR = {Tao, Terence and Vu, Van},
     TITLE = {On random {$\pm1$} matrices: singularity and determinant},
   JOURNAL = {Random Structures Algorithms},
  FJOURNAL = {Random Structures \& Algorithms},
    VOLUME = {28},
      YEAR = {2006},
    NUMBER = {1},
     PAGES = {1--23},
      ISSN = {1042-9832,1098-2418},
   MRCLASS = {15A52 (60C05)},
  MRNUMBER = {2187480},
MRREVIEWER = {Hsien-Kuei\ Hwang},
       DOI = {10.1002/rsa.20109},
       URL = {https://doi.org/10.1002/rsa.20109},
}

@book {Tao_ac,
    AUTHOR = {Tao, Terence and Vu, Van},
     TITLE = {Additive combinatorics},
    SERIES = {Cambridge Studies in Advanced Mathematics},
    VOLUME = {105},
 PUBLISHER = {Cambridge University Press, Cambridge},
      YEAR = {2006},
     PAGES = {xviii+512},
      ISBN = {978-0-521-85386-6; 0-521-85386-9},
   MRCLASS = {11-02 (05-02 05D10 11B13 11P70 11P82 28D05 37A45)},
  MRNUMBER = {2289012},
MRREVIEWER = {Serge\u i\ V.\ Konyagin and Ilya\ D.\ Shkredov},
       DOI = {10.1017/CBO9780511755149},
       URL = {https://doi.org/10.1017/CBO9780511755149},
}

@article {Tao_jams,
    AUTHOR = {Tao, Terence and Vu, Van},
     TITLE = {On the singularity probability of random {B}ernoulli matrices},
   JOURNAL = {J. Amer. Math. Soc.},
  FJOURNAL = {Journal of the American Mathematical Society},
    VOLUME = {20},
      YEAR = {2007},
    NUMBER = {3},
     PAGES = {603--628},
      ISSN = {0894-0347,1088-6834},
   MRCLASS = {60C05 (15A52)},
  MRNUMBER = {2291914},
MRREVIEWER = {Ben\ Joseph\ Green},
       DOI = {10.1090/S0894-0347-07-00555-3},
       URL = {https://doi.org/10.1090/S0894-0347-07-00555-3},
}

@article {Tao_gafa,
    AUTHOR = {Tao, Terence and Vu, Van},
     TITLE = {Random matrices: the distribution of the smallest singular
              values},
   JOURNAL = {Geom. Funct. Anal.},
  FJOURNAL = {Geometric and Functional Analysis},
    VOLUME = {20},
      YEAR = {2010},
    NUMBER = {1},
     PAGES = {260--297},
      ISSN = {1016-443X,1420-8970},
   MRCLASS = {60B20},
  MRNUMBER = {2647142},
MRREVIEWER = {Steven\ Joel\ Miller},
       DOI = {10.1007/s00039-010-0057-8},
       URL = {https://doi.org/10.1007/s00039-010-0057-8},
}

@article {Tao_ptrf,
    AUTHOR = {Tao, Terence},
     TITLE = {The asymptotic distribution of a single eigenvalue gap of a
              {W}igner matrix},
   JOURNAL = {Probab. Theory Related Fields},
  FJOURNAL = {Probability Theory and Related Fields},
    VOLUME = {157},
      YEAR = {2013},
    NUMBER = {1-2},
     PAGES = {81--106},
      ISSN = {0178-8051,1432-2064},
   MRCLASS = {60B20},
  MRNUMBER = {3101841},
MRREVIEWER = {Nizar\ Demni},
       DOI = {10.1007/s00440-012-0450-3},
       URL = {https://doi.org/10.1007/s00440-012-0450-3},
}

@article {Tao_comb,
    AUTHOR = {Tao, Terence and Vu, Van},
     TITLE = {Random matrices have simple spectrum},
   JOURNAL = {Combinatorica},
  FJOURNAL = {Combinatorica. An International Journal on Combinatorics and
              the Theory of Computing},
    VOLUME = {37},
      YEAR = {2017},
    NUMBER = {3},
     PAGES = {539--553},
      ISSN = {0209-9683,1439-6912},
   MRCLASS = {05C80 (05C50 60C99)},
  MRNUMBER = {3666791},
MRREVIEWER = {Igor\ Shinkar},
       DOI = {10.1007/s00493-016-3363-4},
       URL = {https://doi.org/10.1007/s00493-016-3363-4},
}

@article {Tikhomirov,
    AUTHOR = {Tikhomirov, Konstantin},
     TITLE = {Singularity of random {B}ernoulli matrices},
   JOURNAL = {Ann. of Math. (2)},
  FJOURNAL = {Annals of Mathematics. Second Series},
    VOLUME = {191},
      YEAR = {2020},
    NUMBER = {2},
     PAGES = {593--634},
      ISSN = {0003-486X,1939-8980},
   MRCLASS = {60B20 (15A18 15B52)},
  MRNUMBER = {4076632},
MRREVIEWER = {Anamaria\ Savu},
       DOI = {10.4007/annals.2020.191.2.6},
       URL = {https://doi.org/10.4007/annals.2020.191.2.6},
}

@article {Vershynin_rsa,
    AUTHOR = {Vershynin, Roman},
     TITLE = {Invertibility of symmetric random matrices},
   JOURNAL = {Random Structures Algorithms},
  FJOURNAL = {Random Structures \& Algorithms},
    VOLUME = {44},
      YEAR = {2014},
    NUMBER = {2},
     PAGES = {135--182},
      ISSN = {1042-9832,1098-2418},
   MRCLASS = {60B20 (15B52)},
  MRNUMBER = {3158627},
       DOI = {10.1002/rsa.20429},
       URL = {https://doi.org/10.1002/rsa.20429},
}

@article {Vu_crmt,
    AUTHOR = {Vu, Van H.},
     TITLE = {Recent progress in combinatorial random matrix theory},
   JOURNAL = {Probab. Surv.},
  FJOURNAL = {Probability Surveys},
    VOLUME = {18},
      YEAR = {2021},
     PAGES = {179--200},
      ISSN = {1549-5787},
   MRCLASS = {60B20 (05C50 05C80 15B52 60-02)},
  MRNUMBER = {4260513},
MRREVIEWER = {Ke\ Wang},
       DOI = {10.1214/20-ps346},
       URL = {https://doi.org/10.1214/20-ps346},
}

@article {Zhou_ber,
    AUTHOR = {Zhou, Shuheng},
     TITLE = {Sparse {H}anson-{W}right inequalities for subgaussian
              quadratic forms},
   JOURNAL = {Bernoulli},
  FJOURNAL = {Bernoulli. Official Journal of the Bernoulli Society for
              Mathematical Statistics and Probability},
    VOLUME = {25},
      YEAR = {2019},
    NUMBER = {3},
     PAGES = {1603--1639},
      ISSN = {1350-7265,1573-9759},
   MRCLASS = {60E15 (60F10 62H12)},
  MRNUMBER = {3961224},
       DOI = {10.3150/17-BEJ978},
       URL = {https://doi.org/10.3150/17-BEJ978},
}

@article {AB_aop,
    AUTHOR = {Ben Arous, G\'erard and Bourgade, Paul},
     TITLE = {Extreme gaps between eigenvalues of random matrices},
   JOURNAL = {Ann. Probab.},
  FJOURNAL = {The Annals of Probability},
    VOLUME = {41},
      YEAR = {2013},
    NUMBER = {4},
     PAGES = {2648--2681},
      ISSN = {0091-1798,2168-894X},
   MRCLASS = {60B20 (11M50 15B52)},
  MRNUMBER = {3112927},
MRREVIEWER = {Florent\ Benaych-Georges},
       DOI = {10.1214/11-AOP710},
       URL = {https://doi.org/10.1214/11-AOP710},
}

@article {Bour_jems,
    AUTHOR = {Bourgade, Paul},
     TITLE = {Extreme gaps between eigenvalues of {W}igner matrices},
   JOURNAL = {J. Eur. Math. Soc. (JEMS)},
  FJOURNAL = {Journal of the European Mathematical Society (JEMS)},
    VOLUME = {24},
      YEAR = {2022},
    NUMBER = {8},
     PAGES = {2823--2873},
      ISSN = {1435-9855,1435-9863},
   MRCLASS = {60B20},
  MRNUMBER = {4416591},
       DOI = {10.4171/jems/1141},
       URL = {https://doi.org/10.4171/jems/1141},
}

@article {EY_jems,
    AUTHOR = {Erd\H os, L\'aszl\'o{} and Yau, Horng-Tzer},
     TITLE = {Gap universality of generalized {W}igner and
              {$\beta$}-ensembles},
   JOURNAL = {J. Eur. Math. Soc. (JEMS)},
  FJOURNAL = {Journal of the European Mathematical Society (JEMS)},
    VOLUME = {17},
      YEAR = {2015},
    NUMBER = {8},
     PAGES = {1927--2036},
      ISSN = {1435-9855,1435-9863},
   MRCLASS = {82B44 (15B52)},
  MRNUMBER = {3372074},
MRREVIEWER = {Sasha\ Sodin},
       DOI = {10.4171/JEMS/548},
       URL = {https://doi.org/10.4171/JEMS/548},
}

@incollection {Tao_fourmoment,
    AUTHOR = {Tao, Terence and Vu, Van},
     TITLE = {Random matrices: the four-moment theorem for {W}igner
              ensembles},
 BOOKTITLE = {Random matrix theory, interacting particle systems, and
              integrable systems},
    SERIES = {Math. Sci. Res. Inst. Publ.},
    VOLUME = {65},
     PAGES = {509--528},
 PUBLISHER = {Cambridge Univ. Press, New York},
      YEAR = {2014},
      ISBN = {978-1-107-07992-2},
   MRCLASS = {60B20 (15B52)},
  MRNUMBER = {3380700},
MRREVIEWER = {Sasha\ Sodin},
}

\appendix
\section{Proof of Proposition \ref{Net size}}\label{Proof of Proposition Net size}

In these appendices, our main goal is to prove Proposition \ref{Net size}. In particular, the proof of Proposition \ref{Net size} is similar to the proof of Theorem 8.1 in \cite{Campos_jams}. We will adapt their method according to our assumptions to ensure that Proposition \ref{Net size} is carried out. We first introduce the following definition.
\begin{mydef}\label{box}
    Define a $(N,\kappa,d)$ box as being a set of the form $\mathcal{B} = B_{1} \times \dots B_{n} \subseteq \mathbb{Z}^{n}$, where $|B_{i}|  \ge N$ for all $i \in [n]$; $B_{i} = [-\kappa N,-N]\cup [N,\kappa N]$, for $i \in [d]$; and $|\mathcal{B}| \le (\kappa N)^{n}$.
\end{mydef}
We now give the following proposition for the size of $(N,\kappa,d)$ box.
\begin{mypropo}\label{Propo_4.1}
    There exist absolute positive constants $c_{\ref{Propo_4.1}}$ and $C_{\ref{Propo_4.1}}$ such that the following holds. For $T\ge 1$, $L  \ge 2$, $\kappa  \ge 2$ and $0<c_{0} < c_{\ref{Propo_4.1}}T^{-4}$, let $n ,d\in \mathbb{N}$ with $d  \ge 2^{7}T^{4}$ and $d\in [c_{0}^{2}n/4,c_{0}^{2}n]$. There exist constants $n_{\ref{Propo_4.1}}=n_{\ref{Propo_4.1}}(\kappa,L,T,c_{0})$ depending on $\kappa$, $T$, $c_{0}$ and $L$, $Q_{\ref{Propo_4.1}}=Q_{\ref{Propo_4.1}}(T,L,c_{0})$ depending on $T$, $L$ and $c_{0}$ such that the following holds. For $N  \ge 2$ satisfying $N < 8^{-1}\exp{\left( Q_{\ref{Propo_4.1}}d \right) }$. Let $\mathcal{B}$ be a $(N,\kappa,d)$-box. If $X$ is chosen uniformly at random from $\mathcal{B}$, then 
    \begin{align}
        \mathbb{P}_{X}\left( \mathbb{P}_{M}\left( \Vert MX\Vert_{2} \le2m \right)  \ge (L/N)^{m} \right) \le(R_{\ref{Propo_4.1}}/L)^{2m}.
    \end{align}
    where $R_{\ref{Propo_4.1}} :=C_{\ref{Propo_4.1}}\kappa c_{0}^{-3} T^{6}$.
\end{mypropo}
We now turn to showing that Proposition \ref{Propo_4.1} implies Proposition \ref{Net size}. For this, we now present a key lemma to approximate the ``discretized grid" by the boxes, which will be introduced in \cite{Campos_jams}.
\begin{mylem}\label{Net}
    For all $\varepsilon \in (0,1)$, $\kappa  \ge\max{\left( \kappa_{1}/\kappa_{0},2^{8}\kappa_{0}^{-4} \right)}$, there exists a family $\mathcal{F}$ of $(N,\kappa,d)$-boxes with $|\mathcal{F}| \le\left( \kappa \right)^{n}$ such that
    \begin{align}
        \Lambda_{\varepsilon} \subseteq \bigcup_{\mathcal{B} \in \mathcal{F}}\left( 4\varepsilon n^{-1/2} \right)\cdot \mathcal{B},\nonumber
    \end{align}
    where $N=\kappa_{0}/(4\varepsilon)$.
\end{mylem}

\begin{proof}[\textsf{Proof of the Proposition \ref{Net size} assuming Proposition \ref{Propo_4.1}}]
    We start by defining parameters. We fix $\kappa= \max{\left( \kappa_{1}/\kappa_{0},2^{8}\kappa_{0}^{-4} \right) } $, $L_{0}=\kappa_{0}L/4$ and denote 
    \begin{align}
        c_{\ref{Net size}}=c_{\ref{Propo_4.1}},~~C_{\ref{Net size}}=C_{\ref{Propo_4.1}}^{2}\kappa^{4}T^{12}/\kappa_{0},~~n_{\ref{Net size}}=n_{\ref{Propo_4.1}}(\kappa,T,L_{0},c_{0}),\nonumber
    \end{align}
    and 
    \begin{align}
        Q_{\ref{Net size}}=Q_{\ref{Propo_4.1}}(T,L_{0},c_{0}).\nonumber
    \end{align}
    Applying Lemma \ref{Net} with $\kappa$ and using the fact $\mathcal{N}_{\varepsilon} \subseteq \Lambda_{\varepsilon}$ to write 
    \begin{align}
        \mathcal{N}_{\varepsilon} \subseteq \bigcup_{\mathcal{B} \subseteq \mathcal{F}}{\left(  \left( 4\varepsilon  n^{-1/2} \right) \cdot \mathcal{B} \right) \cap \mathcal{N}_{\varepsilon} }.\nonumber
    \end{align}
    Thus, we have
    \begin{align}
        \left| \mathcal{N}_{\varepsilon} \right| \le\sum_{\mathcal{B} \in \mathcal{F}}{ \left| \left( \left( 4\varepsilon n^{-1/2} \right) \cdot \mathcal{B}\right) \cap \mathcal{N}_{\varepsilon} \right| } \le(\kappa)^{n}\cdot \max_{\mathcal{B} \in \mathcal{F}}{\left| \left( \left( 4\varepsilon n^{-1/2} \right)\cdot  \mathcal{B} \right) \cap \mathcal{N}_{\varepsilon}  \right|}.\nonumber
    \end{align}
    Note that $n < 2m$, we have 
    \begin{align}
        \left| \left( 4\varepsilon n^{-1/2} \cdot \mathcal{B}\right) \cap \mathcal{N}_{\varepsilon}  \right| \le\left| \left\{ X \in \mathcal{B}: \mathbb{P}\left( \Vert MX\Vert_{2} \le2m \right)  \ge\left( L\varepsilon \right)^{m}   \right\}   \right|.\nonumber
    \end{align}
    The right-hand side of the inequality is bouned by
    \begin{align}
        \left( \frac{R_{\ref{Propo_4.1}}}{L_{0}} \right)^{2m}\cdot \left| \mathcal{B} \right| \le\left( \frac{C_{\ref{Propo_4.1}}\kappa T^{6}}{c_{0}^{3}L_{0}} \right)^{2m} \cdot \left( \kappa N \right)^{n},\nonumber
    \end{align} 
    where $L_{0}:=\kappa_{0}L/4$. Furthermore, we can get 
    \begin{align}
        \left| \mathcal{N}_{\varepsilon} \right| \le \left( \frac{\kappa^{2} \kappa_{0}}{4\varepsilon}\right)^{n} \cdot \left( \frac{C_{\ref{Propo_4.1}}\kappa T^{6}}{c_{0}^{3}\kappa_{0} L} \right)^{2m} \le L^{-2m} \left( \frac{C_{\ref{Net size}}}{c_{0}^{6}\varepsilon} \right)^{n}.\nonumber
    \end{align}
    This completes the proof of Proposition \ref{Net size}.
\end{proof}
Now, we only need to prove Proposition \ref{Propo_4.1}, which is our main goal in the following sections.

\textbf{Overview of the appendix:}
In Appendix \ref{Section High-dimensional inhomogeneous Littlewood-Offord problem}, we provide estimates for the High-dimensional inhomogeneous Littlewood–Offord problem. In Appendix \ref{Mixed rank estimation for the ``zeroed-out" matrix}, we apply the results from Appendix \ref{Section High-dimensional inhomogeneous Littlewood-Offord problem} to derive an anticoncentration estimate for the singular value of the random matrix. Finally, in Appendix \ref{Section Nets for unstructured vectors}, we adapt the approach in \cite{Campos_jams} to prove the inhomogeneous inversion of randomness, thereby completing the final part of the proof of Proposition \ref{Propo_4.1}.

\section{High-dimensional inhomogeneous Littlewood-Offord problem}\label{Section High-dimensional inhomogeneous Littlewood-Offord problem}

In this section, we will introduce the high-dimensional inhomogeneous inverse Littlewood–Offord problem.\par 
We first give some definitions, given a $n \times l$ matrix $W$, define the ``$W-level$'' set for $t  \ge0$ and $ \xi \sim G_{n}(T)$ by 
\begin{align}
      S_{W}^{\xi}(t):= \left \{ \theta \in \mathbb{R}^{l}: \Vert W \theta \Vert_{\xi} \le\sqrt{t}  \right \}.\nonumber
\end{align}
Let $\gamma_{l}(S):= \mathbb{P}(g \in S)$, where $g \sim N(0,(2\pi)^{-1} I_{l})$. Let $W$ be a $2n \times k$ matrix and $Y$ be a $n \times m$ matrix, we define $W_{Y}$ as  
\begin{align}
W_{Y}:=    
\begin{bmatrix}
W & \begin{bmatrix}
    Y \\
    \textbf{0}_{n \times m}
\end{bmatrix} & \begin{bmatrix}
    \textbf{0}_{n \times m}\\
    Y
\end{bmatrix}
\end{bmatrix},\nonumber
\end{align}
where $\textbf{0}_{n \times m}$ is $n \times m$ zero matrix.

We now give the main theorem in this section:
\begin{mytheo}\label{Smallball}
    For $T \ge 1$, $n$, $k \in \mathbb{N}$, $m \in \mathbb{N}^{+}$, $\alpha\in (0,1)$ and $\mu \in (0,1]$. Let $\xi \sim G_{n}(T)$, $\eta=(\xi,\xi')$ where $\xi'$ is an independent copy of $\xi$ and $\tau \in \Phi_{\mu}(2n,\eta)$. Let $Y_{1},\dots,Y_{m} \in \mathbb{R}^{n}$, $Y:=(Y_{1},\dots,Y_{m})$ be a $n \times m$ matrix with $\Vert Y_{i}\Vert_{2}=y_{i}$, and $W$ be a $2n \times k$ matrix with $\Vert W\Vert \le2$ and $\Vert W\Vert_{\mathrm{HS}}  \ge \sqrt{k}/2$. There exist absolute positive constants $d_{\ref{Smallball}}$, $c_{\ref{Smallball}}$ and $C_{\ref{Smallball}}$ such that the following holds. For any $\gamma  \ge d_{\ref{Smallball}}T^{2}\mu^{-1/2}$, let $L:=\gamma \sqrt{m}$. For $0 < c_{0} < c_{\ref{Smallball}}T^{-4}\mu$, if $\mathrm{RlogD}_{L,\alpha}^{\xi}\left( Y \right) > 16\sqrt{m}$, then
    \begin{align}
        \mathcal{L}\left( W_{Y}^{T}\eta,c_{0}^{1/2}\sqrt{k+m} \right) \le \frac{R_{\ref{Smallball}}^{2m}}{\prod_{i=1}^{m}y_{i}^{2}}\cdot e^{-c_{0}k},\nonumber
    \end{align}
    where $R_{\ref{Smallball}}:=C_{\ref{Smallball}}\gamma \alpha^{-1}$.
\end{mytheo}
Following the approach of \cite{Campos_jams,Campos_pi}, we will proceed with the proof in three steps. First, we will provide some probabilistic estimates using Fourier analysis. Second, we present a geometric inequality in the Gaussian space of \cite{HanY1}. Finally, we will complete the proof of Theorem \ref{Smallball}.
\subsection{Fourier approach}\label{Fourier approach}
The maneuver that uses the Fourier approach is derived from Halasz's \cite{Halasz1} and \cite{Halasz2}. To some extent, his conclusions also reflect that the upper bound of the small ball probability is related to a specific ``integer structure'' of $W$. Using the Fourier analysis method in this subsection, we will introduce two lemmas that characterize the relationship between the Gaussian measure and L\'{e}vy functions. To begin with, recall the characteristic function of random variables.
\\
For $\mu \in (0,1]$, $\xi \sim G_{n}(T)$ and $\tau \in \Phi_{\mu}(n,\xi)$, we have
\begin{align}
    \phi_{\tau_{j}}(t) = \mathbb{E}e^{\mathrm{i}2\pi t \tau_{j}}=1-\mu + \mu \mathbb{E}_{\xi}\cos{(2\pi t \overline{\xi})}.\nonumber
\end{align}
Note that the following inequalities
\begin{align}
    1-20\Vert a \Vert_{\mathbb{T}}^{2} \le\cos{(2\pi a)} \le1- \Vert a \Vert_{\mathbb{T}}^{2},\nonumber
\end{align}
where the inequality on the left holds only when $\mu \le1/4$.\\
These imply that for all $\mu \in (0,1]$,
\begin{align}\label{Up}
     \phi_{\tau_{j}}(t) \le\exp{(-\mu \mathbb{E}_{\xi}\Vert t \overline{\xi}_{j} \Vert_{\mathbb{T}}^{2})},
\end{align}
and for all $\mu \in (0,1/4)$,
\begin{align}\label{Down}
    \exp{(-32 \mu \mathbb{E}_{\xi}\Vert t\overline{\xi}_{j} \Vert_{\mathbb{T}}^{2})} \le\phi_{\tau_{j}}(t) .
\end{align}
Now, we first give the upper bound of the L\'{e}vy function by the Gaussian measure.
\begin{mylem}\label{F1}
    Let $\beta >0$, $\nu \in (0,1]$, $W$ be a $2n \times l$ matrix, $\xi \in G_{2n}(T)$ and $\tau \in \Phi_{\nu}(2n,\xi)$. Then there exists $m >0$ such that 
    $$\mathcal{L}(W^{T}\tau,\beta \sqrt{l}) \le2 \exp{(2\beta^{2}l-\nu m/2)}\gamma_{l}(S_{W}^{\xi}(m)).$$
\end{mylem}
\begin{proof}
For all $\omega \in \mathbb{R}^{l}$, we apply Markov's inequality to obtain
    \begin{align}
        \mathbb{P}_{\tau}(\Vert W^{\mathrm{T}}\tau - \omega \Vert_{2} \le\beta \sqrt{l}) \le\exp{(\frac{\pi}{2}\beta^{2}l)}\mathbb{E}_{\tau}\exp{\left(-\frac{\pi\Vert W^{\mathrm{T}}\tau-\omega \Vert_{2}^{2}}{2} \right)}.\nonumber
    \end{align}
    Using the Fourier transform of a Gaussian measure(see \cite{Tao_ac}, p.290), we have
    \begin{equation}
        \begin{aligned}
            \mathbb{E}_{\tau}\exp{\left(-\frac{\pi \Vert W^{\mathrm{T}}\tau-\omega \Vert_{2}^{2}}{2} \right)} & = \int_{ \mathbb{R}^{l}} {e^{-\pi \Vert \theta \Vert_{2}^{2}} \cdot e^{-2 \pi \mathrm{i}\left \langle \omega,\theta \right \rangle}\varphi_{W^{\mathrm{T}}\tau}(\theta)} \mathrm{d}\theta\\
            & = \mathbb{E}_{g}\left[ e^{-2\pi \mathrm{i} \left \langle \omega,g \right \rangle} \varphi_{W^{\mathrm{T}}\tau}(g) \right],\nonumber 
        \end{aligned}
    \end{equation}
    where $g \sim  N(0,(2\pi)^{-1} I_{l})$.\\
    Recall the inequality (\ref{Up}). The characteristic function of \( W^{\mathrm{T}} \tau \) can be bounded by
    \begin{equation}
        \begin{aligned}
            \varphi_{W^{\mathrm{T}}\tau}(g) =\mathbb{E}_{\tau}e^{2\pi \mathrm{i} \left \langle \tau , Wg \right \rangle} 
             = \prod_{j=1}^{2n}{\phi_{\tau_{j}}(\left \langle W_{j},g \right \rangle)} 
            & \le\exp{(-\nu \sum_{j=1}^{2n}{\mathbb{E}_{\xi_{j}}\Vert \left \langle W_{j},g \right \rangle \overline{\xi}_{j}\Vert_{\mathbb{T}}^{2} })}\\
            & = \exp{(-\nu \Vert Wg \Vert_{\xi}^{2})}.\nonumber
        \end{aligned}
    \end{equation}
    It yields the
    \begin{align}
        \mathbb{E}_{\tau}\exp{\left( -\frac{\pi }{2}\Vert W^{\mathrm{T}}\tau - \omega \Vert_{2}^{2} \right)} \le\mathbb{E}_{g}\exp{\left(-\nu \Vert Wg \Vert_{\xi}^{2} \right)}.\nonumber
    \end{align}
    We rewrite the right-hand side of the above inequality as
    \begin{equation}
        \begin{aligned}
            \int_{0}^{1}{\mathbb{P}_{g}\left( \exp{(-\nu \Vert Wg\Vert_{\xi}^{2})} \ge t \right)} \mathrm{d}t
            & = \nu \int_{0}^{\infty}{\mathbb{P}_{g}\left( \Vert Wg\Vert_{\xi}^{2} \le u \right)e^{-\nu u}} \mathrm{d}u\\
            & = \nu \int_{0}^{\infty}{\gamma_{l}\left(S_{W}^{\xi}(u) \right)e^{-\nu u}} \mathrm{d}u.\nonumber
        \end{aligned}
    \end{equation}
    Choosing $m$ to maximize $\gamma_{l}(S_{W}^{\xi}(u))e^{-\nu u/2}$, we have
    \begin{align}
        \mathbb{E}_{g}\exp{(-\nu \Vert Wg\Vert_{\xi}^{2})} \le\nu \gamma_{l}(S_{W}^{\xi}(m))e^{-\nu m/2} \cdot \int_{0}^{\infty}{e^{-\nu u/2}} \mathrm{d}u \le2\gamma_{l}(S_{W}^{\xi}(m))e^{-\nu m/2}.\nonumber
    \end{align}
    Thus,
    \begin{align}
        \mathbb{P}_{\tau}\left( \Vert W\tau-\omega \Vert_{2} \le\beta \sqrt{l} \right) \le2\exp{(2\beta^{2} l -\nu m/2)}\cdot \gamma_{l}\left(S_{W}^{\xi}(m) \right),\nonumber
    \end{align}
    which implies the result.
\end{proof}
We next provide a lower bound estimate on the probability.
\begin{mylem}\label{F2}
    Let $\beta >0$, $\nu \in (0,1/4)$, and $W$ be a $2n \times l$ matrix, let $\xi \in G_{2n}(T)$ and $\tau \in \Phi_{\nu}(2n,\xi)$. Then for all $t  \ge0$, we have 
    \begin{align}
        \gamma_{l}(S_{W}^{\xi}(t))e^{-32\nu t} \le\mathbb{P}_{\tau}\left(\Vert W^{\mathrm{T}} \tau \Vert_{2} \le\beta \sqrt{l}\right) + e^{-\beta^{2}l}.\nonumber
    \end{align}
\end{mylem}
\begin{proof}
    Let $X:=\Vert W^{\mathrm{T}}\tau \Vert_{2}$. Then we can estimate $\mathbb{E}_{X}e^{-\pi X^{2}/2}$ as follows:
    \begin{equation}
    \begin{aligned}
        \mathbb{E}_{X}e^{-\pi X^{2}/2} 
        & \le\mathbb{E}_{X}\left[ \mathbf{1}_{\{ X\le\beta \sqrt{l} \}}e^{-\pi X^{2}/2}\right] +\mathbb{E}_{X}\left[ \mathbf{1}_{\{ X \ge\beta \sqrt{l} \}}e^{-\pi X^{2}/2} \right]\\ 
        & \le\mathbb{P}_{X}(X\le\beta \sqrt{l}) +e^{-\pi \beta^{2}l/2}.\nonumber 
    \end{aligned}
    \end{equation}
 Similarly to the proof of Lemma \ref{F1} and using inequality (\ref{Down}), for all $t  \ge0$, we have
    \begin{equation}
        \begin{aligned}
            \mathbb{E}_{X}e^{-\pi X^{2}/2}  \ge\mathbb{E}_{g}\left[ \exp{\left( -32 \nu \Vert Wg \Vert_{\xi}^{2}\right)} \right] 
            & = 32\nu \int_{0}^{\infty}{\gamma_{l}\left( S_{W}^{\xi}(u) \right)e^{-32 \nu u}} \mathrm{d}u\\
            & = 32\nu \gamma_{l}\left( S_{W}^{\xi}(t)\right)\int_{t}^{\infty}{e^{-32\nu u}} \mathrm{d}u\\
            &  \ge e^{-32\nu t}\cdot \gamma_{l}\left( S_{W}^{\xi}(t) \right).\nonumber
        \end{aligned}
    \end{equation}
    This completes the proof.
\end{proof}

\subsection{Geometric inequality of Gaussian space}\label{Geometric inequality of Gaussian space}
This section presents a geometric inequality in the Gaussian space with respect to high-dimensional slices. This inequality provides an existence result that allows us to assert the existence of points with certain properties. This, in turn, will serve as a powerful tool for proving the converse of Theorem \ref{Smallball} in the subsequent sections. We begin by introducing several definitions.
\par 
For $r,s > 0$, $k,m \in \mathbb{N}$ and $\textbf{y}=(y_{1},\dots,y_{2m}) \in \mathbb{R}^{2m}$, define $\Gamma_{r,s}$ by 
\begin{align}\label{z}
    \Gamma_{r,s} := \left\{ \theta \in \mathbb{R}^{k+2m} : \Vert \theta_{[k]} \Vert_{2} \le r ,\Vert\theta_{[k+1,k+2m]}\Vert_{\infty} \le s \right\},
\end{align}
Set 
\begin{align}
    \Gamma_{r,\textbf{y},s} := \left\{ \theta \in \mathbb{R}^{k+2m} : \Vert \theta_{[k]} \Vert_{2} \le r ,\Vert\theta_{[k+1,k+2m]}\Vert_{\textbf{y}} \le s  \right\}.
\end{align}
For a measurable set $S \subseteq \mathbb{R}^{k+2m}$ and $y \in \mathbb{R}^{k+2m}$, on the one hand, we define the set 
\begin{align}\label{slice1}
    F_{y}(S;a,b)=\left\{ \theta_{[k]}=(\theta_{1},\cdots,\theta_{k}) \in \mathbb{R}^{k} ; (\theta_{[k]},a,b) \in S-y \right\},
\end{align}
where $a,b \in \mathbb{R}^{m}$. On the other hand, define:
\begin{align}\label{slice2}
    S(\theta_{[k]}):=\left\{ x \in \mathbb{R}^{2m}:\left( \theta_{[k]},x \right) \in S  \right\}.
\end{align}
We now present the main result of this section, which is introduced by \cite{HanY1}.
\begin{mylem}[Lemma 12.7 in \cite{HanY1}]\label{Geometric}
    For $\textbf{y}\in \mathbb{R}^{2m}$ and $S \subseteq \mathbb{R}^{k+2m}$ be a measurable set. We have for $s > 0$ satisfy
    \begin{align}
        2(\prod_{i=1}^{2m}y_{i}^{-1})s^{2m}(\frac{12}{\sqrt{2m}})^{2m} \left( e^{-k/8} +\max_{a.b.y}{\left( \gamma_{k}\left( F_{y}(S;a,b)-F_{y}(S;a,b) \right) \right)^{1/4}} \right) \le\gamma_{k+2m}(S).\nonumber 
    \end{align}
Then there exists an $x \in S$ such that 
\begin{align}
    (\Gamma_{2\sqrt{k},16} \setminus \Gamma_{2\sqrt{k},\textbf{y},s} + x )\cap S \neq \emptyset.\nonumber
\end{align}
\end{mylem}

\subsection{Proof of Theorem \ref{Smallball}}\label{Proof of Theorem Smallball}
We first provide a simple lemma.
\begin{mylem}\label{S1}
    For any $2n \times l$ matrix $W$ and $t >0$, we have
    \begin{align}
        S_{W}^{\xi}(t)-S_{W}^{\xi}(t) \subseteq S_{W}^{\xi}(4t).\nonumber
    \end{align}
\end{mylem}
\begin{proof}
    Assume that $x,y \in S_{W}^{\xi}(t)$, which means $\Vert x \Vert_{\xi},\Vert y \Vert_{\xi} \le\sqrt{t}$. Then we have 
    \begin{equation}
    \begin{aligned}
        \Vert x-y \Vert_{\xi}^{2} 
        & = \mathbb{E}_{\xi}\mathrm{dist}^{2}\left( \overline{\xi} \star W(x-y),\mathbb{Z}^{2n}\right)\\
        & \le2 \mathbb{E}_{\xi}\mathrm{dist}^{2}\left( \overline{\xi} \star Wx,\mathbb{Z}^{2n}\right) + 2 \mathbb{E}_{\xi}\mathrm{dist}^{2}\left( \overline{\xi} \star Wy,\mathbb{Z}^{2n}\right) \\
        & \le4t.\nonumber
    \end{aligned}
    \end{equation}
This completes the proof.
\end{proof}
 Then, we give a lemma that is close to the contrapositive of the final conclusion.
\begin{mylem}\label{Smallball_1}
    For $T \ge 1$, $n,k \in \mathbb{N}$, $m \in \mathbb{N}^{+}$, $\alpha \in (0,1)$, $\mu \in (0,1]$ and $0 \le c_{0} \le1$. Let $\xi \sim G_{n}(T)$, $\eta=(\xi,\xi')$ where $\xi'$ is an independent copy of $\xi$, $\tau \in \Phi_{\mu}(2n,\eta)$ and $\tau' \sim \Phi_{\nu}(2n,\eta)$ with $\nu = 2^{-7}\mu$. Moreover, let $\gamma \ge 8/\sqrt{\mu}$, $\beta \in[8\gamma^{-1},1]$ and $\beta' \in (0,1/2)$. Let $Y_{1},\dots,Y_{m} \in \mathbb{R}^{n}$, $Y:=(Y_{1},\dots,Y_{m})$ be a $n \times m$ matrix with $y_{i}=\Vert Y_{i}\Vert_{2}$, and $W$ be a $2n \times k$ matrix with $\Vert W\Vert \le2$. There exists an absolute positive constant $C_{\ref{Smallball_1}}$ such that the following holds. If
    \begin{align}\label{ineq_1}
        \mathcal{L}\left( W_{Y}^{\mathrm{T}}\tau, \beta \sqrt{k+1} \right)  \ge\left( \frac{R_{\ref{Smallball_1}}}{G(Y^{T})} \right)^{2m}\exp{(4\beta^{2}k)}\left( U_{k}(\tau',\beta') \right)^{1/4}.
    \end{align}
    Then $$\mathrm{RlogD}_{L,\alpha}^{\xi}(Y) \le16\sqrt{m},$$
    where  $U_{k}(\tau',\beta'):=\mathbb{P}(\Vert W^{\mathrm{T}}\tau' \Vert_{2} \le\beta' \sqrt{k}) + \exp{(-\beta'^{2}k)} $ and $R_{\ref{Smallball_1}}:= C_{\ref{Smallball_1}} \gamma \alpha^{-1}$.
\end{mylem}
\begin{proof}
    Recall the definition of $S_{W}^{\xi}(t)$ and Lemma \ref{F1}, there exists $t$ such that for $S = S_{W_{Y}}^{\eta}(t)$, we have  
    \begin{align}
        \mathcal{L}\left( W_{Y}^{\mathrm{T}}\tau, \beta \sqrt{k+m} \right) \le2e^{2\beta^{2}(k+2m)-\mu t/2}\gamma_{k+2m}(S) \le e^{8m}e^{-\mu t /2+2\beta^{2}k} \gamma_{k+2m}(S).\nonumber
    \end{align}
    Assuming that the inequality \eqref{ineq_1} occurs, we have 
    \begin{align}\label{Smallball_1.1}
        \gamma_{k+2m}(S)  \ge e^{-8m}e^{\mu t/2 + 2\beta^{2}k}\frac{\left(R_{\ref{Smallball_1}}\right)^{2m}}{\prod_{i=1}^{m}y_{i}^{2}} \left( U_{k}(\tau',\beta') \right)^{1/4}.
    \end{align}
    Define 
    \begin{align}\label{Smallball_1.2}
        r_{0} := \sqrt{k} \quad and \quad s_{0} := \gamma\alpha^{-1}\exp{\left( \frac{8t+32k}{\gamma^{2}m} \right)}\sqrt{m}, 
    \end{align}
    Firstly, we have the following claim.
    \begin{mycla}\label{Scla1}
        There exists $x \in S \subseteq \mathbb{R}^{k+2m}$,  satisfying
        \begin{align}
            (\Gamma_{2r_{0},16} \setminus \Gamma_{2r_{0},\textbf{y},s_{0}} + x ) \cap S \neq \emptyset,
        \end{align}
        which $\textbf{y}:=(y_{1},\dots,y_{m},y_{1},\dots,y_{m})$ and $y_{i}=\Vert Y_{i}\Vert_{2}$.
    \end{mycla}
    \begin{proof}
        By Lemma \ref{S1}, 
        \begin{align}
            S_{W}^{\xi}(t)-S_{W}^{\xi}(t) \subseteq S_{W}^{\xi}(4t).\nonumber
        \end{align}
        It implies that 
        \begin{align}
            F_{y}(S;a,b)-F_{y}(S;a,b) \subseteq F_{0}(S_{W_{Y}}^{\eta}(4t);0,0) = S_{W}^{\eta}(4t).
        \end{align}
        Applying Lemma \ref{F2} with $\nu$ and $\tau'$ yields that 
        \begin{align}
            \gamma_{k}(S_{W}^{\eta}(4t)) \le e^{128\nu t}U_{k}(\tau',\beta').\nonumber
        \end{align}
        Let $M:=\max_{a,b,y}{\gamma_{k}\left( F_{y}(S;a,b)-F_{y}(S;a,b) \right)}$, we have  
        \begin{align}
           M \le\gamma_{k}\left(S_{W}^{\eta}(4t) \right) \le e^{128\nu t}U_{k}\left( \tau',\beta' \right).\nonumber
        \end{align}
        Otherwise, we have $U_{k}\left( \tau',\beta' \right)^{1/4}  \ge e^{-k/8}$. Combining the two inequalities above, we can obtain
        \begin{align}
             2(\prod_{i=1}^{m}y_{i}^{-2})(\frac{12}{\sqrt{2m}})^{2m}s_{0}^{2m}\left( e^{-k/8} +M^{1/4} \right)
             & \le \frac{ \left( C\gamma/\alpha \right)^{2m} }{\prod_{i=1}^{m}y_{i}^{2}}  e^{\frac{16t+64k}{\gamma^{2}}+32\nu t} \cdot U_{k}(\tau',\beta')^{1/4}\nonumber \\
             & \le \frac{(C\gamma /\alpha)^{2m}}{\prod_{i=1}^{m}y_{i}^{2}}e^{\mu t/2+2\beta^{2}k}U_{k}(\tau',\beta')^{1/4} \\
             & \le \gamma_{k+2m}(S),\nonumber
        \end{align}
        which using $16\gamma^{-2}+32\nu \le \mu /2$ and $64 \gamma^{2} \le 2 \beta^{2}$.\par 
        Thus, applying Lemma \ref{Geometric}, there exists $x \in S$
        \begin{align}
            (\Gamma_{2r_{0},16} \setminus \Gamma_{2r_{0},\textbf{y} ,s_{0}} + x)\cap S \neq \emptyset.\nonumber
        \end{align}
    \end{proof}
    We note that claim \ref{Scla1} can be rewritten to
    \begin{align}
        S_{W_{Y}}^{\eta}(4t) \cap (\Gamma_{2r_{0},16} \setminus \Gamma_{2r_{0},\textbf{y} ,s_{0}} ) \ne \emptyset.\nonumber
    \end{align}
    Then, we make the following claim and complete the proof of Lemma \ref{Smallball_1} by this claim.
    \begin{mycla}\label{Scla2}
        If $\phi \in S_{W}^{\eta}(4t) \cap (\Gamma_{2r_{0},16} \setminus \Gamma_{2r_{0},\textbf{y},s_{0}})$, then there exists $\theta \in \mathbb{R}^{m}$ satisfy $$\theta \in \left\{\phi_{[k+1,k+m]},\phi_{[k+m+1,2m]}\right\} $$ so that
        \begin{align}
            \Vert  Y\theta \Vert_{\xi} < L\sqrt{\log_{+}{\frac{\alpha \Vert \theta \Vert_{Y}}{L}}}.\nonumber
        \end{align}
    \end{mycla}
    \begin{proof}
        Let $\theta$ be a vector within its feasible range that satisfies $\Vert \theta \Vert_{\infty} \le16$ and $\Vert \theta \Vert_{Y}\ge s_{0}$. According to $\phi \in S_{W_{Y}}^{\eta}(4t)\cap (\Gamma_{2r_{0},16} \setminus \Gamma_{2r_{0},\textbf{y},s_{0}})$, we know that such $\theta$ is attainable.
        Consider $\phi \in S_{W_{Y}}^{\eta}(4t)$, we have
        \begin{align}
            \mathbb{E}_{\eta}\mathrm{dist}^{2}(\overline{\eta} \star W_{Y}\phi ,\mathbb{Z}^{2n}) \le4t.
        \end{align}
        Furthermore, note that 
        $$W_{Y} \phi = W \phi_{[k]} + \begin{bmatrix}
            Y \\
            0_{n}
        \end{bmatrix}\phi_{[k+1,k+m]} +  \begin{bmatrix}
            0_{n}\\
            Y
        \end{bmatrix}\phi_{[k+m+1,k+2m]}.$$
        It yields the
        \begin{eqnarray*}
                & &\mathbb{E}_{\xi}\mathrm{dist}^{2}(\overline{\xi} \star Y\theta,\mathbb{Z}^{2n}) 
                 \le\mathbb{E}_{\eta}\mathrm{dist}^{2}(\overline{\eta} \star (W_{Y}\phi - W \phi_{[k]}),\mathbb{Z}^{2n})\\
                &\le&2\mathbb{E}_{\eta}\Vert \overline{\eta} \star W_{Y}\phi\Vert_{\mathbb{T}}^{2} + 2 \mathbb{E}_{\eta}\Vert \overline{\eta} \star W \phi_{[k]} \Vert_{2}^{2} \\
                & \le&8t + 2\Vert W \Vert_{2}^{2}\Vert \phi_{[k]}\Vert_{2}^{2}\\
                & \le& 8t + 32k.\nonumber
       \end{eqnarray*}
        Note that \begin{align}
            L^{2}\log_{+}{\frac{\alpha \Vert \theta \Vert_{Y}}{L}}  \ge L^{2}\log_{+}{\frac{\alpha s_{0}}{L}}  \ge 8t+32k  \ge\Vert Y\theta\Vert_{\xi}^{2}.\nonumber
        \end{align}
        Thus, we complete the proof of this claim.
    \end{proof}
     We now prove claim \ref{Scla2}, which leads to the proof of this lemma.
\end{proof}
 Next we introduce a sparse Hanson-Wright inequality to estimate the size of $U_{k}(\tau',\beta')$.
\begin{mylem}[Sparse Hanson-Wright]\label{HS}
     Let $\nu \in (0,1/4)$, $\tau' \in \Phi_{\nu}(2d,\eta)$, $\beta' \in (0, 2^{-2}\sqrt{\nu})$, and $W$ be a $2d \times k$ matrix with $\Vert W \Vert_{\mathrm{HS}}  \ge\sqrt{k}/2$, $\Vert W \Vert \le2$. Then
    \begin{align}
        \mathbb{P}(\Vert W^{\mathrm{T}} \tau' \Vert_{2} \le\beta' \sqrt{k}) \le2 \exp{(-c_{\ref{HS}}T^{-4}\nu k)},\nonumber
    \end{align}
    where $c_{\ref{HS}}$ is an absolute positive constant.
\end{mylem}
\begin{myrem}
    A standard form of the Hanson-Wright inequality can be found in \cite{Rudelson_HW}. For the Hanson-Wright inequality regarding sparse sub-Gaussian random variables, refer to \cite{Zhou_ber}. More generally, for a sparse version of the Hanson-Wright inequality applicable to more general random variables, consult \cite{Wangke}.
\end{myrem}
\begin{proof}
    Applying Theorem 1 in \cite{Zhou_ber} for $B=WW^{T}=(b_{ij})_{2d \times 2d}$ and $\tau'$, we have 
    \begin{align}
        \mathbb{P}\left( \left| \Vert W^{\mathrm{T}}\tau'\Vert_{2}^{2}-\nu\Vert W\Vert_{\mathrm{HS}}^{2}\right| > t \right) \le2\exp{\left( -M(t,W) \right) },\nonumber
    \end{align}
    where $M(t,W):=c\min{\left(\frac{t^{2}}{T^{4}(\nu\sum_{i}b_{ii}^{2}+\nu^{2}\sum_{i \ne j}b_{ij}^{2}},\frac{t}{T^{2}\Vert B\Vert}  \right)}$.
    We can bound 
    \begin{align}
        \nu \sum_{i}b_{ii}^{2}+\nu \sum_{i \ne j}b_{ij}^{2}=\nu \Vert W\Vert_{\mathrm{HS}}^{2}+\nu^{2}(\Vert WW^{T}\Vert_{\mathrm{HS}}^{2}-\Vert W\Vert_{\mathrm{HS}}^{2}) \le (\nu+3\nu^{2})\Vert W\Vert_{\mathrm{HS}}^{2},\nonumber
    \end{align}
    where the last inequality follows from $\Vert W\Vert \le 2$.
    
    Let $t = \nu\Vert W\Vert_{\mathrm{HS}}^{2}/2$, we immediately have 
    \begin{align}
        M(t,W) \ge c\min{ \left(  \frac{\nu\Vert W\Vert_{\mathrm{HS}}^{2}}{4T^{4}(1+3\nu)},\frac{\nu\Vert W\Vert_{\mathrm{HS}}^{2}}{T^{2}\Vert W\Vert^{2}} \right)} \ge cT^{-4}\nu k,\nonumber
    \end{align}
    which implies this result.
\end{proof}
We can now provide the proof for Theorem \ref{Smallball}.
\begin{proof}[\textsf{Proof of Theorem \ref{Smallball}}]
    Let $\beta=2^{-10}c_{\ref{HS}}T^{-2}\sqrt{\mu}$, $\beta'=8\beta$ and 
    \begin{align}
        d_{\ref{Smallball}}=2^{13}c_{\ref{HS}}^{-1},~~c_{\ref{Smallball}}=2^{-20}c_{\ref{HS}}^{2},~~C_{\ref{Smallball}}=3C_{\ref{Smallball_1}}.\nonumber
    \end{align}
    Applying Lemma \ref{HS} with $\nu=2^{-7}\mu$, we have 
    \begin{align}
        U_{k}\left( \tau',\beta' \right) \le2\exp{\left( -2^{-7}c_{\ref{HS}}T^{-4}\mu k \right)} + \exp{(-\beta'^{2}k)} \le 3e^{-64\beta^{2} k},\nonumber
    \end{align}
   Note that for $\gamma \ge d_{\ref{Smallball}}T^{2}\mu^{-1/2}$ and $0 < c_{0} < c_{\ref{Smallball}}T^{-4}\mu $, we have $\beta 8 \gamma^{-1}$ and $\beta^{2} \ge c_{0}$.
   By Lemma \ref{Smallball_1}, we have 
   \begin{align}
       \mathcal{L}\left( W_{Y}^{T}\tau,c_{0}^{1/2}\sqrt{k+m} \right) \le  3^{1/4}e^{-12\beta^{2}k}\frac{R_{\ref{Smallball_1}}^{2m}}{\prod_{i=1}^{m}y_{i}^{2}} \le e^{-c_{0}k}\frac{R_{\ref{Smallball}}^{2m}}{\prod_{i=1}^{m}y_{i}^{2}},\nonumber
   \end{align}
   which implies the result.
\end{proof}
\section{Mixed rank estimation for the ``zeroed-out" matrix}\label{Mixed rank estimation for the ``zeroed-out" matrix}

This section presents the inhomogeneous ``Mixed rank estimation" for random matrices. We need to compute the mixed probability of two events: the $k$-th smallest singular value of an inhomogeneous $N\times n$ random matrix $H$ being small, and the Euclidean norm of the product of this matrix $H$ and a specific vector with a large log-RLCD being small. First, we will introduce the one-dimensional corollary of Theorem \ref{Smallball}.
\begin{mycoro}\label{coro4.1}
    For $T \ge 1$, $n,k \in \mathbb{N}$, $\alpha \in (0,1)$ and $\mu \in (0,1]$. Let $\xi \sim G_{n}(T)$, $\eta = (\xi,\xi')$ and $\tau \in \Phi_{\mu}(2n,\eta)$. Let $Y \in \mathbb{R}^{n}$ and $W$ be a $2n\times k$ matrix with $\Vert W \Vert \le 2$ and $\Vert W\Vert_{\mathrm{HS}}  \ge \sqrt{k}/2$. There exist absolute positive constants $d_{\ref{coro4.1}}$, $c_{\ref{coro4.1}}$ and $C_{\ref{coro4.1}}$ such that the following holds. For any $\gamma  \ge d_{\ref{coro4.1}}T^{2}\mu^{-1/2}$ and $0< c_{0} \le c_{\ref{coro4.1}}T^{-4}\mu$, if $\mathrm{RlogD}_{\gamma,\alpha}^{\xi}(Y) >16$, then
    \begin{align}
        \mathcal{L}\left( W_{Y}^{T}\tau,c_{0}^{1/2}\sqrt{k+1} \right) \le\left( R_{\ref{coro4.1}}t \right)^{2}e^{-c_{0}k},\nonumber
    \end{align}
    where $R_{\ref{coro4.1}}:=C_{\ref{coro4.1}}\gamma \alpha^{-1}$ and $t=\Vert Y\Vert_{2}^{-1}$.
\end{mycoro}
We can now introduce the main result in this section, which extends the rank estimation of Campos, Jenssen, Michelen, and Sahasrabudhe \cite{Campos_jams, Campos_pi} to the case of different distributions.
\begin{mypropo}[Mixed rank estimation via log-RLCD]\label{propo4.1}
    For $T\ge 1$, $t >0$, $n,m,k,d\in \mathbb{N}$ with $n  \ge m \ge d \ge  k$, $\alpha \in (0,1)$ and $\nu \le 1/4$. Let $X \in \mathbb{R}^{d}$ satisfy $\Vert X\Vert_{2} \ge 32\sqrt{n}c_{0}^{-1}t^{-1}$ and $H$ be a random $(m-d) \times 2d$ matrix with independent rows satisfying $\mathrm{Row}_{j}(H) \in \Phi_{\nu}(2d,(\xi_{j},\xi_{j}'))$ for all $j\in [m-d]$, where $\xi_{1},\dots,\xi_{m-d} \in G_{d}(T)$ are independent random variables. There exist absolute positive constants $c_{\ref{propo4.1}}$, $c_{\ref{propo4.1}}'$, $d_{\ref{propo4.1}}$ and $C_{\ref{propo4.1}}$ such that the following holds. For $0 < c_{0} \le c_{\ref{propo4.1}}T^{-4} \nu$, $\gamma \ge d_{\ref{propo4.1}}T^{2}\mu^{-1/2}$, $d \le 2c_{0}^{2}m$ and $t  \ge c_{\ref{propo4.1}}'\alpha\gamma^{-1}\exp{(-2^{-3} d)}$, if $\mathrm{RlogD}_{\gamma,\alpha}^{\xi_{j}}(r_{n}X) >16$, then 
    \begin{align*}
        \mathbb{P}_{H}\left(  \sigma_{2d-k+1}(H) \le \frac{c_{0}\sqrt{m}}{16}  \text{ and } \Vert H_{1}X\Vert_{2}, \Vert H_{2}X\Vert_{2} \le m \right) \le e^{-\frac{c_{0}nk}{4}}\cdot (R_{\ref{propo4.1}}t)^{2m-2d},
    \end{align*}
    where $H_{1}:=H_{[m-d]\times[d]}$, $H_{2}:=H_{[m-d]\times[d+1,2d]}$, $r_{n}:= \frac{c_{0}}{32\sqrt{n}}$ and $R_{\ref{propo4.1}} := C_{\ref{propo4.1}}\gamma/\alpha$.
\end{mypropo}
We now give some auxiliary lemmas from \cite{Campos_jams}, and then complete the proof of the main result of this section. 
\begin{mylem}[Lemma 7.3 in \cite{Campos_jams}]\label{Lem4.1}
    For $d < m$ and $k  \ge0$, let $W$ be an $2d \times (k+2)$ matrix and $H$ be an $(m-d)\times 2d$ random matrix with independent rows from $\Phi_{\nu}(2d,\eta_{j})$, $j \in [m-d]$. Let $\tau_{j} = \mathrm{Row}_{j}(H)$. If $\beta \in (0,1/8)$, then
    \begin{align}
        \mathbb{P}_{H}\left( \Vert HW \Vert_{HS}  \le\beta^{2}\sqrt{(k+1)(m-d)} \right) \le\left( 2^{5}e^{2\beta^{2}k}\right)^{m-d} \prod_{j=1}^{m-d}{\mathcal{L}(W^{\mathrm{T}}\tau_{j},\beta\sqrt{k+1})}.\nonumber
    \end{align}
\end{mylem}
\begin{mylem}
    [Lemma 7.4 in \cite{Campos_jams}]\label{Lem4.2} For $k \le d$ and $\delta \in (0,1/2)$, there exists $W \in \mathcal{W}_{2d,k} \subset \mathbb{R}^{[2d]\times [k]}$ with $|\mathcal{W}| \le (2^{6}/\delta)^{2dk}$ so that for any $U \in \mathcal{U}_{2d,k}$, any $r \in \mathbb{N}$, and $r \times 2d$ matrix $A$, there exists $W \in \mathcal{W}$, satisfying 
\begin{enumerate}
         \item $\Vert A(W-U)\Vert_{\mathrm{HS}} \le\delta(k/2d)^{1/2} \cdot \Vert A\Vert_{\mathrm{HS}}$;
         \item $\Vert W-U \Vert_{\mathrm{HS}} \le\delta \sqrt{k}$;
         \item $\Vert W-U\Vert \le8\delta$.  
     \end{enumerate}  
\end{mylem}
\begin{mylem}[Fact 7.5 in  \cite{Campos_jams}]\label{Lem4.3}
    For $3d < m$, let $H$ be an $(m-d) \times 2d$ matrix. If $\sigma_{2d-k+1}(H) \le x$, then there exist $k$ orthogonal unit vectors $w_{1},\dots,w_{k} \in \mathbb{R}^{2d}$,  such that $\Vert Hw_{i}\Vert_{2} \le x$. In particular, there exists $U \in \mathcal{U}_{2d,k}$ so that $\Vert HU \Vert_{\mathrm{HS}} \le x\sqrt{k}$.
\end{mylem}
We also have the following simple inequality for the ``Hilbert-Schmidt norm" of random matrices. 
\begin{mylem}\label{Lem4.4}
    Let $H$ be the $(m-d) \times 2d$ random matrix whose rows from $\Phi_{\nu}(2d,\xi_{j})$, $\xi_{j} \in G_{2d}(T)$. Then
    \begin{align}
        \mathbb{P}\left( \Vert H \Vert_{\mathrm{HS}}  \ge T\sqrt{2d(m-d)} \right) \le \exp{\left( -2^{-2}d(m-d) \right)}.
    \end{align}
\end{mylem}

\begin{proof}[\textsf{Proof of Proposition \ref{propo4.1}}]
    We begin by selecting the parameters. Let
    \begin{align}
      c_{\ref{propo4.1}}=\min{(2^{-36},c_{\ref{coro4.1}})},~~d_{\ref{propo4.1}}=d_{\ref{coro4.1}},~~c_{\ref{propo4.1}}'=C_{\ref{propo4.1}}^{-1},~~C_{\ref{propo4.1}}=2^{6}C_{\ref{propo4.1}}.\nonumber
    \end{align}
    Set $Y := \frac{c_{0}X}{32\sqrt{n}}$. By Lemma \ref{Lem4.3}, we obtain that  the upper bound of the probability is at most
    \begin{align}
        \mathbb{P}\left( \exists U \in \mathcal{U}_{2d,k} : \Vert HU_{Y} \Vert_{\mathrm{HS}} \le c_{0}\sqrt{m(k+1)}/8 \right).\nonumber
    \end{align}
    Set $\delta := T^{-1}c_{0}/8$ and $\mathcal{E}:=\left\{H: \Vert H \Vert_{\mathrm{HS}} \le T\sqrt{2d(m-d)} \right\}$. If $\mathcal{E}$ occurs, there exists $U\in \mathcal{U}_{2d,k}$, such that $\Vert HU_{Y}\Vert_{\mathrm{HS}} \le c_{0}\sqrt{m(k+1)}/8$, we have found $W \in \mathcal{W}$ by Lemma \ref{Lem4.2}, which satisfies
    \begin{align}
        \Vert HW_{Y}\Vert_{\mathrm{HS}} \le \Vert H(W_{Y}-U_{Y})\Vert_{\mathrm{HS}} + \Vert HU_{Y} \Vert_{\mathrm{HS}} \le c_{0}\sqrt{m(k+1)}/4.\nonumber
    \end{align}
    This shows that
   \begin{eqnarray*}
        & & \mathbb{P}_{H}\left( \exists U \in \mathcal{U}_{2d,k}: \Vert HU_{Y}\Vert_{\mathrm{HS}} \le\frac{c_{0}}{8}\sqrt{m(k+1)},\mathcal{E} \right)\\ 
        &\le& \mathbb{P}_{H}\left( \exists W \in \mathcal{W}:\Vert HW_{Y}\Vert_{\mathrm{HS}} \le\frac{c_{0}}{4}\sqrt{m(k+1)},\mathcal{E} \right).\nonumber
   \end{eqnarray*}
    We note that the right-hand side of the above is at most
    \begin{align}
        \sum_{W \in \mathcal{W}}{\mathbb{P}\left( \Vert HW_{Y}\Vert_{2} \le c_{0}\sqrt{m(k+1)}/4 \right)}.\nonumber
    \end{align}
    Consider the size of $\mathcal{W}$,
    \begin{align}
        |\mathcal{W}| \le(2^{6}/\delta)^{2dk} \le(2^{9}T/c_{0})^{2dk} \le c_{0}^{-\frac{5}{2}dk} \le\exp{(c_{0}k(m-d)/8)},\nonumber
    \end{align}
    where the inequality holds since $c_{0} \le c_{\ref{propo4.1}}T^{-4}\nu \le2^{-36}T^{-4}$ and $d \le 2c_{0}^{2}m$. Thus
    \begin{eqnarray*}
        & &\sum_{W \in \mathcal{W}}{\mathbb{P}_{H}\left( \Vert HW_{Y}\Vert_{\mathrm{HS}} \le c_{0}\sqrt{m(k+1)}/4\right)}\nonumber \\
        &\le& \exp{(c_{0}k(m-d)/8)}\max_{W\in \mathcal{W}}{\mathbb{P}_{H}\left( \Vert HW_{Y}\Vert_{\mathrm{HS}} \le c_{0}\sqrt{m(k+1)}/4 \right)}.\nonumber
   \end{eqnarray*}
    For each $W \in \mathcal{W}$, apply Lemma \ref{Lem4.1} with $\beta := \sqrt{c_{0}/3}$ to obtain
    \begin{align}
        \mathbb{P}_{H}\left( \Vert HW_{Y}\Vert_{2} \le c_{0}\sqrt{m(k+1)}/4 \right) \le \left( 2^{5}e^{2\beta^{2}k}\right)^{m-d}\prod_{j=1}^{m-d}{\mathcal{L}\left( W^{\mathrm{T}}\tau_{j},c_{0}^{1/2}\sqrt{k+1}\right)}.\nonumber
    \end{align}
    Note that $\Vert Y \Vert_{2}=c_{0}\Vert X\Vert_{2}/32\sqrt{n}  \ge t^{-1}$ and for each $W \in \mathcal{W}$, we have $\Vert W\Vert \le c_{0}/T+1 \le2$, $\Vert W \Vert_{\mathrm{HS}}  \ge \sqrt{k}-c_{0}\sqrt{k}/8T > \sqrt{k}/2$.\\
    Thus, we apply Corollary \ref{coro4.1} to show that
    \begin{align}
        \mathcal{L}(W_{Y}^{\mathrm{T}}\tau_{j},c_{0}^{1/2}\sqrt{k+1}) \le(R_{\ref{coro4.1}}t)^{2}e^{-c_{0}k}.\nonumber
    \end{align}
    Furthermore, we have
    \begin{align}
        \mathbb{P}_{H}\left( \Vert HW_{Y}\Vert_{2} \le c_{0}\sqrt{m(k+1)}/4 \right)  
        & \le2^{5(m-d)}(Rt)^{2m-2d}e^{-c_{0}k(m-d)/3} \nonumber\\ 
        & \le(2^{5}R_{\ref{coro4.1}}t)^{2m-2d}\cdot e^{-c_{0}k(m-d)/3}.\nonumber
    \end{align}
    Combining with the previous bounds, and noting 
    \begin{align}
        \exp{(-2^{-2}d(m-d))} \le(R_{\ref{coro4.1}}t)^{2m-2d}\cdot e^{-c_{0}k(m-d)/4}.\nonumber
    \end{align}
    We have
    \begin{align}
        \mathbb{P}\left( \sigma_{2d-k+1}(H) \le\frac{c_{0}\sqrt{m}}{16} \ and \ \Vert H_{1}X\Vert_{2},\Vert H_{2}X\Vert_{2} \le m \right) \le(R_{\ref{propo4.1}}t)^{2m-2d}e^{-c_{0}km/4}.\nonumber
    \end{align}
    Now, we complete the proof of this proposition.
\end{proof}
\section{Nets for structured vectors}\label{Section Nets for unstructured vectors}

\subsection{Random vectors with a large log-RLCD}\label{Section Random vectors with a large log-RLCD}
The goal of this subsection will be to ensure that the probability of the random vectors $X$ in Proposition \ref{Propo_4.1} with small log-RLCD is super-exponentially small, so that we can use the main result in Section 4 to prove Proposition \ref{Propo_4.1}. Before starting, we give a property of $\overline{\xi}$ for $\xi \in S_{2}(T)$.
\begin{mylem}\label{p} 
     Let $\xi \in S_{2}(T)$, and $\xi'$ is the independent copy of $\xi$, we have set $\overline{\xi}:= \xi -\xi'$ above, then
    \begin{align}
        \mathbb{P}\left( |\overline{\xi}|  \ge1 \right)  \ge p,
    \end{align}
    where $p:=\left( 2T \right)^{-4}$.
\end{mylem}
\begin{proof}
    Define $X:=\overline{\xi}^{2}$, we get $\mathbb{E}X = 2$.
    Applying the Paley-Zygmund inequality for $X$,
    \begin{align}
        \mathbb{P}\left( X  \ge1 \right)  \ge\left( 1-\frac{1}{2} \right)^{2}\frac{(\mathbb{E}X)^{2}}{\mathbb{E}X^{2}}  \ge (\mathbb{E}X^{2})^{-1}.\nonumber
    \end{align}
    Note that $\mathbb{E}X^{2}=6+2\mathbb{E}\xi^{4}$ and 
    \begin{align}
        1+\frac{\mathbb{E}\xi^{4}}{2T^{4}} \le\mathbb{E}e^{\frac{\xi^{2}}{T^{2}}} \le2.\nonumber
    \end{align}
    Combining the two inequalities mentioned above, we can derive
    \begin{align}
        \mathbb{P}\left( |\xi|  \ge 1 \right)  \ge(6+4T^{4})^{-1}  \ge p.\nonumber
    \end{align}
    This completes the proof of this lemma.
\end{proof}
In the following discussion, we will fix $p$ here. 
We now give the main result in this subsection.
\begin{mylem}\label{Lem_4.4}
    For $q\in (0,1)$, $\gamma  \ge p^{-1/2}$, $K  \ge 1$ and $\kappa  \ge2$, let $d  \ge2^{7}T^{4}$, $n  \ge\kappa^{2}K^{2}/q^{2}$, $d \in [c_{0}^{2}n/4,c_{0}^{2}n]$ and $ KN < \exp{\left( \frac{q\sqrt{p}d}{\gamma \sqrt{2}} \right)}$. Set $\alpha= (4\kappa)^{-1}p$. Let $\mathcal{B} = \left( [-\kappa N,-N]\cup [N,\kappa N]  \right)^{d}$, X be chosen uniformly at random from $\mathcal{B}$, and $\xi \in G_{d}(T)$.\\
    Then 
    \begin{align}
        \mathbb{P}_{X}\left( \mathrm{RlogD}_{\gamma,\alpha}^{\xi}(r_{n}\cdot X) \le K \right) \le\left(C_{\ref{Lem_4.4}}T^{6}\sqrt{\alpha} \right)^{pd/8},\nonumber
    \end{align}
    where $r_{n}:= c_{0}2^{-5}/\sqrt{n}$ and $C_{\ref{Lem_4.4}}$ is an absolute constant.
\end{mylem}
In fact, this lemma can be regarded as a corollary of Theorem 4.1 in \cite{Livshyts_aop}. However, Livshyts et al. discussed heavy-tailed distributions in \cite{Livshyts_aop}, while here we focus on subgaussian random variables, which have more properties and assumptions, therefore, to obtain a more precise estimate, we provide a detailed proof. Let us first state the following lemma.
\begin{mylem}\label{logrange}
    For $u>2$, $t \in (0,1)$, let $W \in \mathbb{R}^{d}$ satisfy $$\Vert W\Vert_{2}^{2} \le u^{2}td/2  \  and \ \left|  \left\{ i:|w_{i}|  \ge 1  \right\}  \right|  \ge td.   $$
    For $\kappa  \ge 2$, $\gamma_{0} >0$, $q \in (0,1)$, $\alpha = 2^{-1}\sqrt{t}/\kappa$, $d  \ge 4/t$, $n \ge\kappa^{2}K^{2}/q^{2} $ and $d \in [c_{0}^{2}n/4,c_{0}^{2}n]$. Let $X$ be chosen uniformly at random from $\mathcal{B}$ with $KN \le e^{q d/\gamma_{0}}$.\\
    Then 
    \begin{align}
        \mathbb{P}_{X}\left( \exists \phi \in (0,Kr_{n}): \Vert \phi W\star X\Vert_{\mathbb{T}} \le \gamma_{0}\sqrt{\log_{+}{\frac{\alpha \phi \Vert X\Vert_{2}}{\gamma_{0}}}} \right) \le N(2^{10}u^{2}\sqrt{q/t})^{td/4},\nonumber
    \end{align}
    where $r_{n}:= c_{0}2^{-5}/\sqrt{n}$
\end{mylem}
\begin{proof}
    Define $J := \left\{ i : u  \ge|w_{i}|  \ge1 \right\}$, we have 
    \begin{align}
        u^{2}(td-|J|) \le \Vert W\Vert_{2}^{2} \le u^{2}td/2.\nonumber
    \end{align}
    It implies that $|J|  \ge td/2$, without loss of generality, let us assume $|J|=td/2$.

        If $\phi \le(2u\kappa N)^{-1}$, we have 
    \begin{align}
        |\phi w_{j}x_{j}| \le\phi u \kappa N \le1/2 \ for \ all \ j \in J.\nonumber
    \end{align}
    This means 
    \begin{align}
        \Vert \phi W \star X\Vert_{\mathbb{T}}^{2} = \sum_{j\in J}{\Vert \phi w_{j}x_{j}\Vert_{\mathbb{T}}^{2}} = \sum_{j \in J}{\Vert \phi w_{j} x_{j}}\Vert_{2}^{2}  \ge t\phi^{2}N^{2}d/2.\nonumber 
    \end{align}
    Note that $\sqrt{\log_{+}{s}} \le s$, we have 
    \begin{align}
        \gamma_{0}^{2}\log_{+}{\frac{\alpha \phi \Vert X\Vert_{2}}{\gamma_{0}}} \le\alpha^{2} \phi^{2} \Vert X\Vert_{2}^{2}\le\alpha^{2}\kappa^{2}\phi^{2}N^{2}d \le t\phi^{2}N^{2}d/2 \le \Vert \phi W\star X\Vert_{\mathbb{T}}^{2}.\nonumber 
    \end{align}
    Thus, in order to prove this lemma, it is enough to bound the 
     \begin{eqnarray*}
        & &\mathbb{P}_{X}\left( \exists \phi \in \left( (2u\kappa N)^{-1},Kr_{n} \right),\Vert \phi W \star X\Vert_{\mathbb{T}} \le\gamma_{0}\sqrt{\log_{+}{\frac{\alpha \phi \Vert X\Vert_{2}}{\gamma_{0}}}} \right)\nonumber\\
        & \le&\mathbb{P}_{X}\left( \exists \phi \in \left( (2u\kappa N)^{-1},Kr_{n}\right),\Vert \phi W \star X\Vert_{\mathbb{T}} \le\sqrt{qd}  \right),\nonumber
    \end{eqnarray*}
    where using $KN \le e^{q d/\gamma_{0}}$ and $\alpha \kappa r_{n}\sqrt{d}/\gamma_{0} \le1$.\\
    Define $\phi_{i} := \frac{iq}{2u\kappa N}$ and $I := \left[ \frac{1}{2q},uN \right] \cap \mathbb{Z}$. Observe that 
    $$|\phi| < Kr_{n} \le uN\cdot \frac{q}{2u\kappa N}$$
    where last inequality since $n  \ge \kappa^{2}K^{2}/q^{2}$.
    Thus, for all $\phi \in \left( (2u\kappa N)^{-1},Kr_{n} \right)$, there exists $\phi_{i}$ with $i \in I$ such that
    \begin{align}
        |\phi - \phi_{i}| \le\frac{q}{u \kappa N}.\nonumber
    \end{align}
    Furthermore, for all $\phi\in ((2u\kappa N)^{-1},Kr_{n})$, there exists $\phi_{i}$ such that
    \begin{align}
        \Vert \phi_{i}W_{J}\star X_{J}\Vert_{\mathbb{T}} <\Vert \phi W\star X\Vert_{\mathbb{T}} +|\phi_{i}-\phi|\Vert W_{J}\star X_{J}\Vert_{2} <2\sqrt{q d},\nonumber
    \end{align}
    where $X_{J}$ and $W_{J}$ represent vectors composed of all coordinates belonging to $J$ of $X$ and $W$, respectively. Thus, the probability can be bounded by
     \begin{eqnarray*}
        & &\sum_{i \in I}{\mathbb{P}_{X}\left( \Vert \phi_{i} W\star X\Vert_{\mathbb{T}} \le2\sqrt{q d} \right)}\nonumber\\
        &\le& \sum_{i \in I}{\mathbb{P}_{X}\left( \Vert \phi_{i} W_{J}\star X_{J}\Vert_{\mathbb{T}} \le\sqrt{q_{0} |J|} \right)}\nonumber \\
        &\le& 2^{|J|}\cdot \sum_{i \in I}{\left( \prod_{j \in J'} \mathbb{P}_{x_{j}}\left( \Vert \phi_{i}w_{j}x_{j}\Vert_{\mathbb{T}}^{2} \le4q_{0} \right) \right)},\nonumber
    \end{eqnarray*}
    where $q_{0}:=8q/t$ and $J' \subseteq J$ with $|J'|=|J|/2$. Let $p_{ij} \in \mathbb{Z}$ be an integer closest to $\phi_{i} w_{j} x_{j}$, we obtain
    $$|p_{ij}| \le\phi_{i}|w_{j}x_{j}|+1 \le\phi_{i}u\kappa N +1 := S_{i}.$$
    Thus, we have the following inequality:
    \begin{equation}
    \begin{aligned}
        \mathbb{P}_{x_{j}}\left( \Vert \phi_{i}w_{j}x_{j}\Vert_{\mathbb{T}} \le2\sqrt{q_{0}}  \right) 
        & \le\sum_{p_{ij}=-S_{i}}^{S_{i}}{\mathbb{P}\left( |x_{j}-p_{ij}\phi_{i}^{-1}w_{j}^{-1}| \le2\sqrt{q_{0}}\phi_{i}^{-1}w_{j}^{-1} \right) }\\ 
        & \le\sum_{p_{ij}=-S_{i}}^{S_{i}}{\mathbb{P}\left( |x_{j}-p_{ij}\phi_{i}^{-1}w_{j}^{-1}| \le2\sqrt{q_{0}}\phi_{i}^{-1}\right) }\\
        & \le\frac{\left( 2S_{i}+1 \right)\left( 4\sqrt{q_{0}}\phi_{i}^{-1}+1 \right)}{2(\kappa-1)N}\\
        & \le2^{6}u\sqrt{q_{0}},\nonumber
        \end{aligned}
    \end{equation}
    where the last inequality is obtained by using $4\sqrt{q_{0}}\phi_{i}^{-1}  \ge1$, $\kappa  \ge2$ and $u >1$. Combining the aforementioned results, the probability can be bounded by 
    \begin{equation}
        \begin{aligned}
            2^{td/2}\cdot \sum_{i \in I}{\left( 2^{6}u\sqrt{q_{0}} \right)^{td/4}}
            \le N\left( 2^{10}u^{2}\sqrt{\alpha/t} \right)^{td/4}.\nonumber 
        \end{aligned}
    \end{equation}
    This completes the proof of this lemma.
\end{proof}
\begin{proof}[\textsf{Proof of Lemma \ref{Lem_4.4}}]
We start by defining parameters. Define $t=p/2 \text{ and } u=16T^{3}$.\par  
For $\xi \in G_{d}(T)$, we obtain $$\mathbb{P}\left( |\overline{\xi}_{i}|  \ge1 \right)  \ge p:=2^{-4}T^{-4}$$ by Lemma \ref{p}. Thus, apply Markov's inequality,  
\begin{align}\label{e5.1}
    \left| \left\{ i:|\overline{\xi}_{i}|  \ge 1 \right\} \right|  \ge td,
\end{align}
with probability at least $p/2+p^{2}/4$.\par
Note that $\xi$ is subgaussian vector, we get
\begin{align}\label{e5.2}
    \mathbb{P}_{\xi}\left( \sum_{i=1}^{d}{|\overline{\xi}_{i}|^{2}}  \ge u^{2}td/2 \right) \le e^{-u^{2}td/(2T^{2})}\mathbb{E}\left[ \exp{\left( \sum_{i=1}^{d}{|\overline{\xi}_{i}|^{2}}/(2T^{2}) \right)} \right]\le e^{\left( 2\log{2}-2 \right)d}.\nonumber
\end{align}
on the other hand, Observe that 
\begin{align}
    \sum_{i=1}^{d}{|\overline{\xi}_{i}|^{2}} \le u^{2}td/2,
\end{align}
with probability at least $1-e^{-d/2}$.

Let $\mathcal{H}$ be the event that both \eqref{e5.1} and \eqref{e5.2} occur simultaneously, we get 
\begin{align}
    \mathbb{P}(\mathcal{H})  \ge p/2+p^{2}/4-e^{-d/2}  \ge p/2=t,\nonumber
\end{align}
where last inequality occurs since $d > 2^{7}T^{4} = 4/t  \ge 2\log{(4/p^{2})}$.

We now apply Lemma \ref{logrange} with $\alpha/\sqrt{t}= \alpha_{0}=2^{-1}\sqrt{t}/\kappa$ and $\gamma_{0}=\gamma/\sqrt{t}$ to obtain
 \begin{eqnarray*}
        & &\mathbb{P}_{X}\left( \mathrm{RLCD}_{\gamma,\alpha}^{\xi}\left( r_{n}\cdot X \right) > K \right) \\
        &\ge&\mathbb{P}_{X}\left( \forall \phi \in (0,Kr_{n}):\Vert \phi X\Vert_{\xi} > \gamma\sqrt{ \log_{+}{\frac{\alpha \phi \Vert X\Vert_{2}}{\gamma}}
        } \right)\\
        &\ge&\mathbb{P}_{X}\left( \forall \phi \in (0,Kr_{n}) : \Vert \phi \overline{\xi}_{i} \star X\Vert_{\mathbb{T}}^{2} > \frac{\gamma^{2}}{t} \log_{+}{\frac{\alpha \phi \Vert X\Vert_{2}/\sqrt{t}}{\gamma/\sqrt{t}}}
        , \mathcal{H}\right)\\
        &\ge&1-N\left( 2^{10}u^{2}\sqrt{q}/t \right)^{td/2}\\
        &\ge&1-\left( 2^{20}T^{6}\sqrt{q} \right)^{pd/8},\nonumber
   \end{eqnarray*}
where $N \le e^{pd}$ by $\gamma  \ge p^{-1/2}$. Completes the proof of the lemma \ref{Lem_4.4}.
\end{proof}
\subsection{Proof of the Proposition \ref{Propo_4.1}}\label{Proof of the Propo 4.1}
This subsection will prove Proposition \ref{Propo_4.1}. Before starting, we also need to study the ``nonstructure" parts of the vectors in the ``discretized grid". It is a version of Rudelson and Vershynin \cite{Rudelson_imrn} and is introduced from \cite{Campos_jams}.
\begin{mylem}\label{projection}
    For $N \in \mathbb{N}$, $m ,d,k \in \mathbb{N}$ be so that $m-d  \ge 2d  \ge 2k$, $H$ be a $2d\times(m-d)$ matrix with $\sigma_{2d-k}(H)  \ge c_{0}\sqrt{m}/16$ and $B_{1},\dots,B_{m-d} \subseteq \mathbb{Z}$ with $|B_{i}|  \ge N$. If $X$ is chosen uniformly at random from $\mathcal{B}=B_{1} \times \dots \times B_{m-d}$, then
    \begin{align}
        \mathbb{P}_{X}\left( \Vert HX\Vert_{2} \le4m  \right) \le\left( \frac{C_{\ref{projection}}m}{dc_{0}N} \right)^{2d-k},\nonumber    
    \end{align}
    where $C_{\ref{projection}}>0$ is an absolute constant.
\end{mylem}
We have completed all preparations and now we begin to prove Proposition \ref{Propo_4.1}.
\begin{proof}[\textsf{Proof of Proposition \ref{Propo_4.1}}]
    Recall $M$ is defined as 
    \begin{align}
            M := \begin{bmatrix}
        \mathbf{0}_{[d]\times[d]} & H_{1}^{\mathrm{T}}                  & \mathbf{0}_{[d]\times[k]}\\
        H_{1}                     & \mathbf{0}_{[d+1,m] \times [d+1,m]} & \mathbf{0}_{[d+1,m] \times [k]}
    \end{bmatrix},
    \end{align}
    where $H_{1}$ is a $(m-d)\times d$ random matrix with independent rows satisfying $\mathrm{row}_{j}(H_{1}) \in \Phi_{\nu}(d,\xi_{j})$,  $\xi_{j} \in G_{d}(T)$, $\nu :=2^{-14}$. Let $H_{2}$ be an independent copy of $H_{1}$ and define 
    \begin{align}
        H:=\left[ H_{1},H_{2} \right].\nonumber
    \end{align}
    For a vector $X \in \mathbb{R}^{n}$, we define the events $\mathcal{A}_{1}:= \mathcal{A}_{1}(X)$ and $\mathcal{A}_{2}:=\mathcal{A}_{2}(X)$ by 
    \begin{align}
        \mathcal{A}_{1}:= \left\{ H:\Vert H_{1}X_{[d]}\Vert_{2} \le2m \text{ and } \Vert H_{2}X_{[d]}\Vert_{2} \le2m \right\},\nonumber
    \end{align}
    \begin{align}
        \mathcal{A}_{2}:=\left\{ H : \Vert H^{T}X_{[d+1,m]}\Vert_{2} \le4m  \right\}.\nonumber
    \end{align}
    We derive the bound on $\mathbb{P}_{M}\left( \Vert MX\Vert_{2}^{2} \le2m \right)$ in terms of $\mathcal{A}_{1}$ and $\mathcal{A}_{2}$. Let $M'$ be an independent copy of $M$, Obviously,
   \begin{eqnarray*}
 \mathbb{P}_{M}\left( \Vert MX\Vert_{2} \le2m \right)^{2}&=&\mathbb{E}_{M}\mathbb{E}_{M'}\mathbf{1}_{  \left\{ \Vert MX\Vert_{2} , \Vert M'X\Vert_{2} \le2m \right\}  }\\
 &\le & \mathbb{E}_{M}\mathbb{E}_{M'}\mathbf{1}_{ \left\{ \Vert H_{1}X_{[d]}\Vert_{2} \le2m,\Vert H_{2}X_{[d]} \Vert_{2} \le2m \text{ and } \Vert H^{T}X_{[d+1,m]}\Vert_{2} \le4m \right\}   } \\
 &=& \mathbb{P}_{H}\left( \mathcal{A}_{1} \cap \mathcal{A}_{2}\right).
\end{eqnarray*}

Thus, 
    \begin{align}
        \mathbb{P}_{M}\left(\Vert MX\Vert_{2} \le2m  \right)^{2} \le\mathbb{P}_{H}\left( \mathcal{A}_{1} \cap \mathcal{A}_{2} \right).\nonumber
    \end{align}
    For $k=0,\dots,2d$, denoted the ``rank splitting" events by
    \begin{align}
        \mathcal{E}_{k}:=\left\{ H:\sigma_{2d-k}(H)  \ge c_{0}\sqrt{m}/16 \text{ and }   \sigma_{2d-k+1}(H) \le c_{0}\sqrt{m}/16\right\}.\nonumber
    \end{align}
    We now define the parameters and event, let
    \begin{align}
        q:=T^{-12} L^{-2^{9}T^{4}m/d}, \quad \alpha = p/(4\kappa), \quad Q_{\ref{Propo_4.1}}=2^{-1}T^{-16}L^{-2^{-11}T^{4}/c_{0}^{2}},\nonumber
    \end{align}
    and
    \begin{align}
        \mathcal{T}:=\mathcal{T}(\mathcal{B})= \left\{ X \in \mathcal{B}:\mathrm{RlogD}_{\gamma,\alpha}^{H_{1}}(r_{n}\cdot X_{[d]}) > 8 \right\}.\nonumber
    \end{align}
    Recall $d  \ge 2^{7}T^{4}$, let $n_{\ref{Propo_4.1}}:=2^{6}\kappa^{2}T^{24}\cdot L^{2^{12}T^{4}/c_{0}^{2}}$, we have $$n  \ge n_{\ref{Propo_4.1}}(\kappa,T,c_{0},T) > 8\kappa^{2}/q^{2}$$ and $$8N < \exp{\left( Q_{\ref{Propo_4.1}}d \right)} \le \exp{\left( \frac{q\sqrt{p}d}{\gamma \sqrt{2}} \right)}.$$
    We apply Lemma \ref{Lem_4.4},
    \begin{align}
        \mathbb{P}_{X}\left( X \notin \mathcal{T} \right) = \mathbb{P}_{X}\left( \mathrm{RlogD}_{\gamma,\alpha}^{H_{1}}(r_{n}\cdot X) \le8 \right) \le(m-d)\left( C_{\ref{Lem_4.4}}T^{6}\sqrt{q} \right)^{pd/8} \le(C_{\ref{Lem_4.4}}/L)^{2m}.\nonumber
    \end{align}
    Define
    \begin{align}
        \mathcal{E}:=\left\{ X \in \mathcal{B}: \mathbb{P}_{M}\left( \Vert MX\Vert_{2} \le2m \right)  \ge(L/N)^{m}  \right\},\nonumber
    \end{align}
    and 
    \begin{align}
        f(X):= \mathbb{P}_{M}\left( \Vert MX\Vert_{2} \le2m \right) \mathbf{1}_{\left\{ X \in \mathcal{T} \right\} }.\nonumber
    \end{align}
    We have
    \begin{align}
        \mathbb{P}_{X}(\mathcal{E}) \le\mathbb{P}_{X}\left( \mathcal{E} \cap \left\{ X \in \mathcal{T} \right\} \right) + \mathbb{P}_{X}\left( X \notin \mathcal{T} \right) \le\mathbb{P}_{X}\left( f(X)  \ge(L/N)^{m} \right) + (C_{\ref{Lem_4.4}}/L)^{2m}.\nonumber
    \end{align}
    By Markov's inequality, 
    \begin{align}
        \mathbb{P}_{X}(\mathcal{E}) \le(N/L)^{2m}\mathbb{E}f(X)^{2}+(C_{\ref{Lem_4.4}}/L)^{2m}.\nonumber
    \end{align}
    Thus, in order to prove Proposition \ref{Propo_4.1}, it is enough to prove 
    \begin{align}
        \mathbb{E}_{X}f(X)^{2} \le2(R/N)^{2m}.\nonumber
    \end{align}
    Consider that 
    \begin{align}
        f(X)^{2} \le\mathbb{P}\left( \mathcal{A}_{1} \cap \mathcal{A}_{2}  \right) \mathbf{1}_{ \left\{X \in \mathcal{T} \right\} } = \sum_{k=0}^{2d}{\mathbb{P}_{H}\left( \mathcal{A}_{2}|\mathcal{A}_{1}\cap \mathcal{E}_{k} \right) \mathbb{P}_{H}\left( \mathcal{A}_{1} \cap \mathcal{E}_{k} \right)\mathbf{1}_{\left\{ X \in \mathcal{T} \right\}}}.\nonumber
    \end{align}
    
    Let $Y:=X_{[d]}/2$, we have $\mathrm{RlogD}_{\gamma,\alpha}^{H_{1}}(r_{n}\cdot Y) >16$ by $X \in \mathcal{T}$. In the following argument, we will apply the Proposition \ref{propo4.1} with $X =Y$ and $\nu = 2^{-14}$. 

    In fact, for $0 \le k \le d$, $\Vert Y \Vert_{2}  \ge \sqrt{d}N/2=32\sqrt{n}c_{0}^{-1}/t$, $\alpha=c\kappa^{-1}T^{-4}$ and $\gamma=\max{\left( p^{-1/2}, 2^{7}d_{\ref{propo4.1}}T^{2} \right) }=C_{1}T^{2}$, note that 
    \begin{align}
        t:=64\sqrt{n/d}c_{0}^{-1}N^{-1} > 64c_{0}^{-2}N^{-1}  \ge c_{\ref{propo4.1}}'\alpha \gamma^{-1}e^{-2^{-3} d},\nonumber
    \end{align}
    where $N \le8^{-1}\exp{(Q_{\ref{Propo_4.1}}d)} \le\frac{64\gamma}{c_{\ref{propo4.1}}'\alpha c_{0}^{2}}e^{d/8}$(recall $q$ is small enough if $c_{0}$ is small).\\
    Thus, we have
    \begin{align} 
        \mathbb{P}_{H}\left( \mathcal{A}_{1}\cap \mathcal{E}_{k} \right) 
        & \le \mathbb{P}_{H}\left( \sigma_{2d-k+1}\left( H \right) \le c_{0}\sqrt{m}/16 \text{ and }  \Vert H_{1}Y\Vert_{2},\Vert H_{2}Y\Vert_{2} \le m \right)\nonumber \\
        & \le\left( R_{0}t  \right)^{2m-2d}e^{-c_{0}km/4},\nonumber
    \end{align}
    where $R_{0}:=C_{\ref{propo4.1}} \gamma \alpha^{-1}$, note that 
    \begin{align}
        R_{0}t \le\frac{C_{\ref{propo4.1}}\gamma}{\alpha c_{0}^{2}N}.\nonumber
    \end{align}
    It means that 
    \begin{align}
        \mathbb{P}_{H}\left( \mathcal{A}_{1}\cap \mathcal{E}_{k} \right) \le e^{-c_{0}km/4}\left( R_{1}/N \right)^{2m-2d},\nonumber
    \end{align}
    where $R_{1}:=\frac{C_{\ref{propo4.1}}\gamma}{\alpha c_{0}^{2}}=\frac{C'\kappa T^{6}}{c_{0}^{2}}$.
    Otherwise, for $k  \ge d$,
    \begin{align}
        \sum_{k  \ge d}{\mathbb{P}_{H}\left( \mathcal{A}_{1} \cap \mathcal{E}_{k} \right)}
        & \le\mathbb{P}_{H}\left( \left\{ \sigma_{d}(H) \le c_{0}\sqrt{m}/16  \right\} \cap \mathcal{A}_{1} \right)\nonumber\\ 
        & \le e^{c_{0}dm/4}\left( R_{1}/N \right)^{2m-2d}.\nonumber
    \end{align}
    It implies that 
    \begin{align}
        f(X)^{2} \le \sum_{k=0}^{d}\mathbb{P}_{H}\left( \mathcal{A}_{2} |\mathcal{A}_{1} \cap \mathcal{E}_{k} \right) e^{-c_{0}km/4}\left( R_{1}/N\right)^{2m-2d}+e^{-c_{0}dm/4}\left( R_{1}/N \right)^{2m-2d}.\nonumber
    \end{align}
    Define $g_{k}(X):=\mathbb{P}_{H}\left( \mathcal{A}_{2}|\mathcal{A}_{1}\cap \mathcal{E}_{k} \right)$.
    \begin{align}
        \mathbb{E}_{X}g_{k}(X)=\mathbb{E}_{X}\mathbb{E}_{H}\left[ \mathcal{A}_{2}|\mathcal{A}_{1}\cap \mathcal{E}_{k} \right] = \mathbb{E}_{X_{[d]}}\mathbb{E}_{H}\left[  \mathbb{E}_{X_{[d+1,n]}}\mathbf{1}_{ \mathcal{A}_{2} } | \mathcal{A}_{1}\cap \mathcal{E}_{k} \right].\nonumber
    \end{align}
    For $k \le d$ and each $H \in \mathcal{A}_{1}\cap \mathcal{E}_{k}$ has $\sigma_{2d-k}(H)  \ge c_{0}\sqrt{m}/16$, and thus we may apply Lemma \ref{projection},
    \begin{align}
        \mathbb{E}_{X_{[d+1,n]}}\left[  \mathbf{1}_{\mathcal{A}_{2}} \right] = \mathbb{P}_{X_{[d+1,m]}}\left( \Vert H^{\mathrm{T}}X_{[d+1,m]}\Vert_{2} \le4m  \right) \le\left( \frac{C_{\ref{projection}}m}{dc_{0}N} \right)^{2d-k} \le\left( \frac{C}{c_{0}^{3}N} \right)^{2d-k}.\nonumber
    \end{align}
    Let $R:=\max{\left\{ 2C/c_{0}^{3},2R_{1} \right\} }=Cc_{0}^{-3}\kappa T^{6}  \ge2$, we have
    \begin{align}
        \mathbb{E}_{X}[g_{k}(X)] \le\left( R/(2N) \right)^{2d-k}.
    \end{align}
    We obtain
    \begin{align}
        \mathbb{E}_{X}f(X)^{2} \le\left( \frac{R}{2N} \right)^{2m} \sum_{k=0}^{d}{\left( \frac{2N}{R} \right)^{k}e^{-c_{0}mk/4}}+\left( \frac{R}{2N} \right)^{2m-2d}e^{c_{0}dm/4}.\nonumber
    \end{align}
    Observe that $N \le 8^{-1}\exp{(Q_{\ref{Propo_4.1}}d)} \le e^{c_{0}m/4}$, thus 
    \begin{align}
        \mathbb{E}_{X}f(X)^{2} \le 2\left( \frac{R}{2N} \right)^{2m}.\nonumber
    \end{align}
    This can immediately imply
    \begin{align}
        \mathbb{P}_{X}\left( \mathcal{E} \right) \le2\left( R/2L \right)^{2m} + \left( C_{\ref{Lem_4.4}}/L \right)^{2m} < \left( R/L\right)^{2m},\nonumber
    \end{align}
    where $R:=Cc_{0}^{-3}\kappa T^{6}$ and $C >0$ is an absolute constant. Finally, we complete the proof by Proposition \ref{propo4.1}.
\end{proof}

\end{document}